\documentclass[letterpaper,11pt]{amsart}
\usepackage{hyperref}
\usepackage{amsmath,amssymb,amsthm,amsopn}
\usepackage{fullpage}
\usepackage{verbatim}
\usepackage{graphicx}

\DeclareMathOperator{\Clo}{Clo}
\DeclareMathOperator{\Inv}{Inv}
\DeclareMathOperator{\lcm}{lcm}
\DeclareMathOperator{\CSP}{CSP}
\DeclareMathOperator{\Sg}{Sg}
\DeclareMathOperator{\Aut}{Aut}

\begin{document}

\makeatletter
\newtheorem*{rep@theorem}{\rep@title}
\newcommand{\newreptheorem}[2]{%
\newenvironment{rep#1}[1]{%
 \def\rep@title{#2 \ref{##1}}%
 \begin{rep@theorem}}%
 {\end{rep@theorem}}}
\makeatother

\newtheorem{thm}{Theorem}
\newreptheorem{thm}{Theorem}
\newtheorem{prop}{Proposition}
\newreptheorem{prop}{Proposition}
\newtheorem{cor}{Corollary}
\newtheorem{lem}{Lemma}
\newreptheorem{lem}{Lemma}

\theoremstyle{definition}
\newtheorem{defn}{Definition}
\newtheorem{conj}{Conjecture}
\newreptheorem{conj}{Conjecture}

\theoremstyle{remark}
\newtheorem{rem}{Remark}

\newcommand{\Rho}{\mathrm{P}}
\newcommand{\cS}{\mathcal{S}}
\newcommand{\cM}{\mathcal{M}}
\newcommand{\cF}{\mathcal{F}}
\newcommand{\cG}{\mathcal{G}}
\newcommand{\cP}{\mathcal{P}}
\newcommand{\cV}{\mathcal{V}}
\newcommand{\RR}{\mathbb{R}}
\newcommand{\ZZ}{\mathbb{Z}}
\newcommand{\NN}{\mathbb{N}}
\newcommand{\bA}{\mathbb{A}}
\newcommand{\bB}{\mathbb{B}}
\newcommand{\bC}{\mathbb{C}}
\newcommand{\bD}{\mathbb{D}}
\newcommand{\bE}{\mathbb{E}}
\newcommand{\bF}{\mathbb{F}}
\newcommand{\bI}{\mathbb{I}}
\newcommand{\bS}{\mathbb{S}}
\newcommand{\fA}{\mathbf{A}}
\newcommand{\fB}{\mathbf{B}}
\newcommand{\gk}{\kappa}
\newcommand{\gS}{\Sigma}
\newcommand{\gl}{\lambda}
\newcommand{\gt}{\theta}

\newcommand{\dotcup}{\ensuremath{\mathaccent\cdot\cup}}

\title{Examples, counterexamples, and structure in bounded width algebras}
\author{Zarathustra Brady}
\address{Department of Mathematics \\ Massachusetts Institute of Technology \\ 77 Massachusetts Avenue \\ Building 2, Room 350B\\ Cambridge, MA 02139-7307}
\email{notzeb@mit.edu}

\thanks{This material is based upon work supported by the National Science Foundation Graduate Research Fellowship under Grant No. (DGE-114747) as well as the NSF Mathematical Sciences Postdoctoral Research Fellowship under Grant No. (DMS-1705177).}

\maketitle

\begin{abstract} We study bounded width algebras which are minimal in the sense that every proper reduct does not have bounded width. We show that minimal bounded width algebras can be arranged into a pseudovariety with one basic ternary operation. We classify minimal bounded width algebras which have size at most three, and prove a structure theorem for minimal bounded width algebras which have no majority subalgebra, which form a pseudovariety with a commutative binary operation. As a byproduct of our results, we also classify minimal clones which have a Taylor term.
\end{abstract}

\section{Introduction}

In the past several years, a number of beautiful characterizations of bounded width algebras have been found. We summarize a few of them in the next proposition.

\begin{prop} Let $\bA$ be a finite idempotent algebra. The following are equivalent:
\begin{enumerate}
\item $\bA$ has bounded relational width.
\item $\bA$ has relational width at most $(2,3)$.
\item Every ``cycle consistent'' instance of $\CSP(\bA)$ has a solution. In fact, every ``$pq$ instance'' of $\CSP(\bA)$ has a solution.
\item $\CSP(\bA)$ has a ``robust'' satisfiability algorithm.
\item $\CSP(\bA)$ is solved by the canonical SDP relaxation.
\item $\bA$ generates a congruence meet-semidistributive variety.
\item $\bA$ generates a variety which omits types $\mathbf{1}$ and $\mathbf{2}$.
\item $\CSP(\bA)$ does not have ``the ability to count''.
\item No nontrivial quotient of any subalgebra of $\bA$ is affine.
\item $\bA$ has weak near-unanimity terms $g(x,y,z), h(x,y,z,w)$ satisfying
\[
g(x,x,y) \approx g(x,y,x) \approx g(y,x,x) \approx h(x,x,x,y) \approx h(x,x,y,x) \approx h(x,y,x,x) \approx h(y,x,x,x).
\]
\item For every sufficiently large $k$, $\bA$ has a $k$-ary weak near-unanimity term.
\item The commutator trivializes, that is, $[ \alpha,\beta ] = \alpha \wedge \beta$ for any congruences $\alpha,\beta$ on any algebra in the variety generated by $\bA$.
\end{enumerate}
\end{prop}
\begin{proof} For the equivalence of (1), and (2), see \cite{barto} (or \cite{bulatov-bounded}, for a slightly weaker result). For the equivalence of (1) and (3), see \cite{slac}, although cycle consistency is referred to as ``singleton linear arc consistency'' there. For the equivalence of (1), (4), and (5), see \cite{sdp}. For the equivalence of (6) and (7), see Theorem 9.10 of \cite{hobby-mckenzie}. For the equivalence of (7), (8), and (9), see \cite{ability-to-count}. For the equivalence of (1) and (7), see \cite{local-consistency}. For the equivalence of (7) and (10), see \cite{maltsev}. For the equivalence of (7) and (11), see \cite{weak-near-unanimity}. For the equivalence of (6) and (12) as well as several other equivalent congruence conditions, see Theorem 8.1 of \cite{kearnes-kiss}.
\end{proof}

It is curious that there are only a few explicit examples of bounded width algebras in the literature. The two basic examples of bounded width algebras are the \emph{majority} algebras and \emph{semilattice} algebras. A majority algebra is an algebra having a ternary operation $g$, called a majority operation, which satisfies the identity
\[
g(x,x,y) \approx g(x,y,x) \approx g(y,x,x) \approx x.
\]
The simplest example of a majority algebra is the \emph{dual discriminator} algebra, with the basic operation $d$ which is given by
\[
d(x,y,z) = \begin{cases}x & \text{ if } y \ne z,\\ y & \text{ if } y = z.\end{cases}
\]
Another fundamental example of a majority algebra is the \emph{median} algebra on an ordered set, whose basic operation returns the median of its three inputs. Generalizing majority algebras, there are the algebras of \emph{bounded strict width}, which (by the Baker-Pixley Theorem \cite{baker-pixley}) are the algebras which have a near-unanimity term of some arity. A \emph{near-unanimity} term $t$ of arity $n$ is defined to be a term which satisfies the identity
\[
t(x,x,...,x,y) \approx t(x,x,...,y,x) \approx \cdots \approx t(x,y,...,x,x) \approx t(y,x,...,x,x) \approx x.
\]
A binary operation $s$ is called a \emph{semilattice} operation if it is associative, commutative, and idempotent. A standard example is the operation $\vee$ of any lattice, and in general every semilattice corresponds to a poset in which every pair of elements have a unique least upper bound. Generalizing semilattices are the so-called \emph{2-semilattice} operations, in which associativity is replaced by the weaker identity
\[
s(x,s(x,y)) \approx s(x,y).
\]
Unlike semilattices, a 2-semilattice need not be associated to a consistent partial ordering.

The goal of this paper is to develop a better understanding of bounded width algebras. In particular we will study several Mal'cev conditions described in \cite{optimal-maltsev}. In that paper, they prove that in every locally finite variety of bounded width, there exists an idempotent term $t$ satisfying the identity
\begin{align*}
t_3(x,x,x,y) &\approx t_3(x,x,y,x) \approx t_3(x,y,x,x) \approx t_3(y,x,x,x)\tag{SM 3}\\
&\approx t_3(x,x,y,y) \approx t_3(x,y,x,y) \approx t_3(x,y,y,x)
\end{align*}
using a difficult Ramsey theoretic construction and the fact that every bounded width algebra has width $(2,3)$. In this paper we'll show that there is a much more direct argument for the existence of such a term (and, in fact, for many more terms) using $pq$ instances \cite{slac}.

Using these terms and a simple iteration argument, we show that an idempotent algebra has bounded width iff it has terms $f,g$ which satisfy the identities
\[
g(x,x,y) \approx g(x,y,x) \approx g(y,x,x) \approx f(x,y)
\]
and
\[
f(f(x,y),f(y,x)) \approx f(x,y).
\]
We conjecture that there are even nicer terms, which satisfy identities having a clear relationship to majority and semilattice algebras.

\begin{repconj}{conj-ws} In every bounded width algebra, there are terms $w,s$ satisfying the identities
\[
w(x,x,y) \approx w(x,y,x) \approx w(y,x,x) \approx s(x,y)
\]
and
\[
s(x,s(x,y)) \approx s(s(x,y),x) \approx s(x,y).
\]
\end{repconj}

We also use our terms to give a construction for an infinite family of special weak near-unanimity terms in Appendix \ref{a-special}, settling Problem 4.2 of \cite{optimal-maltsev}.

\begin{repthm}{special} Let $\bA$ be a finite algebra of bounded width. Then $\bA$ has an idempotent term $t(x,y)$ satisfying the identity
\[
t(x,t(x,y)) \approx t(x,y)
\]
along with an infinite sequence of idempotent weak near-unanimity terms $w_n$ of every arity $n > 2\lcm\{1,2,...,|\bA|-1\}$ such that for every sequence $(a_1, ..., a_n)$ with $\{a_1, ..., a_n\} = \{x,y\}$ and having strictly less than $\frac{n}{2\lcm\{1,2,...,|\bA|-1\}}$ of the $a_i$ equal to $y$ we have
\[
w_n(a_1, ..., a_n) \approx t(x,y).
\]
\end{repthm}

Two other conjectured terms from \cite{optimal-maltsev} are idempotent terms satisfying either the identity
\begin{align*}
t_5(x,y,z,y) &\approx t_5(y,x,z,z) \approx t_5(z,x,x,y)\tag{SM 5}
\end{align*}
or the identity
\begin{align*}
t_6(x,y,z,y) &\approx t_6(y,x,x,z) \approx t_6(z,y,x,x).\tag{SM 6}
\end{align*}
We give a three element bounded width algebra which does not have terms satisfying either identity (SM 5) or (SM 6) in Appendix \ref{a-sm}.


Since an algebra with fewer term operations corresponds to a relational clone containing more constraints, it is natural to study bounded width algebras such that their clone of operations is \emph{minimal} among bounded width clones. The main result of this paper is a (nonconstructive) proof of the existence of an underlying structure for bounded width algebras which are minimal in this sense. This structure result rests crucially on the following coherence theorem, the proof of which is inspired by a clever argument of Bulatov from \cite{colored-graph}.

\begin{repthm}{subalg} Let $\cV$ be a locally finite idempotent bounded width variety such that every algebra in $\cV$ is connected through two element semilattice and two element majority subalgebras, and let $\bA \in \cV$.

Suppose there is a term $m$ and a subset $S \subseteq \bA$ which is closed under $m$, such that $(S,m)$ is a bounded width algebra. Let $\cV'$ be the reduct of $\cV$ consisting of all terms $t$ of $\cV$ such that $S$ is closed under $t$ and such that $t|_S \in \Clo(m|_S)$. Then $\cV'$ also has bounded width.
\end{repthm}

Immediate consequences of this result are that every subalgebra and every quotient of a minimal bounded width algebra is also a minimal bounded width algebra. Using these, we show that the collection of minimal bounded width algebras can be arranged into a pseudovariety $\cV_{mbw}$ (see Theorem \ref{minvar} for details) - this is what is meant by the claim that minimal bounded width algebras have an underlying structure.

We note in passing that there is an analogue of Theorem \ref{subalg} for idempotent varieties having $p$-ary cyclic terms. To see this, suppose that $c$ is a $p$-ary cyclic term for $\cV$, that $\bA \in \cV$ and that there is a $p$-ary term $m$ and a subset $S \subseteq \bA$ which is closed under $m$ and such that $m|_S$ is cyclic. Then if we define a $p$-ary term $t$ by
\[
t(x_1, ..., x_p) = c(m(x_1, x_2, ..., x_p), m(x_2, x_3, ..., x_1), ..., m(x_p, x_1, ..., x_{p-1})),
\]
we see that $t$ is cyclic and that $t|_S = m|_S$ (by idempotence of $c$). By one of the main results of \cite{cyclic}, we can conclude that there is also an analogue of Theorem \ref{subalg} for locally finite idempotent Taylor varieties.

Using this structure result, we classify minimal bounded width algebras of size $3$ (see Theorem \ref{class-3} and Figure \ref{doodles}) as well as minimal bounded width algebras of size $4$ which are generated by two elements (see Theorem \ref{class-4} and Figure \ref{four}).

We also introduce a new generalization of 2-semilattices, which we call \emph{spirals}, in section \ref{spiral}.

\begin{defn} An algebra $\bA = (A,f)$ is a \emph{spiral} if $f$ is a commutative idempotent binary operation and every subalgebra of $\bA$ which is generated by two elements either has size two or has a surjective homomorphism to the free semilattice on two generators.
\end{defn}

We show that every spiral has bounded width, and that every minimal bounded width clone which has no majority subalgebra is term equivalent to a spiral (Theorem \ref{recursive}). As an easy consequence of our results, we can prove a classification theorem for minimal clones which are Taylor.

\begin{thm} Every minimal clone which has a Taylor term is either an affine algebra, a majority algebra, or a spiral.
\end{thm}
\begin{proof} Suppose that $\bA$ is a Taylor algebra with $\Clo(\bA)$ minimal. If $\bA$ does not have bounded width, then there is a nontrivial affine algebra $\bB$ in the variety generated by $\bA$ by the classification of bounded width algebras. Then $\bB$ has a mal'cev term $p$ satisfying the identities $p(x,y,y) \approx p(y,y,x) \approx x$, and if these identities do not hold in $\bA$ then one of the binary operations $p(x,y,y), p(y,y,x)$ generates a strictly smaller clone. Similarly every \emph{absorption identity}, that is, any identity of the form $t \approx x$ where $t$ is a term and $x$ is a variable, which holds in $\bB$ must also hold in $\bA$. But then if any identity $t \approx s$ holds in $\bB$, then $p(x,t,s) \approx x$ holds in $\bB$, so it must also hold in $\bA$. Thus
\[
t \approx p(t,t,s) \approx s
\]
holds in $\bA$ as well, so in fact $\bA$ is in the variety generated by $\bB$, and is therefore an affine algebra.

Now suppose that $\bA$ has bounded width. If there is any majority subalgebra $\bB \le \bA$, then $\bB$ has a majority term $m$ satisfying the identities $m(x,x,y) \approx m(x,y,x) \approx m(y,x,x) \approx x$. Since these are absorption identities, they must also hold in $\bA$, so $\bA$ is a majority algebra. If $\bA$ has no majority subalgebra, then $\bA$ is a spiral by Theorem \ref{recursive}.
\end{proof}

{\bf Organization of the paper.} The first few sections of the paper concern general results about bounded width algebras. In Section \ref{s-partial}, we prove several fundamental results which generalize results on 2-semilattices to what we call ``partial semilattice operations'', which exist in every algebra having at least one semilattice subalgebra. In Section \ref{s-cycle}, we state the definitions of cycle consistency and $pq$ consistency as well as the main result of \cite{slac} which we will use a black box. In Section \ref{s-intersect}, we prove the existence of an infinite family of terms which behave similarly to the monotone self-dual functions on a two element domain, as well as an important connectivity result (Theorem \ref{connect}) which can be viewed as a refinement of Bulatov's results on the connectedness of a certain colored graph he associates to bounded width algebras \cite{bulatov-bounded}, \cite{colored-graph}. In Section \ref{s-pseudovariety} we prove our main results on the existence of a pseudovariety of minimal bounded width algebras (Theorem \ref{subalg} and Theorem \ref{minvar}). In Section \ref{s-connectivity}, we discuss several consequences of the ``yellow connectivity property'' proved by Bulatov \cite{bulatov-bounded}.

The last few sections consist of several classification results, examples, and conjectures based on these examples. In Section \ref{spiral}, we classify minimal bounded width algebras which have no majority subalgebras, which we call minimal \emph{spirals}, and give a recursive structure theory for them (Theorem \ref{recursive}). In Section \ref{s-three}, we classify minimal bounded width algebras of size $3$, and in Section \ref{s-two-gen} we classify minimal bounded width algebras of size $4$ which are generated by $2$ elements. In Section \ref{s-conj} we list several conjectures which are supported by the examples found in this paper.

In Appendix \ref{a-yellow}, we prove a refinement of Bulatov's yellow connectivity property. In Appendix \ref{a-special}, we construct an infinite family of special weak near-unanimity terms. In Appendix \ref{a-sm} we give a three element counterexample to the existence of terms satisfying the identities (SM 5) or (SM 6).

\section{Partial Semilattice Operations}\label{s-partial}

In this section we collect some results which generalize results about 2-semilattices, and which apply in all idempotent varieties (even those varieties which do not omit type $\mathbf{1}$). Most of these results were found by Bulatov in \cite{bulatov-bounded} and \cite{colored-graph}, but the presentation here is different. Lemma \ref{prepare} seems to be an entirely new result: it shows that one may modify the basic operations of an idempotent algebra (by ``preparing'' their inputs) in order to make any ``potential'' semilattice subalgebra into an an actual semilattice subalgebra without affecting any of the two-variable linear identities satisfied by these operations.

\begin{defn} We say that a binary operation $s$ is a \emph{partial semilattice} if it satisfies the identities
\[
s(x,s(x,y)) \approx s(s(x,y),x) \approx s(x,y),\;\;\; s(x,x) \approx x.
\]
\end{defn}

Note that the definition of a partial semilattice implies that for all $x,y$, the set $\{x,s(x,y)\}$ is closed under $s$, and the restriction of $s$ to this set is a semilattice operation. Partial semilattice operations are the binary analogue of unary operations $u$ which satisfy $u(u(x)) \approx u(x)$, in that they can be produced using an iteration argument.

\begin{lem}[Semilattice Iteration Lemma]\label{semi-iter} Let $t$ be a binary idempotent term of a finite algebra. Then there exists a partial semilattice $s \in \Clo(t)$ which is built out of $t$ in a nontrivial way (i.e. both $x$ and $y$ show up in the definition of $s(x,y)$ in terms of $t$). In particular, for any $a,b$ such that $t(a,b) = t(b,a) = b$, we have $s(a,b) = s(b,a) = b$.
\end{lem}
\begin{proof} Define $t^i(x,y)$ by $t^0(x,y) = y$, $t^1(x,y) = t(x,y)$, and
\[
t^{i+1}(x,y) = t(x,t^i(x,y)),
\]
and define $t^{\infty}(x,y)$ by 
\[
t^{\infty}(x,y) = \lim_{n\rightarrow \infty} t^{n!}(x,y) \approx t^{|\bA|!}(x,y).
\]
Note that we have
\[
t^{\infty}(x,t^{\infty}(x,y)) \approx t^{\infty}(x,y).
\]
Now define $u(x,y)$ by
\[
u(x,y) = t^{\infty}(x,t^{\infty}(y,x)),
\]
and define $u^i(x,y), u^{\infty}(x,y)$ in analogy with $t^i(x,y), t^{\infty}(x,y)$. We will show that $s = u^{\infty}$ is a partial semilattice.

First, note that $t^{\infty}(x,u^{\infty}(x,y)) \approx u^{\infty}(x,y)$, since we have
\begin{align*}
t^{\infty}(x,u^{i+1}(x,y)) &= t^{\infty}(x,u(x,u^i(x,y)))\\
&= t^{\infty}(x,t^{\infty}(x,t^{\infty}(u^i(x,y),x)))\\
&\approx t^{\infty}(x,t^{\infty}(u^i(x,y),x))\\
&= u(x,u^i(x,y)) = u^{i+1}(x,y)
\end{align*}
for any $i \ge 0$. Thus, for any $i \ge 0$ we have
\begin{align*}
u^{i+1}(u^{\infty}(x,y),x) &= u^i(u^{\infty}(x,y),u(u^{\infty}(x,y),x))\\
&= u^i(u^{\infty}(x,y),t^{\infty}(u^{\infty}(x,y),t^{\infty}(x,u^{\infty}(x,y))))\\
&\approx u^i(u^{\infty}(x,y),t^{\infty}(u^{\infty}(x,y),u^{\infty}(x,y)))\\
&\approx u^{\infty}(x,y),
\end{align*}
where the last step follows from idempotence. From this we see that
\[
u^{\infty}(u^{\infty}(x,y),x) \approx u^{\infty}(x,y),
\]
while
\[
u^{\infty}(x,u^{\infty}(x,y)) \approx u^{\infty}(x,y)
\]
follows directly from the definition of $u^{\infty}$.
\end{proof}

\begin{rem} Essentially the same result is proved in Proposition 10 of \cite{colored-graph}, with a slightly different construction. The construction given here has the nice additional property that for any $a,b \in \bA,$ the value of $s(a,b)$ may be computed from $t|_{\Sg_{\bA}\{a,b\}}$ in time polynomial in the size of $\Sg_{\bA}\{a,b\}$.
\end{rem}

From any partial semilattice operation $s$, we can define certain higher-arity terms $s_n$ which behave nicely when the number of distinct values among their inputs is at most two.

\begin{prop}\label{higher-semilattice} If $s$ is a partial semilattice operation, then for all $n \ge 1$ there are terms $s_n \in \Clo(s)$ of arity $n$ such that if $\{x,x_2, ..., x_n\} = \{x,y\}$, then
\[
s_n(x,x_2, ..., x_n) \approx s(x,y).
\]
\end{prop}
\begin{proof} Define $n$-ary functions $s_n(x_1, ..., x_n)$ by $s_1(x) = x, s_2(x,y) = s(x,y)$ and
\[
s_n(x_1, ..., x_n) = s(s_{n-1}(x_1, ..., x_{n-1}), s(x_1,x_n)).
\]
We will show by induction on $n$ that if $\{x,x_2, ..., x_n\} = \{x,y\}$, then
\[
s_n(x,x_2, ..., x_n) \approx s(x,y).
\]
There are three cases. If $x_2 = \cdots = x_{n-1} = x$, then we must have $x_n = y$, so
\[
s_n(x,x_2, ..., x_n) = s(s_{n-1}(x, ..., x), s(x,y)) \approx s(x,s(x,y)) \approx s(x,y).
\]
If $\{x,x_2, ..., x_{n-1}\} = \{x,y\}$ and $x_n = x$, then by the inductive hypothesis
\[
s_n(x,x_2, ..., x_n) = s(s_{n-1}(x, x_2 ..., x_{n-1}), s(x,x)) \approx s(s(x,y),x) \approx s(x,y).
\]
Finally, if $\{x,x_2, ..., x_{n-1}\} = \{x,y\}$ and $x_n = y$, then by the inductive hypothesis
\[
s_n(x,x_2, ..., x_n) = s(s_{n-1}(x, x_2 ..., x_{n-1}), s(x,y)) \approx s(s(x,y),s(x,y)) \approx s(x,y).\qedhere
\]
\end{proof}

Recall that an identity is \emph{linear} if it does not involve any nesting of operations. Remarkably, the class of bounded width algebras and the class of algebras which omit type $\mathbf{1}$ (i.e. Taylor algebras) can both be characterized by Mal'cev conditions consisting of finitely many two-variable linear identities which involve both variables on each side (see \cite{optimal-maltsev} and \cite{optimal-taylor}), making them well-suited to the next lemma.

\begin{lem}[Semilattice Preparation Lemma]\label{prepare} Let $\bA = (A, (f_i)_{i \in I})$ be a finite idempotent algebra, and let $\Sigma$ be the set of all two-variable linear identities which involve both variables on each side and are satisfied in $\bA$. Then $\bA$ has terms $(f_i')_{i \in I}$ which satisfy the identities in $\Sigma$, with the following property:

For every pair of subalgebras $\bB, \bC$ of $(A,(f_i')_{i\in I})$ such that there exists a term $t \in \Clo((f_i')_{i\in I})$ with $t(b,c), t(c,b) \in \bC$ whenever $b \in \bB, c \in \bC$, we in fact have $f_i'(b_1, ..., b_m) \in \bC$ whenever $b_1, ..., b_m \in \bB\cup \bC$ such that at least one of $b_1, ..., b_m$ is an element of $\bC$. In particular, if there is a $t\in \Clo((f_i')_{i\in I})$ with $t(b,c) = t(c,b) = c$, then $\{b,c\}$ is a semilattice subalgebra of $(A,(f_i')_{i\in I})$.
\end{lem}
\begin{proof} Suppose that $(f_i')_{i\in I}$ are chosen to satisfy the identities in $\Sigma$, such that the number of pairs of subalgebras $\bB, \bC$ of $(A,(f_i')_{i\in I})$ such that for each $m$-ary term $f_i'$ and any $b_1, ..., b_m \in \bB\cup \bC$ with at least one of the $b_i$s in $\bC$ we have $f_i'(b_1, ..., b_m) \in \bC$ is maximized. Suppose that there is a pair of subalgebras $\bB, \bC$ of $(A,(f_i')_{i\in I})$ and a term $t \in \Clo((f_i')_{i \in I})$ with $t(b,c), t(c,b) \in \bC$ for all $b \in \bB, c \in \bC$. Apply Lemma \ref{semi-iter} to produce a nontrivial partial semilattice $s \in \Clo(t)$. By the construction of $s$, we will have $s(b,c), s(c,b) \in \bC$ whenever $b\in \bB, c \in \bC$.


Define functions $s_n$ in terms of $s$ as in Proposition \ref{higher-semilattice}. Now for each $m$-ary $f_i'$, we define $f_i''$ by
\[
f_i''(x_1, ..., x_m) = f_i'(s_m(x_1, ..., x_m), s_m(x_2, ..., x_m, x_1), ..., s_m(x_m, x_1 ..., x_{m-1})).
\]
It is clear that each $f_i''$ is such that for each $m$-ary term $f_i''$ and any $b_1, ..., b_m \in \bB\cup \bC$ with at least one of the $b_i$s in $\bC$ we have $f_i''(b_1, ..., b_m) \in \bC$, since each $s_m$ has this property and $\bC$ is closed under $f_i'$. Now suppose we have an identity
\[
f_i'(a_1, ..., a_m) \approx f_j'(b_1, ..., b_n),
\]
with $\{a_1, ..., a_m\} = \{b_1, ..., b_n\} = \{x,y\}$. Define $a_1', ..., a_m'$ by $a_k' = s(x,y)$ if $a_k = x$ and $a_k' = s(y,x)$ if $a_k = y$, and define $b_1', ..., b_n'$ similarly. Then for each $k$, we have
\[
s_m(a_k, ..., a_m, a_1, ..., a_{k-1}) \approx a_k',
\]
and similarly for the $b_l'$s, so
\[
f_i''(a_1, ..., a_m) \approx f_i'(a_1', ..., a_m') \approx f_j'(b_1', ..., b_n') \approx f_j''(b_1, ..., b_n).\qedhere
\]
\end{proof}

\begin{defn} We say that an idempotent algebra $\bA$ has been \emph{prepared} if for every pair $a,b$ and every binary term $t$ such that $t(a,b) = t(b,a) = b$, $\{a,b\}$ is a semilattice subalgebra of $\bA$. A partial semilattice term $s$ of $\bA$ is called \emph{adapted} to $\bA$ if it can be built out of the basic operations of $\bA$ in a nontrivial way.
\end{defn}

\begin{prop} If $\bA$ has been prepared as above and $a,b,c \in \bA$ have $c \in \Sg\{a,b\}$ with $\{a,c\}$ a two element semilattice subalgebra directed from $a$ to $c$, then $\bA$ has a partial semilattice term $s$ with $s(a,b) = c$.
\end{prop}
\begin{proof} Let $s'$ be an arbitrary partial semilattice term which is adapted to $\bA$, and choose $p$ a binary term of $\bA$ with $p(a,b) = c$. Then take $s(x,y) = s'(x,p(x,y))$. We clearly have $s(a,b) = s'(a,p(a,b)) = s'(a,c) = c$, so we just have to check that $s$ is a partial semilattice. If $p$ is second projection then $s = s'$ and we are done. Otherwise, since $\bA$ has been prepared, $p$ and $s'$ act as the semilattice operation on $\{x,s'(x,p(x,y))\} = \{x,s(x,y)\}$. Thus,
\[
s(x,s(x,y)) = s'(x,p(x,s(x,y))) \approx s'(x,s(x,y)) \approx s(x,y),
\]
and
\[
s(s(x,y),x) = s'(s(x,y),p(s(x,y),x)) \approx s'(s(x,y),s(x,y)) \approx s(x,y).\qedhere
\]
\end{proof}

\begin{defn} If $s$ is a partial semilattice operation and $a,b$ have $s(a,b) = b$, then we write $a \rightarrow_s b$, or just $a \rightarrow b$ if $s$ is understood (or if the algebra has been prepared). We say that $b$ is \emph{reachable} from $a$ if there is a sequence $a = a_0, a_1, ..., a_k = b$ such that $a_i \rightarrow a_{i+1}$ for $i = 0, ..., k-1$.
\end{defn}

\begin{defn} We say that a subset $S$ of an algebra $\bA$ which has a partial semilattice operation $s$ is \emph{upwards closed} if whenever $a\in S$ and $a' \in \bA$ have $a \rightarrow_s a'$, we also have $a' \in S$.
\end{defn}

\begin{defn} We say that a set $A$ is \emph{strongly connected} if for every subset $S \subset A$ with $S\ne \emptyset, A$ there is an $a \in S$ and a $b \in A\setminus S$ such that $a \rightarrow b$. We say that a set $A$ is a \emph{maximal strongly connected component} of an algebra $\bA$ if $A$ is a strongly connected subset which is upwards closed (note that every finite upwards closed set contains at least one maximal strongly connected component). Finally, we call an element of an algebra $\bA$ \emph{maximal} if it is contained in any maximal strongly connected component of $\bA$.
\end{defn}

\begin{rem} If an algebra $\bA$ has been prepared and is strongly connected, then it has no proper absorbing subalgebra in the sense of \cite{cyclic}. However, even in the case of strongly connected algebras the next result is not a consequence of the Absorption Theorem of \cite{cyclic} since it applies even in varieties which have no Taylor term.
\end{rem}

\begin{defn} A relation $\RR \le_{sd} \bA\times\bB$ is \emph{linked} if for any two elements $a,a' \in \bA$ there exists a sequence $a = a_0, b_1, a_1, ..., b_k, a_k = a'$ with $(a_i,b_{i+1}) \in \RR$ and $(a_i,b_i) \in \RR$ for all $i$. In other words, $\RR$ is linked if it is connected when viewed as a bipartite graph on $\bA\sqcup \bB$, or equivalently when $\ker \pi_1 \vee \ker \pi_2$ is the total congruence of $\RR$. An element of $\bA$ is called a \emph{fork} for $\RR$ if there exist $b \ne b'$ such that $(a,b), (a,b') \in \RR$, and similarly for elements of $\bB$.
\end{defn}

\begin{thm}\label{strong-binary} Fix a partial semilattice operation $s$. Suppose $\RR \le_{sd} \bA \times \bB$ is subdirect and $A,B$ are maximal strongly connected components of $\bA, \bB$, respectively.
\begin{itemize}
\item[(a)] The set of $a$ such that $(\{a\} \times B) \cap \RR \ne \emptyset$ is upwards closed. In particular, if $(A\times B)\cap \RR$ is nonempty, then it is subdirect in $A\times B$.

\item[(b)] The set of $a$ such that $\{a\} \times B \subseteq \RR$ is upwards closed.

\item[(c)] If $A$ is contained in a linked component of $\RR$ (that is, a connected component of $\RR$ considered as a bipartite graph on $\bA \sqcup \bB$), $(A\times B) \cap \RR \ne \emptyset$, and $A,B$ are finite, then $A \times B \subseteq \RR$.
\end{itemize}
Additionally, the product $A\times B$ is a maximal strongly connected component of $\bA \times \bB$.
\end{thm}
\begin{proof} For part (a), suppose that $(a,b) \in \RR$ and $b \in B$, and let $a \rightarrow a'$. Since $\RR$ is subdirect, there is some $b'$ with $(a',b') \in \RR$. Then
\[
\begin{bmatrix} a'\\ s(b,b')\end{bmatrix} = s\left(\begin{bmatrix} a\\ b\end{bmatrix}, \begin{bmatrix} a'\\ b'\end{bmatrix}\right) \in \RR,
\]
and $b \rightarrow s(b,b')$, so $s(b,b') \in B$.

For part (b), suppose that $\{a\}\times B \subseteq \RR$ and $a \rightarrow a'$. Let $S$ be the set of $b \in B$ such that $(a',b) \in \RR$, that is, $S = \pi_2((\{a'\}\times B) \cap \RR)$. By part (a), $S$ is nonempty. To finish, we just have to show that $S$ is upwards closed. Suppose $b \in S$ and $b \rightarrow b'$. Then by assumption we have $(a,b') \in \RR$, so
\[
\begin{bmatrix} a'\\ b'\end{bmatrix} = s\left(\begin{bmatrix} a'\\ b\end{bmatrix}, \begin{bmatrix} a\\ b'\end{bmatrix}\right) \in \RR.
\]

For part (c), suppose first that $A \times A \subseteq \RR\circ \RR^-$, where $\RR^- \le \bB \times \bA$ is the reverse of $\RR$, and $\circ$ is relational composition of binary relations, defined by $R\circ S = \{(a,c) \mid \exists b\ (a,b) \in R \wedge (b,c) \in S\}$. Let $a$ be any element of $A$, and let $X$ be the set of $b \in \bB$ such that $(a,b) \in \RR$, that is, $X = \pi_2((\{a\}\times \bB) \cap \RR)$. By part (a), $X \cap B \ne \emptyset$, and by the finiteness of $B$, the intersection $X \cap B$ has a maximal strongly connected component $S$. Since $B$ is a maximal strongly connected component of $\bB$, $S$ is a maximal strongly connected component of $X$.

By the assumption $A \times A \subseteq \RR\circ \RR^-$ and the definition of $X$, we see that $(A\times X)\cap \RR$ is subdirect in $A\times X$. Thus by part (b) and the fact that $\{a\} \times S \subseteq (A\times X)\cap \RR$, we see that $A \times S \subseteq (A\times X)\cap \RR$, so $A \times S \subseteq \RR$. Then by part (b) applied to $\RR^-$, we see that $A \times B \subseteq \RR$.

Now suppose that $A \times A \not\subseteq \RR\circ \RR^-$. From the finiteness of $A$ we see that there is some $k$ such that $A \times A \subseteq (\RR\circ \RR^-)^{\circ k}$. Choose $k$ minimal, and let $\RR' = (\RR \circ \RR^-)^{\circ (k-1)} \le_{sd} \bA^2$. Then $\RR'$ is equal to its own reverse $\RR'^-$, and $A \times A \subseteq \RR' \circ \RR'$ since $2(k-1) \ge k$ for $k \ge 2$. Thus the previous case applied to $\RR'$ shows that $A \times A \subseteq \RR'$, contradicting the minimality of $k$.
\end{proof}

\begin{cor}[Lemma 9 of \cite{bulatov-bounded}]\label{simple-strong} Fix a partial semilattice operation $s$. Suppose that $\RR \le_{sd} \bA\times \bB$ is a subdirect product of finite algebras $\bA, \bB$, and that $\bB$ is simple and $\bB = \Sg(B)$, with $B$ a maximal strongly connected component of $\bB$. Then:
\begin{itemize}
\item[(a)] if $\bA$ is also simple and $\bA = \Sg(A)$ with $A$ a maximal strongly connected component of $\bA$, and if $\RR \cap (A\times B) \ne \emptyset$, then $\RR$ is either the graph of an isomorphism or $\bA\times \bB$, and

\item[(b)] if $\bA$ is arbitrary and $\RR$ is not the graph of a homomorphism from $\bA$ to $\bB$, then there is an $a \in \bA$ with $\{a\}\times \bB \subseteq \RR$.
\end{itemize}
\end{cor}
\begin{proof} If $\bB$ is simple, then the linking congruence of $\RR$ on $\bB$ must either be the trivial congruence $0_\bB$, in which case $\RR$ is the graph of a homomorphism from $\bA$ to $\bB$, or the full congruence $1_\bB$, in which case $\RR$ is linked. In the second case, the results follow from Theorem \ref{strong-binary}(c).
\end{proof}

\begin{thm}\label{strong-ternary} Fix a partial semilattice operation $s$. Suppose $R \subseteq A \times B \times C$ is closed under $s$, $A$ is strongly connected, $\pi_{23}(R)$ is strongly connected, $\pi_{12}(R) = A \times B$, $\pi_{13}(R) = A\times C$, and $A,B,C$ are finite. Then $R = A \times \pi_{23}(R)$.
\end{thm}
\begin{proof} By Theorem \ref{strong-binary}(c), we just need to show that $R$ is linked as a subset of $A \times \pi_{23}(R)$. We will do this by showing that for any $a \rightarrow a'$ in $A$, some fork of $R$ links $a$ to $a'$ in one step.

Since $\pi_1(R) = A$, there exist $b \in B, c \in C$ such that $(a,b,c) \in R$. Since $\pi_{13}(R) = A\times C$, there exists some $b' \in B$ such that $(a',b',c) \in R$. Since
\[
\begin{bmatrix} a'\\ s(b,b')\\ c\end{bmatrix} = s\left(\begin{bmatrix} a\\ b\\ c\end{bmatrix}, \begin{bmatrix} a'\\ b'\\ c\end{bmatrix}\right) \in R,
\]
we may assume without loss of generality that $b' = s(b,b')$, that is, that $b \rightarrow b'$.

Since $\pi_{12}(R) = A\times B$, there exists some $c' \in C$ such that $(a,b',c') \in R$. Since
\[
\begin{bmatrix} a\\ b'\\ s(c,c')\end{bmatrix} = s\left(\begin{bmatrix} a\\ b\\ c\end{bmatrix}, \begin{bmatrix} a\\ b'\\ c'\end{bmatrix}\right) \in R,
\]
we may assume without loss of generality that $c' = s(c,c')$, that is, that $c \rightarrow c'$.

Since $(a',b',c)$ and $(a,b',c')$ are in $R$, we have
\[
\begin{bmatrix} a'\\ b'\\ c'\end{bmatrix} = s\left(\begin{bmatrix} a'\\ b'\\ c\end{bmatrix}, \begin{bmatrix} a\\ b'\\ c'\end{bmatrix}\right) \in R.
\]
Thus both $a$ and $a'$ meet $(b',c') \in \pi_{23}(R)$.
\end{proof}

\begin{cor}[Lemmas 10, 12, 13 of \cite{bulatov-bounded}]\label{triple} Fix a partial semilattice operation $s$. If $R \subseteq A\times B\times C$ is closed under $s$, $A,B,C$ are finite and strongly connected, and $\pi_{1,2} R = A\times B, \pi_{1,3} R = A\times C, \pi_{2,3} R = B\times C$, then $R = A\times B\times C$.
\end{cor}

\begin{defn} An algebra is \emph{polynomially complete} if the clone generated by its basic operations together with all constant functions is the clone of all operations.
\end{defn}

\begin{thm} Suppose $\bA$ is a simple idempotent algebra with a partial semilattice term $s$, which is generated by a maximal strongly connected component $A$. Then $\bA$ is polynomially complete.
\end{thm}
\begin{proof} It's enough to show that every subalgebra $\RR \le \bA^n$ which contains every diagonal tuple and which does not project onto the equality relation on any pair of coordinates is equal to $\bA^n$. First, note that by Corollary \ref{simple-strong} and the fact that $\pi_{i,j} \RR$ properly contains the equality relation, we have $\pi_{i,j} \RR = \bA^2$ for every $i,j$. We may assume by induction that every projection of $\RR$ onto any proper subset of its coordinates is the full relation.

We want to apply Corollary \ref{triple} to $R = \RR \cap A^n \subseteq A\times A\times A^{n-2}$. To this end, we will show that $\pi_{1,2} R = A\times A$, as $\pi_{2,3} R = \pi_{1,3} R = A\times A^{n-2}$ can be proved in exactly the same way. Thinking of $\RR$ as a subdirect product of $(\bA\times \bA) \times \bA^{n-2}$, from $\RR \cap (A\times A)\times A^{n-2} \ne \emptyset$ we see by Theorem \ref{strong-binary}(a) that $R$ is subdirect in $(A\times A)\times A^{n-2}$, and we are done.
\end{proof}

\section{Cycle Consistency and $pq$ instances}\label{s-cycle}

We will use the homomorphism description of the constraint satisfaction problem. Let $\fA$, $\fB$ be two relational structures with the same signature $\sigma$. The constraint satisfaction problem for the pair $\fB,\fA$ asks whether there exists a homomorphism $\fB\rightarrow \fA$. It might be helpful to think of $\fA$ as a set of possible values together with library of constraint relations indexed by $\sigma$, and to think of $\fB$ as an edge-colored, directed hypergraph on a set of variables, where the colors are the elements of $\sigma$, and to think of the homomorphism as an assignment of values to the variables satisfying the constraints corresponding to the edges of $\fB$.

If we restrict to constraint satisfaction problems with fixed target $\fA$, the resulting constraint satisfaction problem is denoted $\CSP(\fA)$. The complexity of this problem is known to only depend on the algebra of polymorphisms of $\fA$, which we will denote $\bA$, so we may also sometimes speak of the problem $\CSP(\bA)$, by which we mean $\CSP(\Inv(\bA))$, where $\Inv(\bA)$ is the relational structure on the same underlying set having as relations all (underlying sets of) subpowers of the algebra $\bA$.

We now define cycle consistent and $pq$ instances of a constraint satisfaction problem, and state the main theorem of \cite{slac} which we will use as a black box.

\begin{defn} Let $\fA, \fB$ be relational structures with the same signature $\sigma$. The instance $(\fB,\fA)$ is called \emph{1-minimal} if for any $R,S \in \sigma$ and any $(v_1, ..., v_k) \in R^{\fB}, (w_1, ..., w_l) \in S^{\fB}$ and $i,j$ such that $v_i = w_j$, we have $\pi_i(R^{\fA}) = \pi_j(S^{\fA})$. In this case, if $b$ is any element of $\fB$, we write $\fA_b$ for $\pi_i(R^{\fA})$, where $R \in \sigma, (v_1, ..., v_k) \in R^{\fB}, 1\le i \le k, v_i = b$ (note that this is well-defined by 1-minimality), or $\fA_b = \fA$ if no such $R, (v_1, ..., v_k), i$ exist.
\end{defn}

\begin{defn} Let $(\fB,\fA)$ be a 1-minimal instance. A \emph{pattern} $p$ of length $k-1>0$ from $b_1$ to $b_k$ (in $\fB$) is a tuple
\[
p = (b_1, (R_1,i_1,j_1), b_2, (R_2,i_2,j_2), ..., (R_{k-1},i_{k-1},j_{k-1}), b_k)
\]
such that each $b_l \in \fB$, each $R_l \in \sigma$ and for each $l$ there is a tuple $(t_1, ..., t_m) \in R_l^{\fB}$ such that $t_{i_l} = b_l, t_{j_l} = b_{l+1}$. The pattern $p$ is \emph{closed} if $b_1 = b_k$.


A \emph{realization} of $p$ (in $\fA$) is a tuple $(f_1, ..., f_{k-1})$ such that for each $l$ we have $f_l \in R_l^{\fA}$ and $\pi_{j_l}(f_l) = \pi_{i_{l+1}}(f_{l+1})$, and we say this realization sends $b_{l+1}$ to $\pi_{j_l}(f_l) = \pi_{i_{l+1}}(f_{l+1})$. We say that this realization \emph{connects} $\pi_{i_1}(f_1)$ to $\pi_{j_{k-1}}(f_{k-1})$, and we say that $p$ \emph{connects} $x \in \fA_{b_1}$ to $y\in \fA_{b_k}$ if it has a realization connecting $x$ to $y$.



If $p = (b_1, (R_1,i_1,j_1), ..., b_k)$ and $q = (b_k, (R_k,i_k,j_k), ..., b_{k+l})$, we put $p + q = (b_1, (R_1,i_1,j_1), ..., b_{k+l})$ (note that the last element of $p$ must match the first element of $q$). For a closed pattern $p$ and $m \ge 1$ we put $m\times p = p + p + \cdots + p$, with $m$ copies of $p$ on the right hand side.
\end{defn}

\begin{defn} An instance $(\fB,\fA)$ is \emph{cycle consistent} if it is 1-minimal and for any $b\in \fB$, any closed pattern $p$ from $b$ to $b$, and any $x \in \fA_b$, $p$ connects $x$ to $x$.
\end{defn}

\begin{defn} An instance $(\fB,\fA)$ is a $pq$ instance if it is 1-minimal and for any $b\in \fB$, any pair of closed patterns $p,q$ from $b$ to $b$, and any $x \in \fA_b$, there exists $m \ge 0$ such that $m\times(p+q) + p$ connects $x$ to $x$.
\end{defn}

Note that every cycle consistent instance is automatically a $pq$ instance.


\begin{thm}[Kozik, \cite{slac}] If a finite algebra $\bA$ has bounded width, then every $pq$ instance $(\fB,\fA)$ has a solution.
\end{thm}

Cycle consistent instances and $pq$ instances are preserved by taking the closure with respect to an algebra of polymorphisms:

\begin{prop}\label{gen} If $(\fB,\mathbf{S})$ is a $pq$ instance, and $\bA$ is a finite algebra with underlying set containing the underlying set of $\mathbf{S}$ and we define $\mathbf{C}$ to be the relational structure such that for each $R\in\sigma$ of arity $k$ we have $R^{\mathbf{C}} = \Sg_{\bA^k}(R^{\mathbf{S}})$ (here $\Sg_{\bA^k}(X)$ is the subalgebra generated by $X$ in $\bA^k$), then $(\fB,\mathbf{C})$ is also a $pq$ instance. In particular, if $\bA$ has bounded width then the instance $(\fB,\mathbf{C})$ has a solution.
\end{prop}

\section{Intersecting families of sets}\label{s-intersect}

\begin{defn} Let $S$ be a set. A family $\cF\subseteq \cP(S)$ is called an \emph{intersecting family} of subsets of $S$ if $A,B\in \cF$ implies $A\cap B\ne 0$.
\end{defn}

\begin{prop} An intersecting family of subsets of a set $S$ is maximal (with respect to containment) if and only if for every set $A\subseteq S$ we have either $A\in\cF$ or $(S\setminus A)\in\cF$.

For every $n\ge 1$ there is a bijection between the collection of maximal intersecting families $\cF$ of subsets of $\{1,...,n\}$ and the collection of self-dual monotone boolean functions $f:\{0,1\}^n\rightarrow \{0,1\}$.
\end{prop}

\begin{thm}\label{intersect} Let $\cV$ be a locally finite idempotent variety of bounded width. Then there is an idempotent term $f$ of arity $2$ of $\cV$ and a family of idempotent terms $h_{\cF}$ in $\cV$, indexed by maximal intersecting families $\cF$ of subsets of $\{1, ..., n\}$ for $n\ge 2$, satisfying the following:

For each maximal intersecting family $\cF$ of $\{1,...,n\}$, $h_{\cF}$ has arity $n$, and for every sequence $(a_1, ..., a_n)$ such that $\{a_1, ..., a_n\} = \{x,y\}$ and $\{i \in \{1,...,n\}\mid a_i = x\} \in \cF$ we have
\[
h_{\cF}(a_1, ..., a_n) \approx f(x,y).
\]
Furthermore, we can choose $f$ such that for any binary term $t$ in the clone generated by the functions $h_{\cF}$ and any pair $a,b$ of distinct elements of some algebra in $\cV$, if $t(a,b) = t(b,a) = a$ then $f(a,b) = f(b,a) = a$, and such that additionally we have
\[
f(f(x,y),f(y,x)) \approx f(x,y).
\]
\end{thm}
\begin{proof} The strategy is to apply Proposition \ref{gen} with $\mathbf{S}$ a relational structure on the two element set $\{x,y\}$, with signature $\sigma$ equal to the collection of all maximal intersecting families $\cF$ of subsets of $\{1, ..., n\}$ for all $n \ge 2$ and $\bA$ equal to the free algebra on two generators in $\cV$, $\cF_{\cV}(x,y)$ (which is finite and bounded width by assumption). A solution to the resulting instance $(\fB,\mathbf{C})$ will then correspond to a family of terms, one for each occurrence of each relation symbol $R\in \sigma$ in $\fB$ and one for each element of $\fB$. In order to show that these terms satisfy the required identities we will think of the solution to this instance as a coloring of the vertices of the directed hypergraph $\fB$, and use a pigeonhole argument to show that each $\cF^{\fB}$ contains an edge such that all of its vertices have the same color.

First we have to describe how we realize the relations $\cF$ in $\mathbf{S}$. If $\cF$ is a maximal intersecting family of subsets of $\{1, ..., n\}$, first we fix any ordering $A_1, ..., A_{2^{n-1}-1}$ of the elements of $\cF\setminus \{\{1,...,n\}\}$ such that $|A_1| \le |A_2| \le \cdots \le |A_{2^{n-1}-1}|$. Then $\cF^{\mathbf{S}}$ will be the arity $2^{n-1}-1$ relation with elements $t^1, ..., t^n \in \{x,y\}^{2^{n-1}-1}$ given by
\[
t^i_j = x\ \iff\ i\in A_j.
\]

Next we describe $\fB$. As the underlying set, we take the set of natural numbers $\mathbb{N}$. For each $\cF$ a maximal intersecting family of subsets of $\{1, ..., n\}$, we take $\cF^{\fB}$ to be the set of all strictly increasing sequences of $\mathbb{N}$ of length $2^{n-1}-1$.

Since no element of any $\cF\setminus \{\{1,...,n\}\}$ is either the empty set or the entire set, for each $1 \le j \le 2^{n-1}-1$ we have $\pi_j(\cF^{\mathbf{S}}) = \{x,y\}$, so the instance $(\fB,\mathbf{S})$ is 1-minimal. In order to check that it is a cycle consistent instance, we note that for any $1 \le i < j \le 2^{n-1}-1$ the relation $\pi_{i,j}(\cF^{\mathbf{S}})$ necessarily contains both $(x,x)$ (since $\cF$ is an intersecting family) and $(y,x)$ (since $|A_i| \le |A_j|$ and $A_i \ne A_j$), and contains at least one of $(x,y), (y,y)$ (since $A_j \ne \{1,...,n\}$).

For any closed pattern $p = (b_1, ..., b_k)$ in $\fB$, there must be some $j$ with $b_{j-1} < b_j > b_{j+1}$ (indices considered cyclically modulo $k-1$). If $1 < j < k$, then $p$ already connects $x$ to $x$ and $y$ to $y$ through realizations that send $b_j$ to $x$. Otherwise, for any closed pattern $q$ from $b_k$ to $b_1$, $p+q+p$ connects $x$ to $x$ and $y$ to $y$ through realizations that send every vertex in $q$ to $x$. Thus $(\fB,\mathbf{S})$ is a $pq$ instance.

Applying Proposition \ref{gen}, we find a solution to the instance $(\fB,\mathbf{C})$ where $\mathbf{C}$ is the closure of $\mathbf{S}$ in $\cF_{\cV}(x,y)$. Applying the infinite pigeonhole principle, we see that there is an infinite subset $X$ of $\fB$ which is all assigned the same value $f(x,y) \in \cF_{\cV}(x,y)$. Thus since each $\cF^{\fB}$ contains a tuple from $X$, there must be an element of $\Sg_{\cF_{\cV}(x,y)^{2^{n-1}-1}}(\cF^{\mathbf{S}})$ such that each coordinate is equal to $f(x,y)$. Calling the $n$ elements of $\cF^{\mathbf{S}}$ $t^1, ..., t^n$ as above, we then have a term $h_{\cF}$ such that each coordinate of $h_{\cF}(t^1, ..., t^n)$ is $f(x,y)$, which is what was needed.

In order to guarantee that whenever there is a term $t$ with $t(a,b) = t(b,a) = b$ we have $f(a,b) = f(b,a) = b$, we apply Lemma \ref{prepare} to the free algebra on $\cV$ with two generators.

Now we show how to modify these functions to ensure that $f(f(x,y),f(y,x)) \approx f(x,y)$. First we define the \emph{dictator family} $\mathcal{D}_n = \{A \subseteq \{1,...,n\}\mid 1 \in A\}$, and set $f_n(a_1, ..., a_n) = h_{\mathcal{D}_n}(a_1, ..., a_n)$. Then if $\{x,a_2, ...,a_n\} = \{x,y\}$, we have
\[
f_n(x,a_2, ..., a_n) \approx f(x,y).
\]
Now we define a sequence of functions $f^i(x,y)$ by $f^0(x,y) = x$, $f^1(x,y) = f(x,y)$,
\[
f^{i+1}(x,y) = f(f^i(x,y),f^i(y,x)),
\]
and define a sequence of functions $h^i_{\cF}(a_1, ..., a_n)$ by $h^1(a_1, ..., a_n) = h(a_1, ..., a_n)$,
\[
h^{i+1}_{\cF}(a_1, ..., a_n) = h^i_{\cF}(f_n(a_1, ..., a_n),f_n(a_2, ..., a_n,a_1), ..., f_n(a_n, a_1, ..., a_{n-1})).
\]
It's then easy to show by induction on $i \ge 1$ that if $\{a_1, ..., a_n\} = \{x,y\}$ and $\{i \in \{1,...,n\}\mid a_i = x\} \in \cF$, then
\[
h^i_{\cF}(a_1, ..., a_n) \approx f^i(x,y).
\]
Letting $\phi:\cF_{\cV}(x,y)^2 \rightarrow \cF_{\cV}(x,y)^2$ be the map $(a,b) \mapsto (f(a,b),f(b,a))$, we see from the finiteness of $\cF_{\cV}(x,y)^2$ that there is some $i$ such that $\phi^{\circ 2i} = \phi^{\circ i}$, and for this $i$ we have
\[
f^i(f^i(x,y),f^i(y,x)) \approx f^i(x,y).
\]
Replacing $f$ by $f^i$ and $h_{\cF}$ by $h^i_{\cF}$ finishes the proof.
\end{proof}

A term $t_3$ satisfying (SM 3) is now given by $h_{\mathcal{H}}$, where
\[
\mathcal{H} = \{\{1,2\},\{1,3\},\{1,4\},\{1,2,3\},\{1,2,4\},\{1,3,4\},\{2,3,4\},\{1,2,3,4\}\}.
\]

Note that the identity
\[
f(f(x,y),f(y,x)) \approx f(x,y)
\]
implies that for any $a,b$ we either have $f(a,b) = f(b,a)$ or the two element set $\{f(a,b),f(b,a)\}$ is closed under all of the operations $h_{\cF}$, and that they act on this set as the corresponding self-dual monotone functions. Furthermore, setting $g = h_{\mathcal{M}}$, where
\[
\cM = \{\{1,2\},\{1,3\},\{2,3\},\{1,2,3\}\},
\]
we see that $g$ acts as a majority operation on the set $\{f(a,b),f(b,a)\}$. Since every self-dual monotone function on a two element set is contained in the clone generated by the majority function, it's easy to show that we can in fact write all the functions $h_{\cF}$ in terms of the $f_n$s (corresponding to the dictator families) and $g$. Curiously, it doesn't seem easy to find a way to write the functions $f_n$ for $n \ge 3$ directly in terms of $g$.

\begin{cor}\label{min} A locally finite variety $\cV$ has bounded width if and only if it has an idempotent weak majority term $g$ satisfying the identity
\[
g(g(x,x,y),g(x,x,y),g(x,y,y)) \approx g(x,x,y)\ (\approx g(x,y,x) \approx g(y,x,x)).
\]
Furthermore, if $\cV$ has bounded width we can choose $g$ such that there exist $f,h_{\cF}$ as in the previous theorem with $g = h_{\cM}$ where $\cM = \{\{1,2\},\{1,3\},\{2,3\},\{1,2,3\}\}$, such that each $h_{\cF} \in \Clo(g)$, and such that $g$ satisfies the additional identity
\[
g(g(x,y,z),g(y,z,x),g(z,x,y)) \approx g(x,y,z).
\]
\end{cor}
\begin{proof} We've already shown that if $\cV$ has bounded width then such a $g$ exists. In order to prove the reverse implication, it's enough to show that such a $g$ can't be realized in a nontrivial module over a ring. Suppose for contradiction that $\mathbb{M}$ is a module over a ring $\RR$, and that we have $g^{\mathbb{M}}(x,y,z) \approx ax + by + cz$ for some fixed $a,b,c \in \RR$. Since $g(x,x,y) \approx g(x,y,x) \approx g(y,x,x)$, we have $ax\approx bx\approx cx$, and from $g(x,x,x) \approx x$, we have $3ax \approx x$. From
\[
g(g(x,x,y),g(x,x,y),g(x,y,y)) \approx g(x,x,y),
\]
we see that
\[
5a^2x + 4a^2y \approx 2ax + ay.
\]
Multiplying both sides by $9$ and using $3ax\approx x$, we get
\[
5x + 4y \approx 6x + 3y,
\]
so $x\approx y$ and the module $\mathbb{M}$ must be trivial.

In particular, for any such $g$ the clone $\Clo(g)$ must be bounded width as well. Choose such a $g$ such that $\Clo(g)$ contains as few terms of arity 3 as possible. By Theorem \ref{intersect}, there are terms $f,h_{\cF} \in \Clo(g)$. Let $g' = h_{\cM}$, then by the choice of $g$ we must have $\Clo(g) = \Clo(g')$, so in particular all $h_{\cF} \in \Clo(g')$. Replace $g$ by this $g'$ to get the second to last assertion of the Corollary.

In order to prove the last assertion, consider the map $\gamma: \cF_{\cV}(x,y,z)^3 \rightarrow \cF_{\cV}(x,y,z)^3$ defined by
\[
\gamma:(a,b,c)\mapsto (g(a,b,c),g(b,c,a),g(c,a,b)).
\]
Since $\cV$ is locally finite, there exists $i \ge 1$ such that $\gamma^{\circ 2i} = \gamma^{\circ i}$. Note that if any two of $a,b,c$ are equal, then $\gamma^{\circ i}(a,b,c) = (g(a,b,c),g(a,b,c),g(a,b,c))$ for all $i\ge 1$, so we may replace $g$ by the first coordinate of $\gamma^{\circ i}$ without changing the value of $g(x,x,y)$.
\end{proof}

\begin{cor} Let $\bA$ be a finite bounded width algebra. If $\Aut(\bA)$ is $2$-transitive, then $\bA$ has a majority term. If $\Aut(\bA)$ is $3$-transitive, then $\bA$ has the dual discriminator as a term.
\end{cor}
\begin{proof} Since $\Aut(\bA)$ is transitive, every unary operation of $\bA$ must be surjective, so $\bA$ is core and we may replace it with its idempotent reduct. Choose $f,g$ as in the previous corollary. If there is any pair $x,y \in \bA$ such that $f(x,y) \ne f(y,x)$, then from $f(f(x,y),f(y,x)) = f(x,y)$ and $2$-transitivity we see that $f$ is first projection and $g$ is a majority operation.

Now suppose that $f(x,y) = f(y,x)$ for all $x,y \in \bA$. Choose any $a \ne b$ in $\bA$, and let $\sigma$ be an automorphism of $\bA$ such that $\sigma(a) = f(a,b)$ and $\sigma(b) = a$. Define a sequence of functions $g_i$ by $g_1 = g$ and $g_{i+1}(x,y,z) \approx g_i(f(y,z),f(z,x),f(x,y))$. Since $\sigma$ is an automorphism of $\bA$, we have
\[
g_{i+1}(a,a,b) = g_i(f(a,b),f(b,a),a) = g_i(\sigma(a),\sigma(a),\sigma(b)) = \sigma(g_i(a,a,b)) = \sigma^{i+1}(a)
\]
by induction on $i$, and similarly we get $g_{i+1}(a,b,a) = g_{i+1}(b,a,a) = \sigma^{i+1}(a)$. Choosing $n$ such that $\sigma^n$ fixes $a$, we see that $g_n(a,a,b) = g_n(a,b,a) = g_n(b,a,a) = a$, and thus by $2$-transitivity $g_n$ is a majority operation.

Now suppose that $\Aut(\bA)$ is $3$-transitive. Let $g$ be a majority operation on $\bA$. By iteration, we may assume that $g$ also satisfies the identity
\[
g(g(x,y,z),y,z) \approx g(x,y,z).
\]
Let $a,b,c$ be any three distinct elements of $\bA$. If $g(a,b,c) \in \{a,b,c\}$, then by $3$-transitivity (and possibly permuting the inputs of $g$) we are done. Otherwise, letting $d = g(a,b,c)$ we see that $d,b,c$ are three distinct elements of $\bA$ with $g(d,b,c) = d$, and we are done by $3$-transitivity.
\end{proof}

\begin{thm}\label{connect} If $\cV$ is a locally finite idempotent bounded width variety with terms $f, g$ as in Corollary \ref{min} such that the clone generated by $g$ is minimal, then there is a sequence of binary terms $p_0, ..., p_n \in \Clo(f)$ such that $p_0(x,y) \approx x$, $p_n(x,y) \approx y$, and for each $0 \le i < n$, the set $\{p_i(x,y), p_{i+1}(x,y)\} \subseteq \cF_{\cV}(x,y)$ is closed under $g$.
\end{thm}
\begin{proof} Let $f,g$ be terms as in Corollary \ref{min}, let $f_3 = h_{\mathcal{D}_3} \in \Clo(g)$ be as in the proof of Theorem \ref{intersect}, let $\bA$ be an algebra in $\cV$, and let $A$ be a subset of $\bA$ which is closed under $f$. Set $\bA_f = (A,f)$, the algebra on the underlying set $A$ with $f$ as its only basic operation. Define a graph $\cG(\bA_f) = \cG(\bA_f,g)$ on $\bA_f$ with an edge connecting a pair of elements $a,b \in \bA_f$ whenever $\{a,b\}$ is closed under $g$. We just need to show that each such graph $\cG(\bA_f)$ is connected. Suppose for contradiction that $\bA_f$ has minimal size such that $\cG(\bA_f)$ is not connected.

For every element $a \in \bA_f$, let $C(a)$ be the connected component of $a$ in $\cG(\bA_f)$. Letting $a,b \in \bA_f$ be any pair of elements with $C(a) \ne C(b)$, from the minimality of $\bA_f$ we see that $\Sg_{\bA_f}\{a,b\} = \bA_f$. Define a subalgebra $\bS = \bS_{a,b}$ of $\bA_f^2$ to be $\Sg_{\bA_f^2}\{(a,b),(b,a)\}$. Since $\{a,b\}$ is not an edge of $\cG(\bA_f)$, the Semilattice Preparation Lemma \ref{prepare} and the minimality of $\Clo(g)$ implies that $(a,a), (b,b) \not\in \bS$. Since $\Sg_{\bA_f}\{a,b\} = \bA_f$, we see that $\bS$ is subdirect in $\bA_f^2$.

Suppose now that $(u,t),(v,w) \in \bS$ and that $\{u,v\}$ is an edge of $\cG(\bA_f)$. Let $\bB = \bB_{u,v} \le \bA_f$ be the subalgebra of $x \in \bA_f$ such that at least one of $(u,x), (v,x)$ is in $\bS$, i.e.
\[
\bB = \pi_2((\{u,v\}\times \bA_f) \cap \bS).
\]
Suppose, for contradiction, that $\bB = \bA_f$. Then at least one of $C(a)$ or $C(b)$ is not equal to $C(u) = C(v)$, say $C(a)$. From $b \in \bB$, we see that at least one of $(u,b), (v,b) \in \bB$, say $(u,b) \in \bB$, and then from $C(a) \ne C(u)$ we see that $(b,b) \in \bA_f\times\{b\} = \Sg_{\bA_f^2}\{(a,b),(u,b)\} \subseteq \bS$, a contradiction.

Thus $\bB$ must be a proper subalgebra of $\bA_f$, and from the minimality of $\bA_f$ the graph $\cG(\bB)$ must be connected. In particular, since $t,w \in \bB$ we see that $C(t) = C(w)$. By a straightforward induction on the length of the path, we now see that for any $(u,t), (v,w) \in \bS$ with $C(u) = C(v)$ we must have $C(t) = C(w)$. Thus there must exist an involution $\iota$ on the set $\mathcal{C}$ of connected components of $\cG(\bA_f)$ such that
\[
\bS \subseteq \bigcup_{C \in \mathcal{C}} C \times \iota(C).
\]

Suppose, for contradiction, that $C(f(a,b)) = C(a)$. From $(a,b),(b,a) \in\bS$ we see that $(f(a,b),f(b,a)) \in \bS$, and so $C(f(b,a)) = \iota(C(f(a,b))) = \iota(C(a)) = C(b)$. Since $\{f(a,b),f(b,a)\}$ is a subalgebra of $\bA_f$, we see that $C(f(a,b)) = C(f(b,a))$ as well, so $C(a) = C(b)$, a contradiction. Thus, for any $a,b$ with $C(a) \ne C(b)$, we have $C(a) \ne C(f(a,b))$ as well.

Defining a sequence of binary terms $f^i$ by $f^0(x,y) = y$, $f^1 = f$, and
\[
f^{i+1}(x,y) = f(x,f^i(x,y)),
\]
we see by induction on $i$ that if $C(a) \ne C(b)$ then $C(a) \ne C(f^i(a,b))$. Letting $N \ge 1$ be such that $f^N = f^{2N}$, we see that $C(a) \ne C(f^N(a,b))$ and $f^N(a,f^N(a,b)) = f^N(a,b)$. Replacing $b$ by $f^N(a,b)$, we may assume without loss of generality that $f^N(a,b) = b$.

At this point, we will attempt to construct functions $f',g',f_3'$ such that
\[
g'(x,x,y) \approx g'(x,y,x) \approx g'(y,x,x) \approx f_3'(x,x,y) \approx f_3'(x,y,x) \approx f_3'(x,y,y) \approx f'(x,y)
\]
and $f'(a,b) = b$.

Define a sequence of ternary terms $g^i$ by $g^0(x,y,z) = z$, $g^1 = g$, and
\[
g^{i+1}(x,y,z) = g(x,y,g^i(x,y,z)).
\]
We now define $f',g',f_3'$ by
\[
f'(x,y) = f(f^N(x,y),f^{N-1}(y,f^N(x,y))),
\]
\begin{align*}
g'(x,y,z) &= g(f^{N-1}(x,g^{N-1}(f(y,z),f(z,y),g(x,y,z))),\\
&\qquad \qquad f^{N-1}(y,g^{N-1}(f(z,x),f(x,z),g(y,z,x))),f^{N-1}(z,g^{N-1}(f(x,y),f(y,x),g(z,x,y)))),
\end{align*}
and
\begin{align*}
f_3'(x,y,z) &= f_3(f^{N-1}(x,f_3(x,y,z)),\\
&\qquad \qquad f^{N-1}(y,f^{N-1}(x,f(f(x,y),f_3(x,y,z)))),f^{N-1}(z,f^{N-1}(x,f(f(x,z),f_3(x,y,z))))).
\end{align*}
We have the identity
\begin{align*}
g'(x,x,y) &\approx g(f^{N-1}(x,g^{N-1}(f(x,y),f(y,x),f(x,y))),\\
&\qquad \qquad f^{N-1}(x,g^{N-1}(f(y,x),f(x,y),f(x,y))),f^{N-1}(y,g^{N-1}(x,x,f(x,y))))\\
&\approx g(f^N(x,y),f^N(x,y),f^{N-1}(y,f^N(x,y)))\\
&\approx f'(x,y),
\end{align*}
and similarly $g'(x,y,x) \approx g'(y,x,x) \approx f'(x,y)$. We also have
\begin{align*}
f_3'(x,x,y) &\approx f_3(f^{N-1}(x,f(x,y)),\\
&\qquad \qquad f^{N-1}(x,f^{N-1}(x,f(x,f(x,y)))), f^{N-1}(y,f^{N-1}(x,f(f(x,y),f(x,y)))))\\
&\approx f_3(f^N(x,y), f^{2N}(x,y), f^{N-1}(y,f^N(x,y)))\\
&\approx f'(x,y),
\end{align*}
and similarly $f_3'(x,y,x) \approx f'(x,y)$, as well as
\begin{align*}
f_3'(x,y,y) &\approx f_3(f^{N-1}(x,f(x,y)),\\
&\qquad \qquad f^{N-1}(y,f^{N-1}(x,f(f(x,y),f(x,y)))), f^{N-1}(y,f^{N-1}(x,f(f(x,y),f(x,y)))))\\
&\approx f_3(f^N(x,y), f^{N-1}(y,f^N(x,y)), f^{N-1}(y,f^N(x,y)))\\
&\approx f'(x,y),
\end{align*}
so the terms $f', g', f_3'$ do indeed satisfy the identity
\[
g'(x,x,y) \approx g'(x,y,x) \approx g'(y,x,x) \approx f_3'(x,x,y) \approx f_3'(x,y,x) \approx f_3'(x,y,y) \approx f'(x,y).
\]
Additionally, from $f^N(a,b) = b$, we have
\[
f'(a,b) = f(f^N(a,b),f^{N-1}(b,f^N(a,b))) = f(b,f^{N-1}(b,b)) = b.
\]

Using $f', g', f_3'$ we will construct terms $f''$, $g''$ with
\begin{align*}
f''(x,y) &\approx g''(x,x,y) \approx g''(x,y,x) \approx g''(y,x,x)\\
&\approx f''(f''(x,y),f''(y,x))
\end{align*}
and with $C(f''(a,b)) = C(a)$, $C(f''(b,a)) = C(b)$. Then $g''$ generates a bounded width algebra by Corollary \ref{min}, so by minimality of $\Clo(g)$ we will have $\Clo(g) = \Clo(g'')$. Thus, since $\{f''(a,b),f''(b,a)\}$ will be closed under $g''$, it will also be closed under $g$, so $\{f''(a,b),f''(b,a)\}$ will be an edge of $\cG(\bA_f)$ which connects $C(a)$ and $C(b)$, giving us a contradiction.

In order to construct $f'', g''$, we define sequences $f'^i, g'^i$ by $f'^1 = f', g'^1 = g'$, and
\begin{align*}
f'^{i+1}(x,y) &= f'^i(f'(x,y),f'(y,x))\\
g'^{i+1}(x,y,z) &= g'^i(f_3'(x,y,z),f_3'(y,z,x),f_3'(z,x,y)).
\end{align*}
We then choose an even $K \ge 1$ such that $f'^K = f'^{2K}$ and take $f'' = f'^K, g'' = g'^K$.

It remains to show that $C(f''(a,b)) = C(a)$ and $C(f''(b,a)) = C(b)$. We will prove by induction on $i$ that, considering $\bS$ as an undirected graph on $\bA_f$, $a$ is connected to $f'^i(a,b)$ by a sequence of $i$ edges of $\bS$ such that the first edge of the sequence is $(a,b)$, from which it follows that $a$ is connected to $f'^i(b,a)$ by a sequence of $i+1$ edges of $\bS$ since $(f'^i(a,b),f'^i(b,a)) \in \bS$.

For $i = 1$, this follows from $(a,b) = (a,f'(a,b)) \in \bS$ and $(f'(a,b),f'(b,a)) = (b,f'(b,a)) \in \bS$. For the induction step, let $u_0, u_1, ..., u_i$ be such that $u_0 = a, u_1 = b, u_i = f'^i(a,b)$ with $(u_j,u_{j+1}) \in \bS$ for $0 \le j < i$, and set $u_{i+1} = f'^i(b,a)$ so that $(u_i,u_{i+1}) \in \bS$ as well. Now for $1 \le j \le i+1$, set $v_j = f'(u_{j-1},u_j)$, and set $v_0 = a$. Then $(v_0, v_1) = (a, f'(a,b)) = (a,b) \in \bS$, and for $1 \le j \le i$ we have
\[
\begin{bmatrix} v_j\\ v_{j+1}\end{bmatrix} = f'\left(\begin{bmatrix} u_{j-1}\\ u_j\end{bmatrix}, \begin{bmatrix} u_j\\ u_{j+1}\end{bmatrix}\right) \in \bS,
\]
and $v_{i+1} = f'(u_i,u_{i+1}) = f'(f'^i(a,b),f'^i(b,a)) = f'^{i+1}(a,b)$. Thus $v_0, ..., v_{i+1}$ is a path through $i+1$ edges of $\bS$ connecting $a$ to $f'^{i+1}(a,b)$ with $v_1 = b$, and the induction is complete.

By $\bS \subseteq \bigcup_{C \in \mathcal{C}} C \times \iota(C)$, we have
\[
C(f'^i(a,b)) = \iota^i(C(a)) = \begin{cases}C(a) & i \equiv 0 \pmod{2},\\ C(b) & i \equiv 1 \pmod{2},\end{cases}
\]
for each $i$, so $C(f''(a,b)) = C(f'^K(a,b)) = C(a)$, $C(f''(b,a)) = \iota(C(f''(a,b))) = C(b)$, and we are done.
\end{proof}

\section{Pseudovariety of minimal bounded width clones}\label{s-pseudovariety}

\begin{defn}\label{mindef} For each $k \ge 1$, let $\cG_k$ be the set of pairs $f,g$ where $g:\{1,...,k\}^3\rightarrow \{1,...,k\}$ is an idempotent weak majority function satisfying
\[
g(x,x,y) \approx g(x,y,x) \approx g(y,x,x) \approx f(x,y),
\]
\[
f(f(x,y),f(y,x)) \approx f(x,y),
\]
and
\[
g(g(x,y,z),g(y,z,x),g(z,x,y)) \approx g(x,y,z).
\]
Define a quasiorder $\preceq$ on $\cG_k$  by $g' \preceq g$ if $g' \in \Clo(g)$. Also, define an action of $S_2\times S_k$ on $\cG_k$ by having the nontrivial element of $S_2$ take $g\in\cG$ to $\tilde{g}$ given by
\[
\tilde{g}(x,y,z) = g(x,z,y)
\]
and by having $\sigma\in S_k$ take $g$ to $\sigma g$ given by $(\sigma g)(x,y,z) = \sigma(g(\sigma^{-1}(x),\sigma^{-1}(y),\sigma^{-1}(z)))$.
\end{defn}

\begin{prop} The clones on $\{1,...,k\}$ of bounded width which are minimal with respect to containment are in bijection with the minimal equivalence classes of the quasiorder $\preceq$ on $\cG_k$. In particular, the number of such clones is at most $|\cG_k|$, which is in turn trivially bounded by $k^{k^3}$.
\end{prop}

Straightforward computer calculations give the values $|\cG_3| = 2,\!728$ and $|\cG_4| = 8,\!124,\!251,\!747,\!605$.

\begin{defn}  Say that $g \in \cG_k$ is \emph{minimal} if it is in a minimal equivalence class of the quasiorder $\preceq$, say that an algebra $\bA$ is a minimal bounded width algebra if it is isomorphic to an algebra which is term equivalent to $(\{1,...,|\bA|\},g)$ for some minimal $g \in \cG_{|\bA|}$, and say that a variety $\cV$ is a minimal bounded width variety if every finite algebra $\bA \in \cV$ is minimal.
\end{defn}

The proof of the next result is based on an argument of Bulatov, specifically, the proof of Case 2.2 of Theorem 5 of \cite{colored-graph}.

\begin{thm}\label{subalg} Let $\cV$ be a locally finite idempotent bounded width variety such that every algebra in $\cV$ is connected through two element semilattice and two element majority subalgebras, and let $\bA \in \cV$.

Suppose there is a term $m$ and a subset $S \subseteq \bA$ which is closed under $m$, such that $(S,m)$ is a bounded width algebra. Let $\cV'$ be the reduct of $\cV$ consisting of all terms $t$ of $\cV$ such that $S$ is closed under $t$ and such that $t|_S \in \Clo(m|_S)$. Then $\cV'$ also has bounded width.
\end{thm}
\begin{proof} For any algebra $\bF \in \cV$, let $\bF'$ be the reduct of $\bF$ having the same underlying set and having as operations all operations of $\cV'$. For any algebra $\bF$, define a graph $\cG(\bF)$ on $\bF$ by connecting two vertices $a,b$ of $\bF$ whenever there is a term $t$ of $\bF$ such that $\{a,b\}$ is closed under $t$ and such that $(\{a,b\},t)$ has bounded width. It's enough to show that for every finite $\bF \in \cV$ and every subalgebra $\bB' \le \bF'$ we have $\cG(\bB')$ connected, since every algebra in $\cV'$ is a quotient of a subalgebra of the reduct $\bF'$ of some algebra $\bF$ of $\cV$, and since no affine algebra can have a two element bounded width algebra as a subalgebra. We will prove this by induction on the size of $\bB'$, which we may assume to be generated by two elements $a,b \in \bF$.

First we will show that every two element subalgebra of $\bF$ is in $\cG(\bF')$. We may as well suppose that $|\bF| = 2$, so that $\bF$ is either a semilattice or a majority algebra. Let $p > |S|$ be prime. By one of the main results of \cite{cyclic}, there is a $p$-ary term $c \in \Clo(m)$ such that $c|_S$ is a cyclic term. Let $u$ be a $p$-ary term of $\cV$ such that its restriction to $\bF$ is a cyclic term. Setting
\[
u'(x_1, ..., x_p) \approx u(c(x_1, ..., x_p), c(x_2, ..., x_p, x_1), ..., c(x_p, x_1, ..., x_{p-1})),
\]
we see that $u'|_S = c|_S$, so $u'$ is a term of $\cV'$. Additionally, the restriction of $u'$ to $\bF$ is a cyclic term, hence it is either a semilattice operation or a near-unanimity operation, and we see that $\bF'$ is term equivalent to $\bF$.

Now let $a,b \in \bF$, and suppose that $\bB' = \Sg_{\bF'}\{a,b\}$ has $\cG(\bB')$ disconnected, but that every proper subalgebra $\mathbb{C}'$ of $\bB'$ has $\cG(\mathbb{C}')$ connected. Let $\gS$ be the collection of ordered pairs $(p_i,q_i)$ of binary terms in the clone of $m$ such that $p_i|_S = q_i|_S$. Set
\[
D_\gS = \{(p_i(c,d),q_i(c,d)) \mid c,d \in \bB', (p_i, q_i) \in \gS\},
\]
and let $\theta$ be the congruence of $\bB' = \Sg_{\bF'}\{a,b\}$ generated by $D_\gS$. Letting
\[
T_\gS = \{(t(a,b,p_i(c,d)),t(a,b,q_i(c,d))) \mid c,d \in \bB', (p_i,q_i) \in \gS, t \text{ a term of } \cV'\}
\]
and considering $T_{\gS}$ as a graph on $\bB'$, we see that $\theta$ is the transitive closure of $T_{\gS}$, since $T_\gS$ contains the image of $D_{\gS}$ under all unary polynomials of $\bB'$.

Let $(t(a,b,p_i(c,d)),t(a,b,q_i(c,d)) \in T_{\gS}$. Let $r$ be any binary term of $\cV$, and define a $4$-ary term $r'$ by
\[
r'(x,y,z,w) = r(t(x,y,p_i(z,w)), t(x,y,q_i(z,w))).
\]
We clearly have $r'(x,y,z,w)|_S = t(x,y,p_i(z,w))|_S$, so $r' \in \cV'$, and thus
\[
r(t(a,b,p_i(c,d)),t(a,b,q_i(c,d))) = r'(a,b,c,d) \in \bB'
\]
Since $r$ was an arbitrary binary term of $\cV$ we have
\[
\Sg_{\bF}\{t(a,b,p_i(c,d)),t(a,b,q_i(c,d))\} \subseteq \bB'
\]
so, by our assumption that every algebra of $\cV$ was connected through two element subalgebras and the fact that every two element subalgebra of $\bF$ is in $\cG(\bF')$, we see that $t(a,b,p_i(c,d)),t(a,b,q_i(c,d))$ are in the same connected component of $\cG(\bB')$. Since $\cG(\bB')$ was assumed to be disconnected, we see that $T_\gS$ is disconnected and therefore $\bB'/\theta$ has at least two elements. By the choice of $\bB'$, every congruence class of $\theta$ is a proper subalgebra of $\bB'$, and is therefore connected in $\cG(\bB')$.

By the choice of $\theta$, $p_i$ agrees with $q_i$ on $\bB'/\theta$ whenever $(p_i,q_i) \in \gS$, so the reduct $(\bB'/\theta,m)$ is in the variety generated by $(S,m)$ (we only need to consider identities involving two variables since $\bB'$ is generated by two elements) and therefore has bounded width. Since by assumption (and Theorem \ref{connect}) $\bB'$ does not have bounded width and every proper subalgebra of $\bB'$ does have bounded width, there must be a congruence $\theta'$ such that $\bB'/\theta'$ is a nontrivial affine algebra, and by the choice of $\bB'$ every congruence class of $\theta'$ is connected in $\cG(\bB')$. Since it is both affine and bounded width, $\bB'/(\theta\vee \theta')$ must be a trivial algebra, so the transitive closure of $\theta \cup \theta'$ is $\bB'$, and we see that in fact $\cG(\bB')$ is connected.
\end{proof}

An immediate consequence of Theorem \ref{subalg} is that every subalgebra and every quotient of a minimal bounded width algebra is also a minimal bounded width algebra. It's easy to see that this also holds for powers, so in fact every minimal bounded width algebra generates a minimal bounded width variety. The same is not true for products, however: as we will see, up to the action of $S_2\times S_3$ there are two minimal elements $g \in \cG_3$ with $g \ne \tilde{g}$, and for either one of these the product $(\{1,2,3\},g)\times (\{1,2,3\}, \tilde{g})$ has $(\{1,2,3\},g)^2$ as a reduct and is thus not a minimal bounded width algebra.

\begin{thm}\label{minvar} There exists an idempotent pseudovariety $\cV_{mbw}$ with basic operations $f,g$ satisfying the identities in Definition \ref{mindef} such that every finite algebra in $\cV_{mbw}$ is a minimal bounded width algebra, and such that every finite core bounded width algebra has a reduct which is term equivalent to an algebra in $\cV_{mbw}$.
\end{thm}
\begin{proof} Fix an enumeration $\bA_1, \bA_2, ...$ of the collection of all finite algebras (up to isomorphism) in the variety with basic operations $f,g$ satisfying the identities in Definition \ref{mindef}. For each $i \ge 1$, let $g_i$ be a ternary term of $\bA_1\times \cdots \times \bA_i$ such that when restricted to $\bA_1\times \cdots \times \bA_{i-1}$ it is in the clone generated by $g_{i-1}$, and such that the reduct of $\bA_1\times \cdots \times \bA_i$ with basic operation $g_i$ is a minimal bounded width algebra - that such a term $g_i$ exists follows from Theorem \ref{subalg}. Since there are only finitely many possible ternary terms $g_i|_{\bA_1}$ on $\bA_1$, one such must occur infinitely many times - call this term $g^1$. Similarly, we inductively define a sequence of ternary terms $g^i$ such that for each $i$, there are infinitely many $j > i$ with $g_j|_{\bA_k} = g^k$ for all $k \le i$. Finally, we let $\cV_{mbw}$ be the pseudovariety generated by the algebras $\bA_i' = (A_i, f^i, g^i)$, where each $A_i$ is the underlying set of $\bA_i$, and $f^i$ is given by $f^i(x,y) = g^i(x,x,y)$.

From the construction of $\cV_{mbw}$, it is clear that every finite core bounded width algebra has a reduct which is term equivalent to an algebra in $\cV_{mbw}$. Now suppose that $\bA$ is a finite algebra in $\cV_{mbw}$. Then $\bA$ is in the variety generated by finitely many algebras $(\bA_i,g^i)$, so there is some $j$ such that $\bA$ is in the variety generated by the minimal bounded width algebra $(\bA_1\times\cdots\times\bA_j,g_j)$, so by Theorem \ref{subalg} $\bA$ must be a minimal bounded width algebra as well.
\end{proof}

\section{A connectivity result of Bulatov}\label{s-connectivity}

We need a result proved implicitly by Bulatov in \cite{bulatov-bounded}. Note that by Theorem \ref{intersect}, a pair $a,b$ in a minimal bounded width algebra forms a semilattice subalgebra directed from $a$ to $b$ if and only if we have $f(a,b) = b$, since in this case the function $t(x,y) \approx f(x,f(x,y))$ has $t(a,b) = f(a,f(a,b)) = f(a,b) = b$ and $t(b,a) = f(b,f(b,a)) = f(f(a,b),f(b,a)) = f(a,b) = b$. Thus, by Theorem \ref{subalg}, a pair $a,b$ with $f(a,b) = b$ is the same as what Bulatov calls a ``thin semilattice edge'' or a ``thin red edge'' in \cite{colored-graph} and \cite{bulatov-bounded}.

From the above discussion, we see that reachability of $y$ from $x$ is the same as the existence of a (directed) thin red path from $x$ to $y$ in Bulatov's colored graph attached to $\bA$, and that being maximal in $\bA$ is the same as being contained in what Bulatov calls a ``maximal strongly connected component'' of $\bA$. The key result we need from Bulatov's work is his ``yellow connectivity property''.


\begin{prop}[Yellow Connectivity Property]\label{yellow} If $\bA$ is a minimal bounded width algebra and $A,B$ are upwards closed subsets of $\bA$, then there are $a \in A$ and $b \in B$ such that $\{a,b\}$ is a majority subalgebra of $\bA$.
\end{prop}
\begin{proof} This follows from Proposition 6 of \cite{bulatov-bounded}, since every minimal bounded width algebra is necessarily what Bulatov calls a conglomerate algebra and since every upwards closed set contains a maximal strongly connected component. For the sake of completeness we give a proof of a slight refinement of the yellow connectivity property, based on Bulatov's argument, in Appendix \ref{a-yellow}.
\end{proof}

\begin{lem}\label{maj-triple} Suppose that $\bA$ is generated by $a,b \in \bA$, and let $\RR$ be a subalgebra of $\bA^3$ which contains $\Sg_{\bA^3}\{(a,a,b),(a,b,a),(b,a,a)\}$. Let $U,V,W$ be any three maximal strongly connected components of $\bA$. If $\RR \cap (U\times V\times W) \ne \emptyset$, then $U \times V \times W \subseteq \RR$.
\end{lem}
\begin{proof} By the definition of $\RR$ and the assumption $\bA = \Sg\{a,b\}$, we see that $\{a\}\times \bA \subseteq \pi_{1,2}\RR$. Thus by Theorem \ref{strong-binary}(c) we have $U\times V \subseteq \pi_{1,2}\RR$. Let $R = \RR\cap (U\times V\times W)$. Considering $\RR$ as a subdirect product of $\pi_{1,2}\RR$ and $\bA$, we conclude from Theorem \ref{strong-binary}(a) that $\RR\cap (U\times V\times W)$ is subdirect in $(U\times V)\times W$. Thus $\pi_{1,2} R = U\times V$, and similarly $\pi_{1,3}R = U\times W$, $\pi_{2,3}R = V\times W$. Applying Corollary \ref{triple}, we have $R = U\times V\times W$.
\end{proof}

\begin{lem}\label{max-maj} If $\bA$ is a minimal bounded width algebra and $a,b \in \bA$ are such that $a$ is a maximal element of $\Sg_{\bA}\{a,b\}$, then $(a,a,a) \in \Sg_{\bA^3}\{(a,a,b),(a,b,a),(b,a,a)\}$.
\end{lem}
\begin{proof} We may assume without loss of generality that $\bA = \Sg\{a,b\}$. Let $A$ be a maximal strongly connected component of $\bA$ which contains $a$, and set $\RR = \Sg_{\bA^3}\{(a,a,b),(a,b,a),(b,a,a)\}$.

Letting $B$ be a maximal strongly connected component which is reachable from $b$, we see from Theorem \ref{strong-binary}(a) and Lemma \ref{maj-triple} that $A\times A \times B \subseteq \RR$. By the yellow connectivity property, there are $a' \in A, b' \in B$ such that $\{a',b'\}$ is a majority subalgebra. Then since $(a',a',b') \in A\times A\times B \subseteq \RR$, we have
\[
\begin{bmatrix} a'\\ a'\\ a'\end{bmatrix} = g\left(\begin{bmatrix} a'\\ a'\\ b'\end{bmatrix}, \begin{bmatrix} a'\\ b'\\ a'\end{bmatrix}, \begin{bmatrix} b'\\ a'\\ a'\end{bmatrix}\right) \in \RR \cap (A\times A\times A).
\]
Now we can apply Lemma \ref{maj-triple} to see that we have $A\times A\times A \subseteq \RR$. Thus $(a,a,a) \in \RR$.
\end{proof}

\begin{thm}\label{maj-crit} If $a,b$ are distinct elements of a minimal bounded width algebra $\bA$ such that $(a,b)$ is a maximal element of $\Sg_{\bA^2}\{(a,b),(b,a)\}$, then $\{a,b\}$ is a majority subalgebra of $\bA$.
\end{thm}
\begin{proof} By Theorem \ref{subalg}, we just need to show that
\[
\begin{bmatrix} a & b\\a & b\\a & b\end{bmatrix} \in \Sg_{\bA^6}\left\{\begin{bmatrix} a & b\\a & b\\b & a\end{bmatrix}, \begin{bmatrix} a & b\\b & a\\a & b\end{bmatrix}, \begin{bmatrix} b & a\\a & b\\a & b\end{bmatrix}\right\}.
\]
This follows from the previous lemma applied to $\Sg_{\bA^2}\{(a,b),(b,a)\}$.
\end{proof}

\begin{cor}\label{reach} In a minimal bounded width algebra $\bA = (A,f,g)$, the functions $f,g$ can always be chosen such that for all $x,y$, $f(x,y)$ is reachable from $x$.
\end{cor}
\begin{proof} Choose a partial semilattice term $s \in \Clo(g)$ which is adapted to $\bA$. Let $\cF_{\bA}(x,y)$ be the free algebra generated by $\bA$ on $x,y$. By Theorem \ref{subalg} $\cF_{\bA}(x,y)$ will also be a minimal bounded width algebra.

Let $p_0(x,y), ..., p_k(x,y)$ be elements of $\cF_{\bA}(x,y)$ such that $p_0(x,y) = x$, $p_{i+1}(x,y) \in \Sg_{\cF_{\bA}(x,y)}\{p_i(x,y), p_i(y,x)\}$, $p_i(x,y) \rightarrow p_{i+1}(x,y)$, and such that $p_k(x,y)$ is maximal in $\Sg_{\cF_{\bA}(x,y)}\{p_k(x,y),p_k(y,x)\}$. By Theorem \ref{maj-crit} we see that $\{p_k(x,y), p_k(y,x)\}$ is a majority algebra, so
\[
f(p_k(x,y),p_k(y,x)) \approx p_k(x,y).
\]
Let $q_i(x,y)$ be such that $p_i(x,y) = q_i(p_{i-1}(x,y),p_{i-1}(y,x))$, and set $q_i'(x,y) = s(x,q_i(x,y))$, so that $p_i(x,y) \approx q_i'(p_{i-1}(x,y),p_{i-1}(y,x))$ and $q_i'$ is a partial semilattice term. Now we inductively define terms $r_0(x,y,z), ..., r_k(x,y,z)$ by $r_0(x,y,z) = x$, and
\[
r_{i+1}(x,y,z) = q_i'(q_i'(r_i(x,y,z),r_i(y,z,x)),q_i'(r_i(x,y,z),r_i(z,x,y))).
\]
Then one sees, by induction on $i$, that
\[
r_i(x,x,y) \approx r_i(x,y,x) \approx r_i(x,y,y) \approx p_i(x,y).
\]
Now we set $f'(x,y) = p_k(x,y)$ and $g'(x,y,z) = g(r_k(x,y,z),r_k(y,z,x),r_k(z,x,y))$. These satisfy the identities
\[
f'(x,y) \approx g'(x,x,y) \approx g'(x,y,x) \approx g'(y,x,x) \approx f'(f'(x,y),f'(y,x)),
\]
and $f'(x,y)$ is reachable from $x$ by construction.
\end{proof}

\begin{rem} If we use the stronger version of the Yellow Connectivity Property proved in Appendix \ref{a-yellow}, we can show that we may choose $f$ such that not only is $f(x,y)$ reachable from $x$ in $\Sg\{x,y\}$, in fact it is reachable in the set of elements which may be generated from $x,y$ using only the function $f$ (and not $g$). This refinement is useful for classifying small minimal bounded width algebras by hand.
\end{rem}

\begin{cor} If a minimal bounded width algebra has no semilattice subalgebra, then it is a majority algebra.
\end{cor}

\begin{cor} Every minimal bounded width algebra has a partial semilattice term $s(x,y)$ such that there is a sequence of terms $p_0(x,y), ..., p_k(x,y)$ with $p_0(x,y) \approx y$, $p_k(x,y) \approx s(x,y)$, and for each $i$ either $p_i(x,y) \rightarrow p_{i+1}(x,y)$ or $\{p_i(x,y), p_{i+1}(x,y)\}$ is a majority algebra.
\end{cor}
\begin{proof} Choose an $f$ such that $f(x,y)$ is reachable from $x$, and note that $f(y,x)$ is reachable from $y$ and is connected to $f(x,y)$ by a majority algebra. Now apply the Semilattice Iteration Lemma \ref{semi-iter} to $f$, and check by induction that every nontrivial binary operation in $\Clo(f)$ has a path from both of its inputs which travels up along semilattice edges and across along majority edges.
\end{proof}

\section{Spirals}\label{spiral}

Suppose that $\bA = (A,f,g)$ is a minimal bounded width algebra with no two element majority subalgebra. Since $\{f(x,y), f(y,x)\}$ is always a majority subalgebra, this immediately implies that we have $f(x,y) \approx f(y,x)$ in $\bA$. Let $\bA_f = (A,f)$ be the reduct obtained by dropping $g$. By Theorem \ref{connect}, there exist binary terms $p_0, ..., p_n$ of $\bA_f$ such that for any $x,y \in \bA_f$ we have $p_0(x,y) = x, p_n(x,y) = y$, and for each $i$, the set $\{p_i(x,y),p_{i+1}(x,y)\}$ is a subalgebra of $\bA$. Since $\bA$ has no two element majority subalgebras, this must in fact be a semilattice subalgebra, so it is also a semilattice subalgebra of $\bA_f$. Thus any quotient of $\Sg_{\bA_f}\{x,y\}$ which does not identify $x$ and $y$ must contain a two element semilattice, and so $\bA_f$ has bounded width. Since $\bA$ was assumed minimal, we have proved the following claim.

\begin{prop} If $\bA = (A,f,g)$ is a minimal bounded width algebra with no majority subalgebra, then $g \in \Clo(f)$ and $f(x,y) \approx f(y,x)$.
\end{prop}

Thus we will disregard $g$ in this section, and focus solely on $f$. The motivating example for this section is the algebra depicted in Figure \ref{spiral-six}.

\begin{figure}
\begin{tabular}{p{4.0cm} p{3.4cm}}
\begin{minipage}{\linewidth} \includegraphics[]{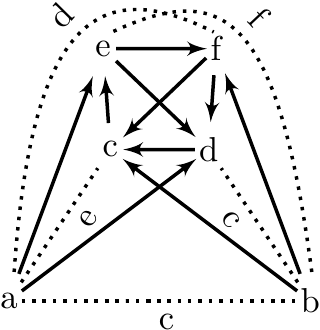} \end{minipage} & \begin{tabular}{c | c c c c c c} $f$ & a & b & c & d & e & f\\ \hline a & a & c & e & d & e & d\\ b & c & b & c & c & f & f\\ c & e & c & c & c & e & c\\ d & d & c & c & d & d & d\\ e & e & f & e & d & e & f\\ f & d & f & c & d & f & f \end{tabular}
\end{tabular}

\caption{A minimal spiral which is not a 2-semilattice.}\label{spiral-six}
\end{figure}

\begin{lem} Suppose $\bA$ is a minimal bounded width algebra with no majority subalgebra, and that $f$ is chosen as in Corollary \ref{reach}. Let $S$ be the set of all maximal elements of $\bA$. Then for any $x\in \bA$ and any $y \in S$, $y$ is reachable from $x$.

In addition, $S$ is a subalgebra of $\bA$, and for any $x \in \bA \setminus S$ the set $S\cup \{x\}$ is a subalgebra of $\bA$ which has a congruence $\theta$ given by the partition $\{S, \{x\}\}$ such that the quotient $(S\cup \{x\})/\theta$ is a semilattice directed from $\{x\}$ to $S$.
\end{lem}
\begin{proof} Since $f$ is commutative, for any $x,y \in \bA$ we see that $f(x,y) = f(y,x)$ is reachable from both $x$ and $y$. Thus, if $y \in S$ then we see from the definition of maximality that $y$ is reachable from $f(x,y)$, which is reachable from $x$, so $y$ is reachable from $x$. Since anything which is reachable from an element of $S$ is necessarily in $S$, we see that for any $x \in \bA$ and any $y \in S$ we have $f(x,y) = f(y,x) \in S$.
\end{proof}

\begin{lem} Suppose $\bA$ is a minimal bounded width algebra with no majority subalgebra which is generated by two elements $a,b \in \bA$. Let $S$ be the set of all maximal elements of $\bA$. Then $\bA = S \cup \{a,b\}$.
\end{lem}
\begin{proof} If $S$ contains $a$ or $b$, then we are done by the previous Lemma. Otherwise, since $S\cup\{a\}$ is a minimal bounded width algebra as well (by Theorem \ref{subalg}), we may apply Theorem \ref{connect} to see that there is an element $p(a,b) \in S$ such that $\{a,p(a,b)\}$ is a semilattice. Letting $s$ be a partial semilattice term adapted to $\bA$ and setting $p'(x,y) = s(x,p(x,y))$, we get $p'(a,b) = p(a,b)$ and $\{x,p'(x,y)\}$ is a semilattice for all $x,y \in \bA$.

Then we may use the proof of Corollary \ref{reach}, taking $p_1(x,y) = p'(x,y)$ and $p_2(x,y) = s(p'(x,y),p'(y,x))$, to see that we may choose $f$ such that $f(x,y) \in \Sg_{\cF_{\bA}(x,y)}\{p_2(x,y), p_2(y,x)\}$, so that in particular we have $f(a,b) \in \Sg_{\bA}\{s(p'(a,b),p'(b,a)),s(p'(b,a),p'(a,b))\}$. Since $p'(a,b) \in S$ and $S \cup \{p'(b,a)\}$ has the properties described in the previous lemma, we see that $s(p'(a,b),p'(b,a))$ and $s(p'(b,a),p'(a,b))$ are in $S$, and so $f(a,b) \in S$. Thus $S\cup\{a,b\}$ is a subalgebra of $\bA$, and so it must be equal to all of $\bA$.
\end{proof}

\begin{thm}\label{recursive} Suppose $\bA$ is a minimal bounded width algebra with no majority subalgebra and which is generated by two elements $a,b \in \bA$. Then either $\bA$ is a two element semilattice, or on letting $S$ be the set of maximal elements of $\bA$ we have $\bA = S \cup \{a,b\}$ with $S \cap \{a,b\} = \emptyset$.

In the second case, $\bA$ has a congruence $\theta$ corresponding to the partition $\{S, \{a\}, \{b\}\}$ such that $\bA/\theta$ is isomorphic to the free semilattice on the generators $\{a\}, \{b\}$.
\end{thm}
\begin{proof} By the previous lemmas, we just need to show that if $\bA$ is not a two element semilattice then $S \cap \{a,b\} = \emptyset$. Suppose for contradiction that $b \in S$. We split into two cases, based on whether $S$ is contained in a linked component of the relation $\RR = \Sg_{\bA^2}\{(a,b),(b,a)\}$.

If $S$ is contained in a linked component of $\RR$, then by Theorem \ref{strong-binary}(c) we have $S\times S \subseteq \RR$, so in particular $(b,b) \in \RR$. This implies that there is a term $t$ such that $t(a,b) = t(b,a) = b$, so $\{a,b\}$ is a two element semilattice, which is a contradiction.

If $S$ is not contained in a linked component of $\RR$, then let $\theta$ be the linking congruence of $\RR$ considered as a congruence on $\bA$, i.e. $\theta = \pi_1(\ker \pi_1 \vee \ker \pi_2)$. In $\bA/\theta$, we have $|S/\theta| > 1$ and $b/\theta \in S/\theta$, so we may assume without loss of generality that $\theta$ is trivial, that is, that $\RR$ is the graph of an automorphism which swaps $a$ and $b$. Then we have $\{a,b\} \subseteq S$, and $\bA$ is strongly connected. By Lemma \ref{maj-triple} we then have $(a,a,a) \in \Sg_{\bA^3}\{(a,a,b),(a,b,a),(b,a,a)\}$, so there is a ternary term which acts as the majority operation on $\{a,b\}$, and so by Theorem \ref{subalg} we see that $\{a,b\}$ is a majority subalgebra of $\bA$, a contradiction.
\end{proof}

\begin{defn} An algebra $\bA = (A,f)$ is a \emph{spiral} if $f$ is a commutative idempotent binary operation and every subalgebra of $\bA$ which is generated by two elements either has size two or has a surjective homomorphism to the free semilattice on two generators. $\bA$ is a \emph{weak spiral} if $f$ is a commutative idempotent binary operation such that for every $x,y \in \bA$, the sequence $f(x,y), f(x,f(x,y)), f(x,f(x,f(x,y))), ...$ is eventually constant.
\end{defn}

\begin{cor} If $\bA$ is a minimal bounded width algebra with no majority subalgebra, then $\bA$ is a spiral. Every spiral is a weak spiral.
\end{cor}
\begin{proof} The first statement follows directly from Theorem \ref{recursive}, so we just need to show that every spiral is a weak spiral. Let $y_0 = y, y_1 = f(x,y)$, $y_{i+1} = f(x,y_i)$. If $f(x,y_i) \ne y_i$ for all $i$, then by Theorem \ref{recursive} we see that $y_i \not\in \Sg_{\bA}\{x,f(x,y_i)\} = \Sg_{\bA}\{x,y_{i+1}\}$, so the size of $\Sg_{\bA}\{x,y_i\}$ is a strictly decreasing function of $i$ which always takes positive integer values, a contradiction.
\end{proof}

\begin{prop} Every weak spiral has bounded width.
\end{prop}
\begin{proof} Define a sequence of terms $f_i(x,y)$ by $f_0(x,y) = y, f_1(x,y) = f(x,y)$, and $f_{i+1}(x,y) = f(x,f_i(x,y))$. We need to show that the identities $f(x,x) \approx x$, $f(x,y) \approx f(y,x)$, and $f_k(x,y) \approx f_{k+1}(x,y)$ can't simultaneously be satisfied in a nontrivial affine algebra.

Suppose for a contradiction that they do hold in some affine algebra, with $f(x,y) \approx \alpha x + \beta y$. From $f(x,x) \approx x$, we see that $(\alpha + \beta) x \approx x$, and from $f(x,y) \approx f(y,x)$ we see that $(\alpha - \beta)x \approx (\alpha - \beta)y$. Thus $2\alpha x \approx 2\beta x \approx x$, and by induction on $n$ we see that $2^nf_n(x,y) \approx (2^n-1)x + y$. Then by multiplying both sides of $f_k(x,y) \approx f_{k+1}(x,y)$ by $2^{k+1}$, we see that $2(2^k-1)x + 2y \approx (2^{k+1}-1)x + y$, so $y \approx x$ and our affine algebra is in fact trivial.
\end{proof}

\begin{cor} Every nontrivial reduct of a spiral has bounded width. In particular, if $\bA$ is a minimal spiral and $t$ is any term of $\bA$ which is not a projection, then $f \in \Clo(t)$.
\end{cor}
\begin{proof} Since $t$ is a nontrivial element of $\Clo(f)$, upon restricting $t$ and $f$ to any two element semilattice contained in $\bA$ we see that there must be some two variable minor $p(x,y)$ of $t$ which acts as the semilattice operation on any two element semilattice of $\bA$.

We just need to show that for every subset $P$ of $\bA$ which is closed under $p$, the graph on $P$ with edges corresponding to two element semilattices of $\bA$ is connected. So suppose that $x,y \in P$ are not connected via a chain of two element semilattices, and that $P$ is minimal such that such a pair $x,y$ exist. Since $\{x,y\}$ is not itself a two element semilattice, we see from the definition of a spiral that we can write $\Sg_{\bA}\{x,y\} = S \cup \{x,y\}$ for some $S$ which is closed under $f$, $S \cap \{x,y\} = \emptyset$, and that $p(x,y) \in S$ since $p$ acts as the semilattice operation on the three element quotient of $S \cup \{x,y\}$. Thus $P \cap (S\cup\{x\})$ is also closed under $p$, contains $p(x,y)$, and does not contain $y$, so by the minimality of $P$ it contains a chain of two element semilattices connecting $x$ to $p(x,y)$. Similarly $P \cap (S\cup\{y\})$ contains a chain of two element semilattices connecting $p(x,y)$ to $y$, so $x,y$ are connected by a chain of two element semilattices contained in $P$.
\end{proof}

\begin{cor} A bounded width algebra $\bA$ has no nontrivial proper reducts if and only if $\bA$ is either a minimal majority algebra or a minimal spiral.
\end{cor}
\begin{proof} If $\bA$ has no majority subalgebra then this follows from the results of this section. Otherwise, if $\bA$ has a majority subalgebra $\{a,b\}$ then the restriction of $f$ to $\{a,b\}$ is first projection, so $\Clo(f)$ does not have bounded width and is therefore a proper reduct. Thus $f$ must be first projection, and $\bA$ is a minimal majority algebra.

It remains to show that any nontrivial reduct of a majority algebra is also a majority algebra. Suppose that $t$ is a nontrivial term of the majority algebra $\bA$ with basic operation $g$. The restriction of $t$ to any two element subset of $\bA$ is then necessarily either a projection or a near-unanimity operation, and in the second case there is a majority term in the clone of $t$. Thus we just need to show that $t$ can't be a semiprojection.

Suppose for a contradiction that $t(x,y,y,...,y) \approx x$ but $t$ is not first projection. Since $t$ is not a basic operation of $\bA$ or a projection, we can write $t = g(t_1, t_2, t_3)$, with $t_1, t_2, t_3$ defined by shorter expressions than $t$. Since every two element subset of $\bA$ is a majority subalgebra, we have $t_i(x,y,y,...,y) \in \{x,y\}$ for $i = 1,2,3$, so since $g$ is a majority operation at least two of the $t_i$s have $t_i(x,y,y,...,y) \approx x$, and we may suppose without loss of generality that these are $t_1, t_2$. By induction, we see that $t_1,t_2$ are first projection, so $t \approx g(\pi_1,\pi_1,t_3) \approx \pi_1$, and we see that $t$ is first projection as well.
\end{proof}

\begin{thm}\label{spiral-ternary} If $\bA$ is a minimal spiral, then for any nontrivial binary terms $p(x,y), q(x,y), r(x,y)$, $\bA$ has a term $w$ such that
\[
w(x,x,y) \approx p(x,y),\;\;\; w(x,y,x) \approx q(x,y),\;\;\; w(y,x,x) \approx r(x,y).
\]
\end{thm}
\begin{proof} Let $\bF = \cF_{\bA}(x,y)$ be the free algebra on two generators in the variety generated by $\bA$. By Theorem \ref{subalg}, we see that $\bF$ is also minimal. By Theorem \ref{recursive}, we can write $\bF = S\cup\{x,y\}$, where $S$ is the maximal strongly connected component of $\bF$, and the partition $\{S,\{x\},\{y\}\}$ of $\bF$ defines a congruence such that the quotient is the free semilattice on two generators. In particular, we must have $f(x,y), p(x,y), q(x,y), r(x,y) \in S$. Let $\RR = \Sg_{\bF^3}\{(x,x,y),(x,y,x),(y,x,x)\}$, then by Lemma \ref{maj-triple} and the fact that $\RR\cap (S\times S\times S) \ne \emptyset$ we have $S\times S\times S \subseteq \RR$.
\end{proof}

Note that the definition of a spiral makes sense for infinite algebras as well.

\begin{thm} An arbitrary product of spirals is a (possibly infinite) spiral, and any subalgebra of a spiral is a spiral. If $\bA$ is a (possibly infinite) spiral and $\theta$ is a congruence of $\bA$ such that the intersection of every class of $\theta$ with every finitely generated subalgebra of $\bA$ is finite, then $\bA/\theta$ is also a spiral.
\end{thm}
\begin{proof} The first two claims follow directly from the definition of a spiral, so we only have to prove the claim about quotients. Let $x/\theta,y/\theta \in \bA/\theta$, and suppose without loss of generality that $y\in y/\theta$ is chosen such that $\Sg_{\bA}\{x,y\} \cap (y/\theta)$ is minimal. Suppose that $\Sg_{\bA/\theta}\{x/\theta,y/\theta\}$ has more than two elements, so that in particular $\Sg_{\bA}\{x,y\}$ also has more than two elements and thus maps to the free semilattice on two generators. We need to show that if $a/\theta, b/\theta \in \Sg_{\bA/\theta}\{x/\theta,y/\theta\}$ with $a/\theta \ne y/\theta$, then $f(a,b)/\theta \ne y/\theta$. We may assume without loss of generality that $a,b \in \Sg_{\bA}\{x,y\}$. Suppose for a contradiction that $f(a,b) \in y/\theta$, then by the choice of $y$ we have $y \in \Sg_{\bA}\{x,f(a,b)\}$. But this is a contradiction, since $f(a,b) \ne y$ and $(\Sg_{\bA}\{x,y\})\setminus\{y\}$ is a subalgebra of $\bA$.
\end{proof}

This shows that the collection of finite spirals forms a pseudovariety. The next result shows that it does not form a variety.

\begin{prop} For any odd $p$, the variety generated by the collection of finite spirals contains the affine algebra $(\ZZ/p,f)$ given by $f(a,b) = \frac{a+b}{2}$.
\end{prop}
\begin{proof} For each $n \ge 1$, we define a finite spiral $\bA_n = (\{1,...,n\},f)$ by
\[
f(x,y) = \begin{cases}x & x = y\\ \min(\frac{p+1}{2}(x+y),n) & x \ne y\end{cases}.
\]
To see that this is a spiral, note that for $x\ne y$ with $x,y < n$, we have $f(x,y) > x,y$, and for any $x \in \{1, ..., n\}$ we have $x \rightarrow n$. Inside the product of the $\bA_n$s, we have an isomorphic copy of the infinite spiral $\bA_{\NN} = (\NN,f)$ given by
\[
f(x,y) = \begin{cases}x & x = y\\ \frac{p+1}{2}(x+y) & x \ne y\end{cases}.
\]
Now we look mod $p$.
\end{proof}

Since the collection of finite spirals forms a pseudovariety, it should be defined by a sequence of identities such that all but finitely many of the identities hold in any finite spiral. The next result gives such a collection of identities.

\begin{thm} For any nontrivial term $p(x,y)$ built out of $f(x,y)$, let $k_p$ be the number of times $f$ occurs in the definition of $p$. Define from $p$ a sequence of terms $p_0(x,y) = y, p_1(x,y) = p(x,y),$ and $p_{i+1}(x,y) = p(x,p_i(x,y))$. Then in any finite spiral $\bA$, the identities
\[
f(x,p_k(x,y)) \approx p_k(x,y)
\]
with $p$ a binary term built out of $f$ and $k \ge |\bA|-1$ hold in $\bA$. Conversely, if $\bA = (A,f)$ with $f$ a commutative idempotent binary operation such that all but finitely many of the identities $f(x,p_k(x,y)) \approx p_k(x,y)$ with $p$ a nontrivial binary term built out of $f$ and $k \ge k_p$ hold, then $\bA$ is a spiral.
\end{thm}
\begin{proof} Let $\bA$ be a finite spiral, and let $a,b \in \bA$. By the Pigeonhole Principle, there are $0 \le i < j \le |\bA|$ with $p_i(a,b) = p_j(a,b)$. Since $p_i(a,b) = p_j(a,b) = p_{j-i}(a,p_i(a,b))$ and $p_{j-i}$ is nontrivial, by the definition of a spiral we must have $a \rightarrow p_i(a,b)$, so for all $k \ge i$ we have $p_k(a,b) = p_i(a,b)$ and $f(a,p_i(a,b)) = p_i(a,b)$.

Now suppose that $\bA = (A,f)$ with $f$ a commutative idempotent binary operation such that all but finitely many of the identities $f(x,p_k(x,y)) \approx p_k(x,y)$ with $p$ a nontrivial binary term built out of $f$ and $k \ge k_p$ hold. In order to show that $\bA$ is a spiral, we need to show that for any $a,b \in \bA$, if there exist $c,d \in \Sg_{\bA}\{a,b\}$ with $c \ne b$ and $f(c,d) = b$, then $a \rightarrow b$. Note that since $c \ne b$ there is a nontrivial binary term $p(x,y)$ with $f(c,d) = p(a,b)$. From $p(a,b) = b$, we see that for all $k \ge 0$ we have $p_k(a,b) = b$, so taking $k$ sufficiently large we see that $f(a,b) = f(a,p_k(a,b)) = p_k(a,b) = b$, so $a \rightarrow b$.
\end{proof}

\section{Bounded width algebras of size three}\label{s-three}

We will draw a doodle for each bounded width algebra $\bA$ of size three as follows. We start by drawing a vertex for each element of $\bA$, for $a,b \in \bA$ we draw a directed edge from $a$ to $b$ if $a \rightarrow b$, and we draw a solid undirected edge connecting $a$ to $b$ if $\{a,b\}$ is a majority algebra. If $\{a,b\}$ is not a subalgebra of $\bA$, then we draw a dashed line connecting $a$ to $b$, and we record the values of $f(a,b), f(b,a)$ next to the dashed line (if $f(a,b) = f(b,a)$, then we only write their common value once). Finally, if $\bA$ has underlying set $\{a,b,c\}$, we draw a dashed circle around the set of elements $d \in \bA$ such that
\[
(d,d) \in \Sg_{\bA^2}\{(a,b),(b,c),(c,a)\}.
\]

Throughout this section, we will also fix the following notation for maps $\alpha, \beta, \gamma:\bA^3 \rightarrow \bA^3$:
\begin{align*}
\alpha(x,y,z) &= (f(x,y),f(y,z),f(z,x)),\\
\beta(x,y,z) &= (f(x,z),f(y,x),f(z,y)),\\
\gamma(x,y,z) &= (g(x,y,z),g(y,z,x),g(z,x,y)).
\end{align*}
Note that if $\bA\in\cV_{mbw}$ has underlying set $\{a,b,c\}$, then $g$ is completely determined by $f$ and $\gamma(a,b,c), \gamma(a,c,b)$, and that $f$ is completely determined by $\alpha(a,b,c)$ and $\beta(a,b,c)$. Figure \ref{doodles} is a summary of the main classification result proved in this section.

\begin{figure}
\begin{tabular}{p{1.5cm} | p{2.5cm} | p{3.2cm} | p{1.7cm} | p{2cm}}
Doodle & $f$ & $\gamma$ & Aut & Quotients \\ \hline
\begin{minipage}{\linewidth} \includegraphics[]{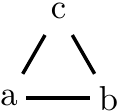} \end{minipage} & \begin{tabular}{c | c c c} $f$ & a & b & c\\ \hline a & a & a & a \\ b & b & b & b \\ c & c & c & c\end{tabular} & \vspace*{-15pt} $\gamma(a,b,c) = (a,b,c)$ $\gamma(b,a,c) = (b,a,c)$ & $S_3$ & Simple \\ \hline
\begin{minipage}{\linewidth} \includegraphics[]{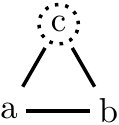} \end{minipage} & \begin{tabular}{c | c c c} $f$ & a & b & c\\ \hline a & a & a & a \\ b & b & b & b \\ c & c & c & c\end{tabular} & \vspace*{-15pt} $\gamma(a,b,c) = (c,c,c)$ $\gamma(b,a,c) = (c,c,c)$ & $\{1, (a\ b)\}$ & \vspace*{-15pt} $\{a\} \textbf{---} \{b,c\}$ $\{b\} \textbf{---} \{a,c\}$ \\ \hline
\begin{minipage}{\linewidth} \includegraphics[]{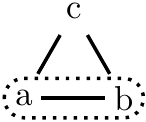} \end{minipage} & \begin{tabular}{c | c c c} $f$ & a & b & c\\ \hline a & a & a & a \\ b & b & b & b \\ c & c & c & c\end{tabular} & \vspace*{-15pt} $\gamma(a,b,c) = (a,a,a)$ $\gamma(b,a,c) = (b,b,b)$ & $\{1, (a\ b)\}$ & $\{c\} \textbf{---} \{a,b\}$ \\ \hline
\begin{minipage}{\linewidth} \includegraphics[]{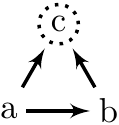} \end{minipage} & \begin{tabular}{c | c c c} $f$ & a & b & c\\ \hline a & a & b & c \\ b & b & b & c \\ c & c & c & c\end{tabular} & \vspace*{-15pt} $\gamma(a,b,c) = (c,c,c)$ $\gamma(b,a,c) = (c,c,c)$ & $1$ & \vspace*{-15pt} $\{a\}{\rightarrow}\{b,c\}$ $\{a,b\}{\rightarrow}\{c\}$ \\ \hline
\begin{minipage}{\linewidth} \includegraphics[]{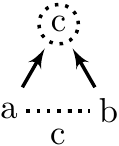} \end{minipage} & \begin{tabular}{c | c c c} $f$ & a & b & c\\ \hline a & a & c & c \\ b & c & b & c \\ c & c & c & c\end{tabular} & \vspace*{-15pt} $\gamma(a,b,c) = (c,c,c)$ $\gamma(b,a,c) = (c,c,c)$ & $\{1, (a\ b)\}$ &  \vspace*{-20pt} $\{a\}{\rightarrow}\{b,c\}$ $\{b\}{\rightarrow}\{a,c\}$ $\{a,b\}{\rightarrow}\{c\}$ \\ \hline
\begin{minipage}{\linewidth} \includegraphics[]{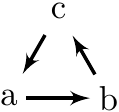} \end{minipage} & \begin{tabular}{c | c c c} $f$ & a & b & c\\ \hline a & a & b & a \\ b & b & b & c \\ c & a & c & c\end{tabular} & \vspace*{-15pt} $\gamma(a,b,c) = (a,b,c)$ $\gamma(b,a,c) = (b,a,c)$ & $A_3$ & Simple \\ \hline
\begin{minipage}{\linewidth} \includegraphics[]{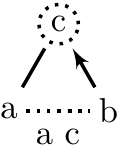} \end{minipage} & \begin{tabular}{c | c c c} $f$ & a & b & c\\ \hline a & a & a & a \\ b & c & b & c \\ c & c & c & c\end{tabular} & \vspace*{-15pt} $\gamma(a,b,c) = (c,c,c)$ $\gamma(b,a,c) = (c,c,c)$ & $1$ & \vspace*{-15pt} $\{a\}\textbf{---}\{b,c\}$ $\{b\}{\rightarrow}\{a,c\}$ \\ \hline
\begin{minipage}{\linewidth} \includegraphics[]{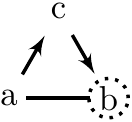} \end{minipage} & \begin{tabular}{c | c c c} $f$ & a & b & c\\ \hline a & a & a & c \\ b & b & b & b \\ c & c & b & c\end{tabular} & \vspace*{-15pt} $\gamma(a,b,c) = (b,b,b)$ $\gamma(b,a,c) = (b,b,b)$ & $1$ & Simple \\ \hline
\begin{minipage}{\linewidth} \includegraphics[]{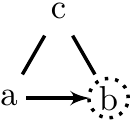} \end{minipage} & \begin{tabular}{c | c c c} $f$ & a & b & c\\ \hline a & a & b & a \\ b & b & b & b \\ c & c & c & c\end{tabular} & \vspace*{-15pt} $\gamma(a,b,c) = (b,b,b)$ $\gamma(b,a,c) = (b,b,b)$ & $1$ & $\{c\}\textbf{---}\{a,b\}$ \\ \hline
\begin{minipage}{\linewidth} \includegraphics[]{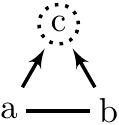} \end{minipage} & \begin{tabular}{c | c c c} $f$ & a & b & c\\ \hline a & a & a & c \\ b & b & b & c \\ c & c & c & c\end{tabular} & \vspace*{-15pt} $\gamma(a,b,c) = (c,c,c)$ $\gamma(b,a,c) = (c,c,c)$ & $\{1, (a\ b)\}$ & $\{a,b\}{\rightarrow}\{c\}$ \\ \hline
\begin{minipage}{\linewidth} \includegraphics[]{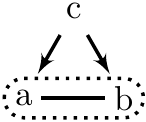} \end{minipage} & \begin{tabular}{c | c c c} $f$ & a & b & c\\ \hline a & a & a & a \\ b & b & b & b \\ c & a & b & c\end{tabular} & \vspace*{-15pt} $\gamma(a,b,c) = (a,a,a)$ $\gamma(b,a,c) = (b,b,b)$ & $\{1, (a\ b)\}$ & $\{c\}{\rightarrow}\{a,b\}$ \\ \hline
\end{tabular}

\caption{The eleven minimal bounded width algebras of size three, up to isomorphism and term equivalence.}\label{doodles}
\end{figure}

\begin{lem}\label{cycle} If $\cV$ is a locally finite idempotent variety of bounded width, then $\cV$ has terms $f,g$ as in Definition \ref{mindef} such that for every three element algebra $\bA$ in $\cV$ we have
\[
f^{\bA}(g^{\bA},\tilde{g}^{\bA}) = g^{\bA}.
\]
Furthermore, either $g^{\bA}$ is cyclic, or $A_3 \subseteq \Aut(g^{\bA})$ and $g^{\bA}(a,b,c) = a$ whenever $a,b,c$ are distinct elements of $\bA$.
\end{lem}
\begin{proof} Since $\cV$ is locally finite, it has finitely many terms of arity $3$. Choose a term $g$ as in Definition \ref{mindef} such that a maximal number of triples $(a,b,c)$ in the finitely many isomorphism classes of three element algebras have $g(a,b,c) = g(b,c,a)$, $g(a,c,b) = g(b,a,c)$ and $f(g(a,b,c),g(a,c,b)) = g(a,b,c)$, $f(g(a,c,b),g(a,b,c)) = g(a,c,b)$.

Among such terms, choose a $g$ such that the image of the map $\gamma: (a,b,c) \mapsto (g(a,b,c),g(b,c,a),g(c,a,b))$ is minimal. Note that if we have $g(a,b,c) = g(b,c,a)$ and $g(a,c,b) = g(b,a,c)$ then we must also have $f(g(a,b,c),g(a,c,b)) = g(a,b,c)$, $f(g(a,c,b),g(a,b,c)) = g(a,c,b)$, since otherwise we may replace $g$ by $g'(x,y,z) = f(g(x,y,z),g(x,z,y))$.

Now suppose $t\in\Clo(g)$ is any term, and $a,b,c$ are distinct elements of a three element algebra, such that $|\{t(a,b,c), t(b,c,a), t(c,a,b)\}| \le 2$. Setting
\[
g'(x,y,z) = t(g(x,y,z),g(y,z,x),g(z,x,y))
\]
and replacing $g$ with
\[
g''(x,y,z) = g'(g'(x,y,z),g'(y,z,x),g'(z,x,y))
\]
shows that in order for the image of the map $\gamma$ to be minimal, we must have $(a,b,c)$ not in the image of $\gamma$.

In particular, if $a,b,c$ are distinct elements of a three element algebra then we may not have $\gamma(a,b,c) = (a,c,b)$, as otherwise $\gamma(a,c,b) = (a,c,b)$ and one of $t = f(g,\tilde{g})$ or $t = g(g(x,y,z),x,y)$ contradicts the above (take $t = f(g,\tilde{g})$ if $\{b,c\}$ is not a majority subalgebra, take $t = g(g(x,y,z),x,y)$ if $\{b,c\}$ is a majority subalgebra).

Defining maps $\alpha, \beta$ by $\alpha(a,b,c) = (f(a,b),f(b,c),f(c,a))$ and $\beta(a,b,c) = (f(a,c),f(b,a),f(c,b))$, we see that if $a,b,c$ are distinct elements of a three element algebra and $\gamma(a,b,c) = (a,b,c)$, then we must also have $\gamma(a,c,b) = (a,c,b)$ and $\alpha(a,b,c),\beta(a,b,c) \in \{(a,b,c),(b,c,a),(c,a,b)\}$, from which we see that $g$ commutes with cyclic permutations of $\{a,b,c\}$.
\end{proof}

\begin{lem}\label{f-arrow} If $\bA \in \cV_{mbw}$ and $a,b \in \bA$ have $f(a,b) = b$, then $a \rightarrow b$.
\end{lem}
\begin{proof} Let $t(x,y) = f(x,f(x,y))$. Then
\[
t(a,b) = f(a,f(a,b)) = f(a,b) = b
\]
and
\[
t(b,a) = f(b,f(b,a)) = f(f(a,b),f(b,a)) = f(a,b) = b,
\]
so by Lemma \ref{prepare} we have $a \rightarrow b$.
\end{proof}

\begin{lem}\label{ac-impossible} If $\bA \in \cV_{mbw}$ and $a,b,c \in \bA$ have $f(a,b) = a, f(b,a) = c,$ and at least one of $f(b,c), f(f(b,c),c), f(f(f(b,c),c),c), ...$ is equal to $b$, then $\{a,b\}$ is a majority algebra.
\end{lem}
\begin{proof} By Theorem \ref{subalg}, we just have to find a ternary function $g' \in \Clo(g)$ which acts as the majority operation on $\{a,b\}$. Define $f^i(x,y)$ by $f^0(x,y) = x, f^1(x,y) = f(x,y)$, $f^{i+1}(x,y) = f(f^i(x,y),y)$, and choose $k \ge 1$ such that $f^{k+2}(b,c) = b$. Define $h^i(x,y,z)$ by
\[
h^1(x,y,z) = g(f(x,y),f(y,x),g(x,y,z))
\]
and
\[
h^{i+1}(x,y,z) = f(h^i(x,y,z),g(x,y,z))
\]
for $i \ge 1$.
Then by induction on $i$, we see that
\[
h^i(x,x,y) \approx f^i(x,f(x,y))
\]
and
\[
h^i(x,y,x) \approx h^i(y,x,x) \approx f(x,y).
\]
Upon setting
\[
i(x,y,z) = g(h^k(x,y,z),h^k(z,x,y),f^k(x,g(x,y,z))),
\]
we see that
\begin{align*}
i(x,x,y) &\approx g(h^k(x,x,y),h^k(y,x,x),f^k(x,f(x,y)))\\
&\approx g(f^k(x,f(x,y)),f(x,y),f^k(x,f(x,y)))\\
&\approx f^{k+1}(x,f(x,y)),\\
i(x,y,x) &\approx g(h^k(x,y,x),h^k(x,x,y),f^k(x,f(x,y)))\\
&\approx g(f(x,y),f^k(x,f(x,y)),f^k(x,f(x,y)))\\
&\approx f^{k+1}(x,f(x,y)),\\
i(y,x,x) &\approx g(h^k(y,x,x),h^k(x,y,x),f^k(y,f(x,y)))\\
&\approx g(f(x,y),f(x,y),f^k(y,f(x,y)))\\
&\approx f(f(x,y),f^k(y,f(x,y))).
\end{align*}
Now we set
\[
g'(x,y,z) = g(i(x,y,z),i(y,z,x),i(z,x,y))
\]
and
\[
f'(x,y) = f(f^{k+1}(x,f(x,y)),f(f(x,y),f^k(y,f(x,y))),
\]
so that
\[
g'(x,x,y) \approx g'(x,y,x) \approx g'(y,x,x) \approx f'(x,y).
\]
We just need to check that $f'(a,b) = a$ and $f'(b,a) = b$. We have
\begin{align*}
f'(a,b) &= f(f^{k+1}(a,f(a,b)),f(f(a,b),f^k(b,f(a,b)))\\
&= f(f^{k+1}(a,a),f(a,f^k(b,a)))\\
&= f(a,f(a,f^{k-1}(f(b,a),a)))\\
&= f(a,f(a,f^{k-1}(c,a))) = a
\end{align*}
since $\{a,c\} = \{f(a,b),f(b,a)\}$ is a majority algebra, and
\begin{align*}
f'(b,a) &= f(f^{k+1}(b,f(b,a)),f(f(b,a),f^k(a,f(b,a)))\\
&= f(f^{k+1}(b,c),f(c,f^k(a,c)))\\
&= f(f^{k+1}(b,c),c)\\
&= f^{k+2}(b,c) = b.\qedhere
\end{align*}
\end{proof}

\begin{lem}\label{subdirect-crit} If $\bA \in \cV_{mbw}$ and $a,b,c$ are distinct elements of $\bA$ such that $\{a,c\}$ is a majority subalgebra, $b \rightarrow c$, and $(a,c) \in \Sg_{\bA^2}\{(a,b),(b,a)\}$, then $\{a,b,c\}$ is a subalgebra of $\bA$ which is isomorphic to the subdirect product of a two element majority algebra and a two element semilattice.
\end{lem}
\begin{proof} Since $(a,c) \in \Sg_{\bA^2}\{(a,b),(b,a)\}$, there is a term $t$ such that $t(a,b) = a$ and $t(b,a) = c$. Let $f'(x,y) = t(t(x,y),x)$. Since $t$ either acts as first projection or second projection on $\{a,c\}$, $f'$ acts as the first projection on $\{a,c\}$, that is, $f'(a,c) = a$ and $f'(c,a) = c$. Furthermore, we have
\[
f'(a,b) = t(t(a,b),a) = t(a,a) = a
\]
and
\[
f'(b,a) = t(t(b,a),b) = t(c,b) = c
\]
since $b\rightarrow c$. Thus $\{a,b,c\}$ is closed under $f'$, and the restriction of $f'$ to $\{a,b,c\}$ agrees with the $f$ for the unique subdirect product of a two element majority algebra and a two element semilattice with size three. We define $\alpha':\bA^3 \rightarrow \bA^3$ by
\[
\alpha'(x,y,z) = (f'(x,y),f'(y,z),f'(z,x)),
\]
and define $g'$ by
\[
g'(x,y,z) = g(\alpha'(\alpha'(x,y,z))).
\]
We have
\begin{align*}
\alpha'(a,b,c) &= (a,c,c),\\
\alpha'(a,c,b) &= (a,c,c),\\
\alpha'(a,b,b) &= (a,b,c),\\
\alpha'(a,a,b) &= (a,a,c),
\end{align*}
so for any $x,y,z \in \{a,b,c\}$, $\alpha'(\alpha'(x,y,z))$ is either diagonal or a cyclic permutation of $(a,a,c)$ or $(a,c,c)$. The value of $g$ is known on any such triple (since $g$ is idempotent and $\{a,c\}$ is a majority algebra), and one easily checks that the restriction of $g'$ to $\{a,b,c\}$ agrees with the $g$ for the subdirect product of a two element majority algebra and a two element semilattice with size three. Now we apply Theorem \ref{subalg}.
\end{proof}

\begin{lem}\label{free-semi-crit} If $\bA \in \cV_{mbw}$ and $a,b,c$ are distinct elements of $\bA$ such that $\{a,c\}$ and $\{b,c\}$ are subalgebras of $\bA$ and $(c,c) \in \Sg_{\bA^2}\{(a,b),(b,a)\}$, then $\{a,b,c\}$ is a subalgebra of $\bA$ which is isomorphic to the free semilattice on the two generators $a,b$.
\end{lem}
\begin{proof} Suppose first, for a contradiction, that $\{a,c\}$ is a majority algebra. If $b \rightarrow c$, then
\[
\begin{bmatrix} a\\ c\end{bmatrix} = f\left(\begin{bmatrix} a\\ b\end{bmatrix}, \begin{bmatrix} c\\ c\end{bmatrix}\right) \in \Sg_{\bA^2}\left\{\begin{bmatrix} a\\ b\end{bmatrix}, \begin{bmatrix} b\\ a\end{bmatrix}\right\},
\]
so by Lemma \ref{subdirect-crit} we see that $\{a,b,c\}$ is isomorphic to the subdirect product of a two element majority algebra and a two element semilattice, but then $(c,c) \not\in \Sg_{\bA^2}\{(a,b),(b,a)\} = \{(a,b),(b,a),(a,c),(c,a)\}$, a contradiction.

If $c \rightarrow b$, then
\[
\begin{bmatrix} b\\ c\end{bmatrix} = f\left(\begin{bmatrix} c\\ c\end{bmatrix}, \begin{bmatrix} b\\ a\end{bmatrix}\right) \in \Sg_{\bA^2}\left\{\begin{bmatrix} a\\ b\end{bmatrix}, \begin{bmatrix} b\\ a\end{bmatrix}\right\},
\]
and similarly $(c,b) \in \Sg_{\bA^2}\{(a,b),(b,a)\}$, so
\[
\begin{bmatrix} b\\ b\end{bmatrix} = f\left(\begin{bmatrix} b\\ c\end{bmatrix}, \begin{bmatrix} c\\ b\end{bmatrix}\right) \in \Sg_{\bA^2}\left\{\begin{bmatrix} a\\ b\end{bmatrix}, \begin{bmatrix} b\\ a\end{bmatrix}\right\},
\]
so by Lemma \ref{prepare} $a \rightarrow b$, but then $(c,c) \not\in \Sg_{\bA^2}\{(a,b),(b,a)\} = \{(a,b),(b,a),(b,b)\}$, a contradiction.

Thus if $\{a,c\}$ is a majority algebra, then we must also have $\{b,c\}$ a majority algebra. From $(c,c) \in \Sg_{\bA^2}\{(a,b),(b,a)\}$, there must be some binary term $t$ such that $t(a,b) = t(b,a) = c$, and we may assume without loss of generality that $t$ acts as first projection on any majority subalgebra of $\bA$. Define terms $h,g'$ by
\[
h(x,y,z) = g(x,t(x,y),t(x,z))
\]
and
\[
g'(x,y,z) = g(h(x,y,z),h(y,z,x),h(z,x,y)).
\]
Then
\begin{align*}
h(a,a,b) &= g(a,a,c) = a,\\
h(a,b,a) &= g(a,c,a) = a,\\
h(b,a,a) &= g(b,c,c) = c,
\end{align*}
so
\[
g'(a,a,b) = g(a,a,c) = a,
\]
and similarly $g'(a,b,a) = g'(b,a,a) = a$. By interchanging the roles of $a$ and $b$ in this argument, we see that $g'$ acts as the majority operation on $\{a,b\}$, so by Theorem \ref{subalg} $\{a,b\}$ is a majority algebra, so $(c,c)\not\in \Sg_{\bA^2}\{(a,b),(b,a)\} = \{(a,b),(b,a)\}$, a contradiction.

We have shown that neither of $\{a,c\}, \{b,c\}$ can be a majority algebra, so they must both be semilattice subalgebras. If $a \rightarrow c \rightarrow b$, then $(c,b) = f((a,b),(c,c)) \in \Sg_{\bA^2}\{(a,b),(b,a)\}$, so $(b,b) = f((b,c),(c,b)) \in \Sg_{\bA^2}\{(a,b),(b,a)\}$, so $a\rightarrow b$, a contradiction. Similarly we can't have $b \rightarrow c \rightarrow a$.

Let $t$ be a binary term such that $t(a,b) = t(b,a) = c$. Suppose that $c \rightarrow a$, $c \rightarrow b$, and define $g'$ by
\[
g'(x,y,z) = g(t(x,y),t(y,z),t(z,x)).
\]
Then
\[
g'(a,a,b) = g(a,c,c) = a,
\]
and similarly $g'(a,b,a) = g'(b,a,a) = a$ as well as $g'(a,b,b) = g'(b,a,b) = g'(b,b,a) = b$, so by Theorem \ref{subalg} $\{a,b\}$ is a majority algebra, a contradiction.

Thus we must have $a \rightarrow c$ and $b \rightarrow c$, so $(\{a,b,c\},t)$ is the free semilattice on the generators $a,b$, and we are done by Theorem \ref{subalg}.
\end{proof}

\begin{thm}\label{f-closed} If $\bA \in \cV_{mbw}$, $\{a,b,c\} \subseteq \bA$ is closed under $f$, and not all three of $\{a,b\}, \{b,c\}, \{a,c\}$ are majority subalgebras of $\bA$, then $\{a,b,c\}$ is a subalgebra of $\bA$ and the restriction of $g$ to $\{a,b,c\}$ is determined by the restriction of $f$ to $\{a,b,c\}$.
\end{thm}
\begin{proof} First we show that the fact that $g$ is determined by $f$ follows from the fact that $\{a,b,c\}$ is necessarily a subalgebra. If there were two different algebras $(\{a,b,c\},f,g)$ and $(\{a,b,c\},f,g')$ in $\cV_{mbw}$, then the subset $\{(a,a),(b,b),(c,c)\}$ of their product would be closed under $f$ and not all three of $\{(a,a),(b,b)\}, \{(b,b),(c,c)\}, \{(a,a),(c,c)\}$ would be majority subalgebras, so it would be a subalgebra of the product. Thus it would be the graph of an isomorphism, and we would have to have $g = g'$.

By Theorem \ref{connect}, at least two of $\{a,b\}, \{b,c\}, \{a,c\}$ are subalgebras of $\bA$. If all three are semilattice subalgebras, then $(\{a,b,c\},f)$ is a 2-semilattice and we are done by Theorem \ref{subalg}.

If all three are subalgebras but not all three have the same type, then we may assume without loss of generality that $a\rightarrow b$ and that $f(b,c) = b$. Then since $\alpha(a,b,c) = (b,b,f(c,a))$ has two of its coordinates equal and $\alpha(a,c,b) = (f(a,c),f(c,b),b)$ either has two of its coordinates equal or is equal to $(a,c,b)$, we see that $\beta(\alpha(x,y,z))$ always has two of its coordinates equal for any $x,y,z \in \{a,b,c\}$. Thus, if we set $g'(x,y,z) = g(\beta(\alpha(x,y,z)))$, then for $x,y \in \{a,b,c\}$ we have
\[
g'(x,x,y) = g'(x,y,x) = g'(y,x,x) = f(x,y)
\]
and $\{a,b,c\}$ is closed under $g'$, so we are done by Theorem \ref{subalg}.

Now suppose that $\{a,b\}$ is not a subalgebra of $\bA$, in which case $\{a,c\}, \{b,c\}$ must be subalgebras by Theorem \ref{connect}. If $f(a,b) = f(b,a) = c$, then by Lemma \ref{free-semi-crit} $\{a,b,c\}$ is a subalgebra of $\bA$ which is isomorphic to a free semilattice on the generators $a,b$.

If $f(a,b) \ne f(b,a)$, then one of them must be $c$ and we may assume without loss of generality that the other one is $a$, so $\{a,c\}$ is a majority subalgebra of $\bA$. If $f(b,a) = a$, then by Theorem \ref{f-arrow} we have $b\rightarrow a$, contradicting the assumption that $\{a,b\}$ was not a subalgebra. Thus we must have $f(a,b) = a, f(b,a) = c$. If $f(b,c) = b$, then by Lemma \ref{ac-impossible} we have $\{a,b\}$ a majority algebra, which is again a contradiction. Since $\{b,c\}$ is a subalgebra and $f(b,c) \ne b$ we must have $b \rightarrow c$, and by Lemma \ref{subdirect-crit} $\{a,b,c\}$ is a subalgebra of $\bA$ which is isomorphic to a subdirect product of a two element majority algebra and a two element semilattice.
\end{proof}

\begin{thm}\label{class-3} We may choose the operation $g$ of $\cV_{mbw}$ such that every three element algebra of $\cV_{mbw}$ is isomorphic to one of the eleven algebras shown in Figure \ref{doodles}.
\end{thm}
\begin{proof} Let $\bA_1, ..., \bA_{11}$ be the algebras given in the corresponding rows of Figure \ref{doodles}. We start by choosing $f,g,f_3$ as in Corollary \ref{min}, and then we modify $g$ to make it satisfy the conclusion of Lemma \ref{cycle} (note that in doing so, we do not change the common value of $g(x,x,y) \approx g(x,y,x) \approx g(y,x,x)$). First we will show that on every three element algebra $\bA = (A,f,g)$ either $\bA$ or $\tilde{\bA} = (A,f,\tilde{g})$ is isomorphic with one of $\bA_1, ..., \bA_{11}$.

If $\Aut(\bA)$ contains a three-cycle, then $\gamma(a,b,c) = (a,b,c), \gamma(a,c,b) = (a,c,b)$, and every two element subset of $\bA$ must be a subalgebra, so $\bA$ is either the three element dual discriminator algebra $\bA_1$ or the strongly connected 2-semilattice $\bA_6$. If $\bA$ is not conservative, then by the proof of Theorem \ref{f-closed} $\bA$ is either isomorphic to $\bA_5$, the free semilattice on two generators, or $\bA_7$, the subdirect product of a two element majority algebra and a two element semilattice.

If $\bA$ is conservative and is not a majority algebra or the strongly connected 2-semilattice, then by the proof of Theorem \ref{f-closed} the restriction of $g$ to $\bA$ is in the clone of the restriction of $g(\beta(\alpha(x,y,z)))$ to $\bA$, and the reduct corresponding to $g(\beta(\alpha(x,y,z)))$ is isomorphic to one of $\bA_4, \bA_8, \bA_9, \bA_{10}, \bA_{11}$. We need to show that in this case, either $g$ or $\tilde{g}$ agrees with $g(\beta(\alpha(x,y,z)))$. To see this, note that by Lemma \ref{cycle} the restriction of $g$ to $\bA$ is cyclic, so we must have $\gamma(a,b,c),\gamma(a,c,b)$ in the intersection of the diagonal of $\bA^2$ with $\Sg_{\bA^2}\{(a,b),(b,c),(c,a)\}$, that is, in the circled set of vertices of the corresponding doodle. By examining the possibilities, we see that in each case the circled set of vertices either has size $1$, in which case $g, \tilde{g},$ and $g(\beta(\alpha(x,y,z)))$ all agree on $\bA$, or else there is an automorphism $\iota$ of $\bA$ with order two which interchanges the two circled vertices. Then since $\gamma$ is cyclic we must have
\[
\gamma(a,c,b) = \iota(\gamma(\iota(a),\iota(c),\iota(b))) = \iota(\gamma(a,b,c)),
\]
so $g$ and $\tilde{g}$ do not agree with each other on $\bA$, and exactly one of them gives a reduct isomorphic with $\bA_{11}$.

Finally we have the case where $\bA$ is a majority algebra with $g$ cyclic. There are two cases: either $\gamma(a,b,c) = \gamma(a,c,b)$, or $\gamma(a,b,c) \ne \gamma(a,c,b)$. In the first case, $\bA$ is the median algebra $\bA_2$. In the second case, letting $\{\gamma(a,b,c),\gamma(a,c,b)\} = \{a,b\}$, we see that either $\bA$ or $\tilde{\bA}$ is isomorphic to $\bA_3$.

Next we show that when $\bA$ is a minimal bounded width algebra of size three with $\bA$ not isomorphic to $\tilde{\bA}$, then $\bA^2$ is term equivalent to a proper reduct of $\bA\times \tilde{\bA}$, so that by Theorem \ref{subalg} at most one of $\bA, \tilde{\bA}$ can be an algebra of the pseudovariety $\cV_{mbw}$. If $\bA$ or $\tilde{\bA}$ is isomorphic to $\bA_{11}$ then this is easy: the reduct of $\bA_{11}\times \tilde{\bA}_{11}$ given by $g'(x,y,z) = g(\alpha(x,y,z))$ is term equivalent with $\bA_{11}^2$. By replacing $g$ with $\tilde{g}$ if necessary, we may assume without loss of generality that $\bA_{11} \in \cV_{mbw}$ and $\tilde{\bA}_{11} \not\in \cV_{mbw}$.

One may easily check that in any algebra $\bA$, if $a$ is such that $f(a,b) = a$ for all $b \in \bA$ then we must have $f_3(a,b,c) = a$ for any $b,c \in \bA$. In particular, $f_3$ agrees with first projection in any majority algebra. In $\bA_{11}$, this gives us
\begin{align*}
f_3^{\bA_{11}}(a,b,c) &= a,\\
f_3^{\bA_{11}}(b,c,a) &= b,
\end{align*}
and we have
\[
f_3^{\bA_{11}}(c,a,b) \in \{a,b\},
\]
since $\bA_{11}$ has a quotient isomorphic to a semilattice directed from $c$ to $\{a,b\}$. By possibly replacing $f_3(x,y,z)$ with $f_3(x,z,y)$, we may assume without loss of generality that $f_3^{\bA_{11}}(c,a,b) = a$.

In order to prove that $\bA_3^2$ is term equivalent to a reduct of $\bA_3\times \tilde{\bA}_3$, we define $h$ by
\[
h(x,y,z) = g(f_3(x,y,z),g(x,y,z),g(x,z,y))
\]
and then define $g'$ by
\[
g'(x,y,z) = g(h(x,y,z),h(y,z,x),h(z,x,y)).
\]
Then
\[
g'^{\bA_3}(a,b,c) = g^{\bA_3}(g^{\bA_3}(a,a,b),g^{\bA_3}(b,a,b),g^{\bA_3}(c,a,b)) = g^{\bA_3}(a,b,a) = a
\]
and
\[
g'^{\tilde{\bA}_3}(a,b,c) = g^{\tilde{\bA}_3}(g^{\tilde{\bA}_3}(a,b,a),g^{\tilde{\bA}_3}(b,b,a),g^{\tilde{\bA}_3}(c,b,a)) = g^{\tilde{\bA}_3}(a,b,a) = a,
\]
and from the definition of $g'$ it is clear that $g'^{\bA}$ is cyclic whenever $g^{\bA}$ is cyclic, so this gives us the desired reduct. Additionally, we have
\[
g'^{\bA_{11}}(a,b,c) = g^{\bA_{11}}(g^{\bA_{11}}(a,a,b),g^{\bA_{11}}(b,a,b),g^{\bA_{11}}(a,a,b)) = g^{\bA_{11}}(a,b,a) = a,
\]
so $g'$ agrees with $g$ on $\bA_{11}$.

Moreover, since
\[
h(x,x,y) \approx h(x,y,x) \approx g(f(x,y),f(x,y),f(x,y)) \approx f(x,y)
\]
and
\[
h(y,x,x) \approx g(f(y,x),f(x,y),f(x,y)) \approx f(x,y),
\]
the function $g'$ given above has the additional property that
\[
g'(x,x,y) \approx g'(x,y,x) \approx g'(y,x,x) \approx f(x,y),
\]
so we may replace $g$ with $g'$ in $\cV_{mbw}$.
\end{proof}

\begin{cor} The number of minimal conservative elements of $\cG_k$, up to term equivalence, is
\[
\sum_{E\subseteq \binom{\{1, ..., k\}}{2}} 7^{\Delta(E)}2^{\binom{k}{2}-|E|} = 7^{\binom{k}{3}}(1+o(1)),
\]
where $\Delta(E)$ is the number of triangles in the graph $(\{1,...,k\},E)$.
\end{cor}

\begin{defn} We define the variety $\cV_{mcbw}$ to be the variety with one basic operation $g$, which is generated by idempotent conservative bounded width algebras $\bA = (A,g^{\bA})$ such that on every three element subset $S \subseteq A$, $(S,g^\bA\!\!\mid_S)$ is isomorphic to one of the nine conservative examples in Figure \ref{doodles}.
\end{defn}

The variety $\cV_{mcbw}$ is easily seen to be locally finite. One may compute by hand that $|\cF_{\cV_{mcbw}}(x,y)| = 4$, and by computer that $|\cF_{\cV_{mcbw}}(x,y,z)| = 2547$. A few of the two and three variable identities satisfied in $\cV_{mcbw}$ are as follows:
\begin{align*}
f(x,y) &\approx f(x,f(x,y))\\
&\approx f(x,f(y,x))\\
&\approx f(f(x,y),x)\\
&\approx f(f(x,y),y)\\
&\approx g(x,y,f(x,y))\\
&\approx g(x,f(x,y),f(y,x)),\\
g(x,y,z) &\approx g(x,f(y,z),z)\\
&\approx g(x,f(y,z),f(z,x)),\\
f(f(x,y),z) &\approx g(z,x,f(x,y)),\\
f(f(x,y),f(x,z)) &\approx g(x,f(x,y),f(x,z)).
\end{align*}

In the next section, we will give two examples of minimal elements of $\cG_4$ for which the identity $f(x,f(y,x)) \approx f(x,y)$ can not be satisfied. Note that our basic six element spiral is an example of a minimal element of $\cG_6$ for which none of the identities listed above which only involve $f$ can be satisfied.

\section{Examples generated by two elements}\label{s-two-gen}

We start by giving a classification of minimal bounded width algebras of size four which are generated by two elements. The classification is summarized in Figure \ref{four}.

\begin{lem} Let $\bA = (\{a,b,c,d\},f,g)$ be a minimal bounded width algebra with $f$ chosen as in Corollary \ref{reach}, and suppose that $\bA$ is generated by $a$ and $b$. Then $f(a,b),f(b,a) \in \{c,d\}$.
\end{lem}
\begin{proof} Suppose for contradiction that $f(a,b) \in \{a,b\}$. If $f(a,b) = b$ then $a \rightarrow b$, contradicting the assumption that $\bA$ is generated by $a,b$. Similarly, we have $f(b,a) \ne a$, so we may assume without loss of generality that $f(a,b) = a$ and $f(b,a) = c$. In this case, $\{a,c\}$ is a majority algebra, so $f(a,c) = a, f(c,a) = c$. By Theorem \ref{f-closed}, if $f(b,c)$ and $f(c,b)$ are both in $\{a,b,c\}$ then $\{a,b,c\}$ is a proper subalgebra of $\bA$ containing $a,b$, contradicting the assumption that $a,b$ generate $\bA$. Thus at least one of $f(b,c), f(c,b)$ must be $d$. In particular, $\{b,c\}$ is not a subalgebra of $\bA$.

By our choice of $f$, $c = f(b,a)$ is reachable from $b$. The only way that this is possible is if we have $b \rightarrow d \rightarrow c$. If $f(b,c) = f(c,b) = d$, then $\{b,c,d\}$ is closed under $f$, and in fact by Lemma \ref{free-semi-crit} this case is impossible. Thus $\{f(b,c),f(c,b)\}$ is a two element majority algebra which contains $d$, so it must be $\{a,d\}$. But then since $a$ is either $f(b,c)$ or $f(c,b)$, $a$ must be reachable from either $b$ or $c$. From $f(b,a) = f(c,a) = c$ and $f(d,a) = d$ we see that $a$ isn't reachable from any of $b,c,d$, a contradiction.
\end{proof}

\begin{lem} Let $\bA = (\{a,b,c,d\},f,g)$ be a minimal bounded width algebra with $f$ chosen as in Corollary \ref{reach}, and suppose that $\bA$ is generated by $a$ and $b$. Then $f(a,b) \ne f(b,a)$.
\end{lem}
\begin{proof} Suppose for contradiction that $f(a,b) = f(b,a) = c$. By our choice of $f$, $c$ must be reachable from both $a$ and $b$. If we have $a \rightarrow c$ and $b \rightarrow c$, then by Theorem \ref{f-closed} $\{a,b,c\}$ is a subalgebra of $\bA$, a contradiction.

Suppose now that neither of $a \rightarrow c$ nor $b \rightarrow c$ holds. Then in order for $c$ to be reachable from $a$ and $b$, we must have $a\rightarrow d$, $b \rightarrow d$, and $d \rightarrow c$. By Theorem \ref{f-closed}, at least one of $f(a,c),f(c,a), f(b,c),f(c,b)$ must be $d$, and we may suppose without loss of generality that either $f(a,c)$ or $f(c,a)$ is $d$. Then $\{f(a,c),f(c,a)\}$ is a majority subalgebra containing $d$, and since every two element algebra containing $d$ is a semilattice, we must have $f(a,c) = f(c,a) = d$. However, this together with $a \rightarrow d \rightarrow c$ contradicts Lemma \ref{free-semi-crit}.

Thus exactly one of $a \rightarrow c$, $b\rightarrow c$ holds, and we may assume without loss of generality that $a \rightarrow c$. In order for $c$ to be reachable from $b$, we must then have $b \rightarrow d$ and at least one of $d \rightarrow a, d\rightarrow c$. Also, by Theorem \ref{f-closed}, at least one of $f(b,c),f(c,b)$ must be $d$. If $f(b,c) = f(c,b) = d$, then $d$ is reachable from $c$, so $c \rightarrow d$ and $b\rightarrow d \rightarrow a \rightarrow c$, and we see in this case that $f$ is commutative, so $\bA$ is a minimal spiral and so by Theorem \ref{recursive} $\{c,d\}$ has to be strongly connected, which is impossible. Thus $\{f(b,c),f(c,b)\}$ must be a majority algebra containing $d$. If $\{f(b,c),f(c,b)\} = \{a,d\}$, then $a$ is reachable from one of $b,c$, which is impossible.

Thus $\{f(b,c),f(c,b)\} = \{c,d\}$, and by Lemma \ref{subdirect-crit} $\{b,c,d\}$ is a subalgebra of $\bA$ which is isomorphic to the subdirect product of a two element majority algebra and a two element semilattice. Also, since $c$ is reachable from $b$, we have $b \rightarrow d \rightarrow a \rightarrow c$. At this point we have completely determined $f$. In order to get a contradiction, we will show that $(a,a) \in \Sg_{\bA^2}\{(a,b),(b,a)\}$. We have $(c,c) = f((a,b),(b,a))$, $(c,d) = f((a,b),(c,c))$, $(d,a) = f((b,a),(c,d))$, and similarly $(a,d) \in \Sg_{\bA^2}\{(a,b),(b,a)\}$. But then $(a,a) = f((a,d),(d,a)) \in \Sg_{\bA^2}\{(a,b),(b,a)\}$, so $b \rightarrow a$, a contradiction.
\end{proof}

\begin{lem} Let $\bA = (\{a,b,c,d\},f,g)$ be a minimal bounded width algebra with $f$ chosen as in Corollary \ref{reach}, and suppose that $f(a,b) = c, f(b,a) = d$. Then $a \rightarrow c$ and $b \rightarrow d$.
\end{lem}
\begin{proof} Suppose for a contradiction that we do not have $a \rightarrow c$. Then in order for $c = f(a,b)$ to be reachable from $f$, we must have $a \rightarrow d \rightarrow b \rightarrow c$. In this case we do not have $b \rightarrow d$, so in order for $d = f(b,a)$ to be reachable from $b$ we must have $b \rightarrow c \rightarrow a \rightarrow d$. This completely pins down $f$. Note that by Theorem \ref{f-closed}, the restriction of $g$ to $\{a,c,d\}$ and the restriction of $g$ to $\{b,c,d\}$ are determined by $f$.

Defining $\alpha(x,y,z) = (f(x,y),f(y,z),f(z,x))$, we have
\begin{align*}
\alpha(a,b,c) &= (c,c,a),\\
\alpha(a,c,b) &= (a,c,d),\\
\alpha(a,a,b) &= (a,c,d),
\end{align*}
so $g'(x,y,z) = g(\alpha(x,y,z))$ is completely determined by $f$. By minimality of $\bA$, we see that $g \in \Clo(g')$, so in particular $\bA$ has $(a\ b)(c\ d)$ as an automorphism. Thus, we have
\[
\Sg_{\bA^2}\{(a,b),(b,a)\} = \{(a,b),(b,a),(c,d),(d,c)\},
\]
which is isomorphic to $\bA$. Since $\bA$ is strongly connected, this shows that $(a,b)$ is a maximal element of $\Sg_{\bA^2}\{(a,b),(b,a)\}$, so by Theorem \ref{maj-crit} $\{a,b\}$ is a majority subalgebra of $\bA$, contradicting the assumption that $f(a,b) = c, f(b,a) = d$.
\end{proof}

\begin{figure}
\begin{tabular}{p{1.5cm} | p{3.0cm} | p{3.2cm} | p{2.2cm} | p{2.1cm}}
Doodle & $f$ & $\gamma$ & Aut & Quotients of Size 2 \\ \hline
\begin{minipage}{\linewidth} \includegraphics[]{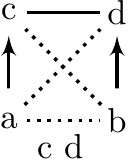} \end{minipage} & \begin{tabular}{c | c c c c} $f$ & a & b & c & d\\ \hline a & a & c & c & c \\ b & d & b & d & d \\ c & c & c & c & c \\ d & d & d & d & d \end{tabular} & \vspace*{-30pt} $\gamma(a,b,c) = (c,c,c)$ $\gamma(b,a,c) = (c,c,c)$ $\gamma(a,c,d) = (c,c,c)$ $\gamma(a,d,c) = (c,c,c)$ & $\{1, (a\ b)(c\ d)\}$ & \vspace*{-30pt} $\{a,b\} {\rightarrow} \{c,d\}$ $\{a\} {\rightarrow} \{b,c,d\}$ $\{b\} {\rightarrow} \{a,c,d\}$ $\{a,c\} \textbf{---} \{b,d\}$ \\ \hline
\begin{minipage}{\linewidth} \includegraphics[]{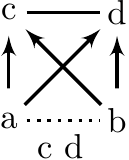} \end{minipage} & \begin{tabular}{c | c c c c} $f$ & a & b & c & d\\ \hline a & a & c & c & d \\ b & d & b & c & d \\ c & c & c & c & c \\ d & d & d & d & d \end{tabular} & \vspace*{-30pt} $\gamma(a,b,c) = (c,c,c)$ $\gamma(b,a,c) = (c,c,c)$ $\gamma(a,c,d) = (c,c,c)$ $\gamma(a,d,c) = (d,d,d)$ & $\{1, (a\ b)(c\ d)\}$ & \vspace*{-21pt} $\{a,b\} {\rightarrow} \{c,d\}$ $\{a\} {\rightarrow} \{b,c,d\}$ $\{b\} {\rightarrow} \{a,c,d\}$ \\ \hline
\begin{minipage}{\linewidth} \includegraphics[]{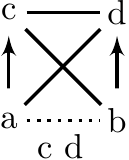} \end{minipage} & \begin{tabular}{c | c c c c} $f$ & a & b & c & d\\ \hline a & a & c & c & a \\ b & d & b & b & d \\ c & c & c & c & c \\ d & d & d & d & d \end{tabular} & \vspace*{-30pt} $\gamma(a,b,c) = (c,c,c)$ $\gamma(b,a,c) = (c,c,c)$ $\gamma(a,c,d) = (c,c,c)$ $\gamma(a,d,c) = (c,c,c)$ & $\{1, (a\ b)(c\ d)\}$ & $\{a,c\} \textbf{---} \{b,d\}$ \\ \hline
\end{tabular}

\caption{Minimal bounded width algebras of size four which are generated by two elements.}\label{four}
\end{figure}

\begin{thm}\label{class-4} Every minimal bounded width algebra $\bA$ of size four which is generated by two elements is term equivalent to an algebra isomorphic to one of the three examples in Figure \ref{four}, and conversely each algebra in Figure \ref{four} is a minimal bounded width algebra.
\end{thm}
\begin{proof} Write $\bA = (\{a,b,c,d\},f,g)$, with $f$ chosen as in Corollary \ref{reach}, and suppose $\bA$ is generated by $a$ and $b$. By the previous results, we may assume without loss of generality that $f(a,b) = c, f(b,a) = d$, and that $a \rightarrow c, b \rightarrow d$.

Suppose first that $a \in \{f(b,c),f(c,b)\}$. Then $\bA$ is generated by $b$ and $c$, so by the previous results we have $\{f(b,c),f(c,b)\} = \{a,d\}$ and either $b \rightarrow a, c\rightarrow d$ or $b\rightarrow d,c\rightarrow a$, and both possibilities are clearly impossible. Thus, $\{b,c,d\}$ is closed under $f$, so by Theorem \ref{f-closed} $\{b,c,d\}$ is a three element subalgebra of $\bA$. Similarly, $\{a,c,d\}$ is also a three element subalgebra of $\bA$.

Since $\{b,c,d\}$ is a three element subalgebra of $\bA$ with $\{c,d\}$ a majority algebra and $b \rightarrow d$, we either have $b \rightarrow c, c\rightarrow b$, $\{b,c\}$ a majority algebra, or $f(b,c) = d, f(c,b) = c$. Similarly, there are four possible ways to assign values to $f(a,d), f(d,a)$, so up to symmetry there are just $\frac{4^2 + 4}{2} = 10$ possible functions $f$ to consider.

We briefly summarize why seven of these cases are not minimal bounded width algebras. If both $c \rightarrow b$ and $d \rightarrow a$, then $\bA$ is strongly connected and has an automorphism interchanging $a$ and $b$, so by Theorem \ref{maj-crit} $\{a,b\}$ is a majority algebra. If $f(b,c) = d, f(c,b) = c$, and $d \rightarrow a$, then $(a,a) \in \Sg_{\bA^2}\{(a,b),(b,a)\}$, so $b \rightarrow a$. If $f(b,c) = d, f(c,b) = c$, and $\{a,d\}$ is a majority algebra, then $(a,d) \in \Sg_{\bA^2}\{(a,b),(b,a)\}$, so by Lemma \ref{subdirect-crit} $\{a,b,d\}$ is a subdirect product of a two element majority algebra and a two element semilattice. In the four remaining cases which don't work, the reader may check that we have either $(c,c) \in \Sg_{\bA}\{(a,b),(b,a)\}$ with $\{b,c\}$ a subalgebra, or the analogous statement with $d$ instead of $c$, and we can apply Lemma \ref{free-semi-crit} to see that either $\{a,b,c\}$ or $\{a,b,d\}$ is a free semilattice on the generators $a,b$.

Now we check that the three algebras displayed in Figure \ref{four} are minimal bounded width algebras. Call them $\bA_1, \bA_2, \bA_3$. $\bA_1$ is isomorphic to $\cF_{\cV_{mcbw}}(x,y)$, the free algebra on two generators in the variety generated by conservative bounded width algebras, so it is minimal. Due to the semilattice quotients of $\bA_2$, we see that for any nontrivial binary term $p(x,y)$, we must have $p(a,b), p(b,a) \in \{c,d\}$, and because $\bA_2$ has the automorphism $(a\ b)(c\ d)$, we have $p(a,b) \ne p(b,a)$, so either $p(x,y) \approx f(x,y)$ or $p(x,y) \approx f(y,x)$. From here we easily see that for any terms $f',g'$ of a bounded width reduct of $\bA_2$, we have $f' \approx f$ and $g \approx g'(\alpha(x,y,z))$.

We are left with checking that $\bA_3$ is minimal. Suppose there are terms $f', g'$ of a bounded width reduct of $\bA_3$, with $f'$ chosen as in Corollary \ref{reach}. Since $(a\ b)(c\ d)$ is an automorphism of $\bA_3$, we either have $f'(a,b) = a, f'(b,a) = b$ or $f'(a,b) = c, f'(b,a) = d$ (we can't have $f'(a,b) = d$ since $d$ is not reachable from $a$). If $f'(a,b) = c, f'(b,a) = d$, then $f' \approx f$ and $g \approx g'(\alpha(x,y,z))$, so we only have to prove that there is no term $g'$ which acts as a majority operation on $\{a,b\}$. In other words, we must show that
\[
(a,a,a) \not\in \Sg_{\bA_3^3}\{(a,a,b),(a,b,a),(b,a,a)\}.
\]
Let $S$ be the set of triples $(u,v,w)$ such that at least two of $u,v,w$ are in $\{a,c\}$, and such that if none of $u,v,w$ is $c$ then $(u,v,w)$ is a cyclic permutation of $(a,a,b)$. We claim that $S$ is closed under $g$ (in fact, it turns out that $S = \Sg_{\bA_3^3}\{(a,a,b),(a,b,a),(b,a,a)\}$). The key observation needed to check this claim is that if $x,y,z \in \bA_3$ have at least two of $x,y,z$ in $\{a,c\}$ and at least one of $x,y,z$ equal to $c$, then $g(x,y,z) = c$. We leave the details to the reader.
\end{proof}

We also give two examples of larger algebras which are generated by two elements in Figure \ref{bigger}. For brevity, we only describe the function $f$ and the subalgebras which are generated by pairs of elements.

\begin{figure}
\begin{tabular}{p{3.9cm} p{6.1cm} p{3.6cm}}
\begin{minipage}{\linewidth} \includegraphics[]{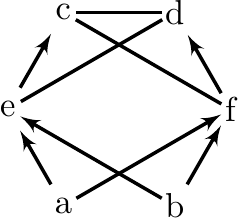} \end{minipage} & \begin{tabular}{c | c c c c c c} $f$ & a & b & c & d & e & f\\ \hline a & a & c & f & e & e & f\\ b & d & b & f & e & e & f\\ c & c & c & c & c & c & c\\ d & d & d & d & d & d & d\\ e & e & e & c & e & e & c\\ f & f & f & f & d & d & f \end{tabular} & \vspace*{-40pt} $\Sg\{a,c\} = \{a,c,f\}$, $\Sg\{a,d\} = \{a,d,e\}$, $\Sg\{b,c\} = \{b,c,f\}$, $\Sg\{b,d\} = \{b,d,e\}$, $\Sg\{e,f\} = \{c,d,e,f\}$\\ \\
\begin{minipage}{\linewidth} \includegraphics[]{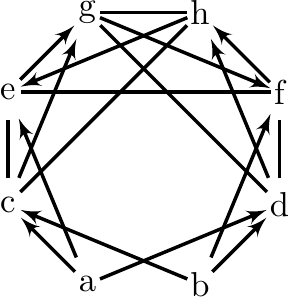} \end{minipage} & \begin{tabular}{c | c c c c c c c c} $f$ & a & b & c & d & e & f & g & h\\ \hline a & a & g & c & d & e & d & d & e \\ b & h & b & c & d & c & f & f & c \\ c & c & c & c & g & c & c & g & c \\ d & d & d & h & d & d & d & d & h \\ e & e & e & e & g & e & e & g & e \\ f & f & f & h & f & f & f & f & h \\ g & g & f & g & g & g & f & g & g \\ h & e & h & h & h & e & h & h & h \end{tabular} & \vspace*{-60pt} $\Sg\{a,f\} = \{a,d,f\}$, $\Sg\{a,g\} = \{a,d,g\}$, $\Sg\{a,h\} = \{a,e,h\}$, $\Sg\{b,e\} = \{b,c,e\}$, $\Sg\{b,g\} = \{b,f,g\}$, $\Sg\{b,h\} = \{b,c,h\}$, $\Sg\{c,f\} = \{c,f,h\}$, $\Sg\{d,e\} = \{d,e,g\}$, $\Sg\{c,d\} = \{c,d,g,h\}$
\end{tabular}

\caption{More examples.}\label{bigger}
\end{figure}

\section{Conjectures}\label{s-conj}

\begin{conj}\label{conj-ws} In every bounded width algebra, there are terms $w,s$ satisfying the identities
\[
w(x,x,y) \approx w(x,y,x) \approx w(y,x,x) \approx s(x,y)
\]
and
\[
s(x,s(x,y)) \approx s(s(x,y),x) \approx s(x,y).
\]
\end{conj}

By Theorem \ref{spiral-ternary}, Conjecture \ref{conj-ws} holds in every spiral, and Conjecture \ref{conj-ws} clearly holds in any majority algebra. Although it looks innocent, if true it would immediately imply Bulatov's yellow connectivity property (as well as Theorem \ref{intersect} and Corollary \ref{reach}) by an iteration argument similar to the end of the proof of Theorem \ref{intersect} (using the terms $s_n$ from Proposition \ref{higher-semilattice} in place of the $f_n$s).

\begin{conj}\label{conj-minority} If $\bA$ is a minimal bounded width algebra and $a, b\in \bA$ have $(b,b) \in \Sg_{\bA^2}\{(a,a),(a,b),(b,a)\}$, then $a\rightarrow b$.
\end{conj}

Conjecture \ref{conj-minority} implies, in particular, that a minimal bounded width algebra $\bA$ has no \emph{minority pairs}, that is, pairs $a \ne b \in \bA$ such that there is a ternary term $p$ with
\[
p(a,a,b) = p(b,a,a) = b
\]
and
\[
p(a,b,b) = p(b,b,a) = a.
\]
Again, Conjecture \ref{conj-minority} holds in every majority algebra and every spiral.

\begin{conj}\label{conj-recursive} If $\bA$ is a minimal bounded width algebra and $a, b \in \bA$, then $(\Sg\{a,b\})\setminus\{a\}$ is a subalgebra of $\bA$.
\end{conj}

Conjecture \ref{conj-recursive} is an analogue of Theorem \ref{recursive} for general bounded width algebras. If true, it shows that any minimal bounded width algebra which is generated by two elements $a,b$ can be built out of the smaller minimal bounded width algebras $(\Sg\{a,b\})\setminus\{a\}$ and $(\Sg\{a,b\})\setminus\{b\}$, together with values for $f(a,b),f(b,a),$ and $g(x,y,z)$ where $\{a,b\} \subset \{x,y,z\}$. Since an algebra has bounded width if and only if every subalgebra which is generated by two elements has bounded width, this would be an enormous help to the classification of minimal bounded width algebras of small cardinality.

\begin{conj}\label{conj-invariant} A minimal bounded width algebra $\bA$ is determined (among minimal bounded width algebras) up to term equivalence by $\Inv_2(\bA)$, the collection of subalgebras of $\bA\times \bA$.
\end{conj}

A general algebra $\bA$ is determined up to term equivalence by $\Inv(\bA)$, the collection of all subalgebras of powers of $\bA$. Since any subpower of a majority algebra is determined by its projections onto pairs of coordinates, Conjecture \ref{conj-invariant} holds for majority algebras. $\Inv_2(\bA)$ contains a large amount of the structural information of $\bA$: every automorphism of $\bA$ and every congruence of $\bA$ is a subalgebra of $\bA\times \bA$, and similarly we can read off the automorphism groups and congruence lattices of all quotients of subalgebras of $\bA$ from $\Inv_2(\bA)$. 

Since every minimal bounded width algebra has a ternary term as its basic operation, in order to describe $\Inv_2(\bA)$ one only has to describe the collection of subalgebras of $\bA\times \bA$ which are generated by at most three elements. Thus, if true Conjecture \ref{conj-invariant} would give an efficient method to test whether two minimal bounded width algebras are term equivalent.

\subsection*{Acknowledgements} The author was inspired to work on this problem by an arXiv paper of Jelena Jovanovi{\'{c}} \cite{jovanovic} which searched for small examples of digraphs with bounded width algebras of polymorphisms in order to make a conjecture about Mal'cev conditions satisfied in all bounded width algebras, as well as the later paper \cite{optimal-maltsev} which settled this conjecture while raising many new ones. The author would also like to thank Petar Markovi{\'{c}} for reaching out and making suggestions on the introduction, and Andrei Bulatov for giving his blessing to cannibalize many of the arguments and ideas from his manuscript \cite{bulatov-bounded}.

\bibliography{csp}

\begin{thebibliography}{10}

\bibitem{baker-pixley}
Kirby~A. Baker and Alden~F. Pixley.
\newblock Polynomial interpolation and the {C}hinese remainder theorem for
  algebraic systems.
\newblock {\em Math. Z.}, 143(2):165--174, 1975.

\bibitem{barto}
Libor Barto.
\newblock The collapse of the bounded width hierarchy.
\newblock {\em Journal of Logic and Computation}, 2014.

\bibitem{cyclic}
Libor Barto and Marcin Kozik.
\newblock Absorbing subalgebras, cyclic terms, and the constraint satisfaction
  problem.
\newblock {\em Log. Methods Comput. Sci.}, 8(1):1:07, 27, 2012.

\bibitem{sdp}
Libor Barto and Marcin Kozik.
\newblock Robust satisfiability of constraint satisfaction problems.
\newblock In {\em Proceedings of the Forty-fourth Annual ACM Symposium on
  Theory of Computing}, STOC '12, pages 931--940, New York, NY, USA, 2012. ACM.

\bibitem{local-consistency}
Libor Barto and Marcin Kozik.
\newblock Constraint satisfaction problems solvable by local consistency
  methods.
\newblock {\em J. ACM}, 61(1):Art. 3, 19, 2014.

\bibitem{2-semilattice}
Andrei~A. Bulatov.
\newblock Combinatorial problems raised from 2-semilattices.
\newblock {\em J. Algebra}, 298(2):321--339, 2006.

\bibitem{bulatov-bounded}
Andrei~A. Bulatov.
\newblock Bounded relational width.
\newblock {\em manuscript.
  \url{http://www.cs.sfu.ca/~abulatov/papers/relwidth.pdf}}, 2009.

\bibitem{colored-graph}
Andrei~A. Bulatov.
\newblock Graphs of finite algebras, edges, and connectivity.
\newblock {\em CoRR}, abs/1601.07403, 2016.

\bibitem{dalmau-approximation}
V{\'\i}ctor Dalmau, Andrei Krokhin, and Rajsekar Manokaran.
\newblock Towards a characterization of constant-factor approximable
  finite-valued {CSP}s.
\newblock {\em Journal of Computer and System Sciences}, 97:14 -- 27, 2018.

\bibitem{dalmau-width-1}
V{\'\i}ctor Dalmau and Justin Pearson.
\newblock Closure functions and width 1 problems.
\newblock In {\em International Conference on Principles and Practice of
  Constraint Programming}, pages 159--173. Springer, 1999.

\bibitem{hobby-mckenzie}
David Hobby and Ralph McKenzie.
\newblock {\em The structure of finite algebras}, volume~76 of {\em
  Contemporary Mathematics}.
\newblock American Mathematical Society, Providence, RI, 1988.

\bibitem{jovanovic}
Jelena Jovanovi\'{c}.
\newblock Some results concerning polymorphisms of small digraphs.
\newblock {\em arXiv preprint arXiv:1405.7946}, 2014.

\bibitem{optimal-maltsev}
Jelena Jovanovi{\'{c}}, Petar Markovi{\'{c}}, Ralph McKenzie, and Matthew
  Moore.
\newblock Optimal strong mal'cev conditions for congruence
  meet-semidistributivity in locally finite varieties.
\newblock {\em Algebra universalis}, pages 1--21, 2016.

\bibitem{kearnes-kiss}
Keith Kearnes and Emil~W Kiss.
\newblock {\em The shape of congruence lattices}, volume 222.
\newblock American Mathematical Soc., 2013.

\bibitem{optimal-taylor}
Keith Kearnes, Petar Markovi{\'c}, and Ralph McKenzie.
\newblock Optimal strong {M}al'cev conditions for omitting type 1 in locally
  finite varieties.
\newblock {\em Algebra Universalis}, 72(1):91--100, 2014.

\bibitem{kearnes-simple}
Keith~A. Kearnes.
\newblock Idempotent simple algebras.
\newblock In {\em Logic and algebra ({P}ontignano, 1994)}, volume 180 of {\em
  Lecture Notes in Pure and Appl. Math.}, pages 529--572. Dekker, New York,
  1996.

\bibitem{slac}
Marcin Kozik.
\newblock Weaker consistency notions for all the {CSP}s of bounded width.
\newblock {\em CoRR}, abs/1605.00565, 2016.

\bibitem{maltsev}
Marcin Kozik, Andrei Krokhin, Matt Valeriote, and Ross Willard.
\newblock Characterizations of several {M}altsev conditions.
\newblock {\em Algebra Universalis}, 73(3-4):205--224, 2015.

\bibitem{lp-width-1}
Gabor Kun, Ryan O'Donnell, Suguru Tamaki, Yuichi Yoshida, and Yuan Zhou.
\newblock Linear programming, width-1 {CSP}s, and robust satisfaction.
\newblock In {\em Proceedings of the 3rd Innovations in Theoretical Computer
  Science Conference}, pages 484--495. ACM, 2012.

\bibitem{ability-to-count}
Benoit Larose, Matt Valeriote, and L{\'a}szl{\'o} Z{\'a}dori.
\newblock Omitting types, bounded width and the ability to count.
\newblock {\em Internat. J. Algebra Comput.}, 19(5):647--668, 2009.

\bibitem{weak-near-unanimity}
Mikl{\'o}s Mar{\'o}ti and Ralph McKenzie.
\newblock Existence theorems for weakly symmetric operations.
\newblock {\em Algebra Universalis}, 59(3-4):463--489, 2008.

\end{thebibliography}
\bibliographystyle{plain}

\appendix

\section{Strengthening of the Yellow Connectivity Property}\label{a-yellow}

The argument in this section follows Bulatov's argument from \cite{bulatov-bounded}. Although the logical structure of the argument has been (violently) rearranged, the main ideas used to prove the intermediate lemmas can be found in Bulatov's work.

\begin{defn} Let $\bA_f = (A,f)$ be an idempotent algebra which has been prepared as in Lemma \ref{prepare}. We say that $\bA_f$ is \emph{yellow connected} if for any pair of upwards closed subsets $A,B$ of $\bA_f$, there are $a \in A$ and $b \in B$ such that $f$ acts as first projection on $\{a,b\}$, that is, $f(a,b) = a$ and $f(b,a) = b$. We say that a pair of maximal elements $a,b$ of $\bA_f$ are yellow connected if there are $a',b'$ in the strongly connected components of $a,b,$ respectively, such that $f$ acts as first projection on $\{a',b'\}$. We say that $\bA_f$ is \emph{hereditarily yellow connected} if every subalgebra of $\bA_f$ is yellow connected.
\end{defn}

We will show that if $\bA = (A,f,g)$ is a minimal bounded width algebra such that $f,g$ are chosen as in Theorem \ref{connect}, then every subalgebra $\bB$ of $\bA_f = (A,f)$ is yellow connected. We will prove this by induction on the size of $\bB$. We easily see that we can reduce to the case that $\bB$ is generated by two elements $a,b$ such that $(a,b)$ is a maximal element of $\Sg_{\bB^2}\{(a,b),(b,a)\}$.

\begin{defn} Say that $\bA, \bB$ are a \emph{good pair} with generators $a,b$ if
\begin{itemize}
\item $\bA = (A,f,g)$ is a minimal bounded width algebra such that $f,g$ are chosen as in Theorem \ref{connect},

\item $\bB = \Sg_{\bA_f}\{a,b\}$ where $\bA_f = (A,f)$,

\item $(a,b)$ is a maximal element of $\Sg_{\bB^2}\{(a,b),(b,a)\}$,

\item every proper subalgebra of $\bB$ is yellow connected, and

\item for any $(a',b')$ in the strongly connected component of $(a,b)$ in $\Sg_{\bB^2}\{(a,b),(b,a)\}$ we have $\Sg_{\bB^2}\{(a',b'),(b',a')\} = \Sg_{\bB^2}\{(a,b),(b,a)\}$.
\end{itemize}
\end{defn}

We will need the following easy result.

\begin{prop}\label{upwards-maj} Suppose that $R \subseteq A \times B$ is subdirect and closed under a partial semilattice operation $s$, and that $s$ acts as first projection on $A$. Then for any $a \in A$, the set $\pi_2(R \cap (\{a\}\times B))$ is upwards closed in $B$.
\end{prop}
\begin{proof} Suppose that $(a,b) \in R$ and that $b' \in B$ with $b \rightarrow b'$. Since $R$ is subdirect, there is $a' \in A$ with $(a',b') \in R$. Then
\[
\begin{bmatrix} a\\ b'\end{bmatrix} = s\left(\begin{bmatrix} a\\ b\end{bmatrix}, \begin{bmatrix} a'\\ b'\end{bmatrix}\right) \in R.\qedhere
\]
\end{proof}


Recall Lemma \ref{maj-triple}, and note that its proof didn't use the yellow connectivity property.

\begin{replem}{maj-triple} Suppose that $\bB$ is generated by $a,b \in \bB$, and let $\RR$ be a subalgebra of $\bB^3$ which contains $\Sg_{\bB^3}\{(a,a,b),(a,b,a),(b,a,a)\}$. Let $U,V,W$ be any three maximal strongly connected components of $\bB$. If $\RR \cap (U\times V\times W) \ne \emptyset$, then $U \times V \times W \subseteq \RR$.
\end{replem}

\begin{prop}\label{ap1} Suppose that $\bB$ is a subproduct of hereditarily yellow connected algebras. Then $\bB$ is hereditarily yellow connected.
\end{prop}
\begin{proof} This follows from Proposition \ref{upwards-maj}.
\end{proof}

\begin{lem}\label{al1} Let $\bA, \bB$ be a good pair generated by $a,b$. Suppose that $\bB$ has a non-full congruence $\theta$ and that there is a ternary term $g'$ of $\bA$ such that $g'$ acts as majority on $\{a/\theta, b/\theta\}$ and such that $\Sg_{\bB^2}\{(a,b),(b,a)\}$ is closed under the binary terms $g'(x,x,y), g'(x,y,x), g'(y,x,x)$. Then $\{a,b\}$ is a majority subalgebra of $\bA$.
\end{lem}
\begin{proof} By reordering the inputs to $g'$, we may assume without loss of generality that $g'$ acts on two element majority subalgebras of $\bA$ as either majority or first projection. Let $\bB_2 = \Sg_{\bB^2}\{(a,b),(b,a)\}$, let $a_2 = (a,b)$, and let $b_2 = (b,a)$, and extend $\theta$ to $\bB_2$ in the natural way. Let $\RR$ be
\[
\Sg_{\bA^6}\left\{\begin{bmatrix} a & b\\a & b\\b & a\end{bmatrix}, \begin{bmatrix} a & b\\b & a\\a & b\end{bmatrix}, \begin{bmatrix} b & a\\a & b\\a & b\end{bmatrix}\right\} \cap \bB_2^3,
\]
so that
\[
\Sg_{\bB_2^3}\left\{\begin{bmatrix} a_2\\ a_2\\ b_2\end{bmatrix}, \begin{bmatrix} a_2\\ b_2\\ a_2\end{bmatrix}, \begin{bmatrix} b_2\\ a_2\\ a_2\end{bmatrix}\right\} \subseteq \RR
\]
and $\RR$ is closed under $f$ and the binary terms $g'(x,x,y), g'(x,y,x), g'(y,x,x)$. Then
\[
g'\left(\begin{bmatrix} a_2\\ a_2\\ b_2\end{bmatrix}, \begin{bmatrix} a_2\\ b_2\\ a_2\end{bmatrix}, \begin{bmatrix} b_2\\ a_2\\ a_2\end{bmatrix}\right) \in \RR \cap (a_2/\theta)^3.
\]
Then by Lemma \ref{maj-triple}, there are maximal strongly connected components $F_1, F_2, F_3$ of $\bB_2$ which are reachable from the elements $g'(a_2,a_2,b_2), g'(a_2,b_2,a_2), g'(b_2,a_2,a_2)$ of $a_2/\theta$ such that $F_1\times F_2\times F_3 \subseteq \RR$. Note that since $a_2$ is maximal in $\bB_2$, we must have $F_i \cap (a_2/\theta) \ne \emptyset$ for each $i$. Since $a_2/\theta \subseteq (a/\theta)\times(b/\theta)$ is a yellow connected algebra, there is some $(f_1,f_2,f_3) \in F_1\times F_2\times F_3 \cap (a_2/\theta)^3 \subseteq \RR \cap (a_2/\theta)^3$ and some $a'$ in the strongly connected component of $\bB_2$ containing $a_2$ such that $\{a',f_1\}$ is a majority subalgebra of $\bA^2$.

Since $\RR$ contains $\Sg_{\bB_2^3}\{(a_2,a_2,b_2), (a_2,b_2,a_2)\}$, there are some maximal elements $u,v \in b_2/\theta$ such that $(a_2,f_2,u) \in \RR$ and $(a_2,v,f_3) \in \RR$. By Lemma \ref{maj-triple} we have $(a',f_2,u), (a',v,f_3) \in \RR$ as well. Thus,
\[
\begin{bmatrix} a'\\ f_2'\\ f_3'\end{bmatrix} = g'\left(\begin{bmatrix} a'\\ f_2\\ u\end{bmatrix}, \begin{bmatrix} a'\\ v\\ f_3\end{bmatrix}, \begin{bmatrix} f_1\\ f_2\\ f_3\end{bmatrix}\right) \in \RR \cap (a_2/\theta)^3.
\]
Letting $A$ be the strongly connected component of $\bB_2$ which contains $a_2$, we see that there are strongly connected components $C,D$ of $\bB_2$ such that $C$ is reachable from $f_2'$ and $D$ is reachable from $f_3'$, with $\RR \cap A\times C\times D \ne \emptyset$. Note that $A,C,D$ must have nontrivial intersection with $a_2/\theta$ (since $a_2$ was assumed maximal). By Lemma \ref{maj-triple}, we have $A\times C\times D \subseteq \RR$.

Since $a_2/\theta$ is yellow connected, there are $a'' \in A \cap (a_2/\theta)$ and $c \in C \cap (a_2/\theta)$ such that $\{a'',c\}$ is a majority subalgebra of $\bA^2$. Letting $d$ be any element of $D\cap (a_2/\theta)$, we have $(a'',c,d), (c,a'',d) \in \RR \cap (a_2/\theta)^3$. From $(a_2,a_2,b_2) \in \RR$ and Lemma \ref{maj-triple}, there is some $w \in b_2/\theta$ with $(a'',a'',w) \in \RR$. So we have
\[
\begin{bmatrix} a''\\ a''\\ d'\end{bmatrix} = g'\left(\begin{bmatrix} a''\\ a''\\ w\end{bmatrix}, \begin{bmatrix} a''\\ c\\ d\end{bmatrix}, \begin{bmatrix} c\\ a''\\ d\end{bmatrix}\right) \in \RR \cap (a_2/\theta)^3.
\]
Applying Lemma \ref{maj-triple} again, we have $A\times A\times E \subseteq \RR$ for $E$ some maximal strongly connected component of $\bB_2$ which is reachable from $d'$.

Since $a_2/\theta$ is hereditarily yellow connected, there are $a''' \in A$ and $e \in E$ with $\{a''',e\}$ a majority subalgebra of $\bA^2$. Since $A\times A\times E \subseteq \RR$, we have
\[
\begin{bmatrix} a'''\\ a'''\\ a'''\end{bmatrix} = g\left(\begin{bmatrix} a'''\\ a'''\\ e\end{bmatrix}, \begin{bmatrix} a'''\\ e\\ a'''\end{bmatrix}, \begin{bmatrix} e\\ a'''\\ a'''\end{bmatrix}\right) \in \RR \cap A^3.
\]
Applying Lemma \ref{maj-triple} one last time, we see that $(a_2,a_2,a_2) \in A\times A\times A \subseteq \RR$, so by Theorem \ref{subalg} $\{a,b\}$ is a majority subalgebra of $\bA$.
\end{proof}

\begin{cor}\label{ac1} Suppose that $\bA,\bB$ is a good pair, and that $\bB$ has a nontrivial quotient $\bB'$ such that $\bB'$ is yellow connected. Then $\bB$ is yellow connected.
\end{cor}
\begin{proof} Let $A,B$ be maximal strongly connected components of $\bB$, and let $\theta$ be the congruence such that $\bB' = \bB/\theta$. Then there are $a \in A, b \in B$ with $\{a/\theta, b/\theta\}$ a set subalgebra of $\bB'$. If $a/\theta = b/\theta$, then we are done since $a/\theta$ is a proper subalgebra of $\bB$. Otherwise $\Sg_{\bB}\{a,b\}$ has a surjective homomorphism to a two element set, so we can apply Lemma \ref{al1} with $g' = g$ to see that $a,b$ are yellow connected.
\end{proof}

\begin{lem}\label{al2} Suppose $\bA, \bB$ are a good pair generated by $a,b$. Let $(c,d) \in \Sg_{\bB^2}\{(a,b),(b,a)\}$ be such that the restriction of $f$ to $\{(a,b),(c,d)\}$ is first projection.

Suppose there is some $(a',b')$ which is reachable from $(a,b)$ and some maximal $(c',d')$ which is reachable from $(c,d)$ in $\Sg_{\bB^2}\{(a,b),(b,a)\}$ such that the restriction of $f$ to $\{(a',b'),(d',c')\}$ is first projection. Then $\{a,b\}$ is a majority subalgebra of $\bA$.
\end{lem}
\begin{proof} As before, let $\bB_2 = \Sg_{\bB^2}\{(a,b),(b,a)\}$, let $a_2 = (a,b)$, let $b_2 = (b,a)$, and define $\RR$ as in the last lemma. Also, set $c_2 = (c,d), d_2 = (d,c)$ and $a_2' = (a',b'), d_2' = (d',c')$. Letting $p$ be a binary term of $\bB$ with $p(a_2,b_2) = c_2$, we see that
\[
\begin{bmatrix} a_2\\ c_2\\ d_2\end{bmatrix} = p\left(\begin{bmatrix} a_2\\ a_2\\ b_2\end{bmatrix}, \begin{bmatrix} a_2\\ b_2\\ a_2\end{bmatrix}\right) \in \RR.
\]
Thus, we have
\[
\begin{bmatrix} a_2\\ a_2\\ d_2\end{bmatrix} = g\left(\begin{bmatrix} a_2\\ a_2\\ b_2\end{bmatrix}, \begin{bmatrix} a_2\\ c_2\\ d_2\end{bmatrix}, \begin{bmatrix} c_2\\ a_2\\ d_2\end{bmatrix}\right) \in \RR.
\]
Applying Lemma \ref{maj-triple}, we see that $(a_2',a_2',d_2') \in \RR$. Applying $g$, we see that $(a_2',a_2',a_2') \in \RR$, and then applying Lemma \ref{maj-triple} again, we have $(a_2,a_2,a_2) \in \RR$, so by Theorem \ref{subalg} $\{a,b\}$ is a majority subalgebra of $\bA$.
\end{proof}

\begin{lem}\label{al3} Suppose that $\bA,\bB$ is a good pair generated by $a,b$. Suppose that there is some $e \in \bB$ such that $\{e\}\times \bB \subseteq \Sg_{\bB^2}\{(a,b),(b,a)\}$. Then $a,b$ are yellow connected.
\end{lem}
\begin{proof} By Theorem \ref{strong-binary}(b) we may assume without loss of generality that $e$ is maximal in $\bB$.

If $b \in \Sg_{\bB}\{a,e\}$, then $(b,b) \in \Sg_{\bB^2}\{(a,b),(b,a)\}$, so $a \rightarrow b$, which contradicts the assumption that $(a,b)$ is maximal. Thus we must have $\Sg_{\bB}\{a,e\} \ne \bB$, and similarly $\Sg_{\bB}\{b,e\} \ne \bB$, so $\Sg_{\bB}\{a,e\}\times \Sg_{\bB}\{b,e\}$ is hereditarily yellow connected.

Thus there is $(e', e'')$ in the strongly connected component of $(e,e)$ and $(a',b')$ in the strongly connected component of $(a,b)$ such that $f$ acts on $\{(a',b'),(e',e'')\}$ as first projection. Now we can apply Lemma \ref{al2} with $(c,d) = (e',e'')$ and $(c',d') = (e'',e')$ to finish.
\end{proof}


\begin{lem}\label{al4} Let $\bA,\bB$ be a good pair generated by $a,b$. Suppose that there is some $e \in \bB$ and some $b'$ in the strongly connected component of $b$ such that $\{b',e\}$ is a majority subalgebra of $\bA$, and such that for all $d \in \bB$ at least one of $(b',d), (e,d)$ is an element of $\Sg_{\bB^2}\{(a,b),(b,a)\}$. Then $a,b$ are yellow connected.
\end{lem}
\begin{proof} Let $\bE = \Sg_{\bB^2}\{(a,b),(b,a)\} \cap (\{b',e\}\times \bB)$. By Lemma \ref{al1}, we may assume that $\bB$ has no homomorphism to $\{b',e\}$, so $\bE$ is linked when considered as a subdirect product of $\{b',e\}$ and $\bB$.

By Lemma \ref{al3}, we may assume that $\{e\}\times \bB \not\subseteq \bE$. Let $A$ be the strongly connected component of $a$ and let $B$ be the strongly connected component of $b$. By Proposition \ref{upwards-maj}, if $(b',b') \in \bE$, then $\{b'\}\times B \subseteq \bE$, and by Theorem \ref{strong-binary}(b) this implies $(b,b) \in \Sg_{\bB^2}\{(a,b),(b,a)\}$, so $a \rightarrow b$, a contradiction.

Thus we must have $(e,b') \in \bE$. By Proposition \ref{upwards-maj}, we have $\{e\} \times B \subseteq \bE$. By Theorem \ref{strong-binary}(c) we have $A\times B \subseteq \Sg_{\bB^2}\{(a,b),(b,a)\}$.

Since $\bE$ is linked, there is some $d \in \bB$ such that both $(e,d)$ and $(b',d)$ are in $\bE$. By Proposition \ref{upwards-maj}, we may assume without loss of generality that $d$ is a maximal element of $\bB$. Let $D$ be the strongly connected component containing $d$, and let $E$ be a maximal strongly connected component which is reachable from $e$. By Proposition \ref{upwards-maj}, $\{e\}\times D \subseteq \bE$, so $E\times D \subseteq \Sg_{\bB^2}\{(a,b),(b,a)\}$ by Theorem \ref{strong-binary}.

Since $\{b'\}\times (A \cup D) \subseteq \Sg_{\bB^2}\{(a,b),(b,a)\}$, $A\cup D$ generate a proper subalgebra of $\bB$, so there are $a'\in A$ and $d' \in D$ with $\{a',d'\}$ a majority subalgebra of $\bA$. Similarly, since $\{e\} \times (B\cup D) \subseteq \Sg_{\bB^2}\{(a,b),(b,a)\}$, there are $b'' \in B, d'' \in D$ with $\{b'',d''\}$ a majority subalgebra of $\bA$. Also, by Proposition \ref{upwards-maj}, $\{b'\} \times E \subseteq \bE$, so $\{b'\} \times (A\cup E) \subseteq \Sg_{\bB^2}\{(a,b),(b,a)\}$, so there are $a'' \in A, e' \in E$ with $\{a'', e'\}$ a majority subalgebra of $\bA$.

Thus $f$ acts as first projection on $\{(a',b'),(d',e)\}$ and $\{(b'',a''),(d'',e')\}$, and $(d'',e')$ is a maximal element of $\Sg_{\bB^2}\{(a,b),(b,a)\}$ which is reachable from $(d',e)$. To finish, we apply Lemma \ref{al2}.
\end{proof}

\begin{lem}\label{al5} Let $\bA, \bB$ be a good pair generated by $a,b$. If $\Sg_{\bB^2}\{(a,b),(b,a)\}$ is linked, then $a,b$ are yellow connected.
\end{lem}
\begin{proof} Since $\Sg_{\bB^2}\{(a,b),(b,a)\}$ is linked, there is a sequence $a = d_1, e_1, d_2, e_2, ..., d_k, e_k, d_{k+1} = b'$ with $(d_i,e_i), (d_{i+1},e_i) \in \Sg_{\bB^2}\{(a,b),(b,a)\}$, with each $d_i, e_i$ a maximal element of $\bB$ and with $b'$ in the strongly connected component of $b$. Let $D_i$ be the strongly connected component containing $d_i$, let $E_i$ be the strongly connected component containing $e_i$, let $A = D_1$ be the strongly connected component containing $a$, and $B = D_{k+1}$ be the strongly connected component containing $b$.

We can assume without loss of generality that $\Sg_{\bB}\{a',b'\} = \bB$ for any $a' \in A, b'\in B$, and we clearly have have $\Sg_{\bB}\{a,a\} = \{a\}$. We label the sets $D_i, E_i$ with the phrases ``full'' or ``non-full'' as follows. For each $D_i$, if every choice of $a' \in A, d_i' \in D_i$ have $\Sg_{\bB}\{a',d_i'\} = \bB$ we call $D_i$ full, and otherwise we call it non-full. For each $E_i$, if every choice of $b' \in B, e_i' \in E_i$ have $\Sg_{\bB}\{b',e_i'\} = \bB$ we call $E_i$ full, and otherwise we call it non-full.


Since $D_1$ is non-full and $D_{k+1}$ is full, there is some adjacent pair of sets in the sequence $D_1, E_1, ..., E_k, D_{k+1}$ such that one is full and  the other is non-full. We assume there is some $i$ such that $E_i$ is non-full and $D_{i+1}$ is full (the other cases are exactly analogous). Suppose $b' \in B, e_i' \in E_i$ generate a proper subalgebra of $\bB$.

Since $b', e_i'$ generate a proper subalgebra of $\bB$, we may assume without loss of generality that $\{b',e_i'\}$ is a majority subalgebra of $\bA$. Then there is some $a' \in A$ such that $(b',a') \in \Sg_{\bB^2}\{(a,b),(b,a)\}$ and there is a $d_{i+1}' \in D_{i+1}$ such that $(e_i', d_{i+1}') \in \Sg_{\bB^2}\{(a,b),(b,a)\}$. Then $\bE = \Sg_{\bB^2}\{(a,b),(b,a)\} \cap (\{b',e_i'\}\times \bB)$ has $\Sg_{\bB}\{a', d_{i+1}'\} \subseteq \pi_2 \bE$, so $\pi_2 \bE = \bB$ and we can finish by using the Lemma \ref{al4}.
\end{proof}

\begin{lem}\label{al6} Let $\bA, \bB$ be a good pair generated by $a,b$. If there is some $c \in \bB$ such that $\{a,c\}$ is a majority algebra of $\bA$ and such that some maximal element $c'$ reachable from $c$ and $b$ are yellow connected, then $a$ and $b$ are yellow connected.
\end{lem}
\begin{proof} Let $d \in \bB$ be such that $(c,d) \in \Sg_{\bB^2}\{(a,b),(b,a)\}$, then by Lemma \ref{al4} we may assume that $\Sg_{\bB}\{b,d\} \ne \bB$. So we may as well assume that $\{b,d\}$ is a majority algebra of $\bA$. Let $\{b',c'\}$ be a majority subalgebra of $\bA$, with $b'$ in the strongly connected component of $b$, and let $(a',b')$ be in the strongly connected component of $(a,b)$. There is some $d'$ with $(c',d') \in \Sg_{\bB^2}\{(a,b),(b,a)\}$, with $(c',d')$ a maximal element reachable from $(c,d)$. By Lemma \ref{al4}, we may assume that $\Sg_{\bB}\{a',d'\} \ne \bB$, so it is hereditarily yellow connected. Thus $\{(a,b),(c,d)\}$ is a majority subalgebra of $\bA^2$ and $(c',d')$ is yellow connected to $(b',a')$, and we can apply Lemma \ref{al2}.
\end{proof}

\begin{cor} If $\bA,\bB$ is a good pair, then yellow connectivity is an equivalence relation on the set of maximal strongly connected components of $\bB$.
\end{cor}

\begin{cor}\label{ac-link} If $\bA, \bB$ is a good pair and $\theta, \theta'$ are two non-full congruences of $\bB$ such that $\theta \vee \theta'$ is the full congruence, then $\bB$ is yellow connected.
\end{cor}
\begin{proof} If $a,b$ are maximal elements of $\bB$, then there is some sequence $a = a_0, a_1, ..., a_k = b$ such that $a_{2i}/\theta = a_{2i+1}/\theta$, $a_{2i+1}/\theta' = a_{2i+2}/\theta'$. Then there is a sequence of maximal elements $a_i'$ with the same properties, such that $a_0'$ is in the strongly connected component of $a$ and $a_k'$ is in the strongly connected component of $b$. Since each $\Sg_{\bB}\{a_i',a_{i+1}'\} \ne \bB$, each adjacent pair is yellow connected, so $a,b$ are yellow connected.
\end{proof}

\begin{lem}\label{al7} Let $\bA, \bB$ be a good pair generated by $a,b$. If there is some element $c \in \bB$ and some congruence $\theta$ of $\bB$ such that $c/\theta \rightarrow a/\theta$ and $b$ is reachable from $c$, then $\{a,b\}$ is a majority subalgebra of $\bA$.
\end{lem}
\begin{proof} By Lemma \ref{al5} we may assume $\Sg_{\bB^2}\{(a,b),(b,a)\}$ is not linked. Assume first that it is the graph of an automorphism which interchanges $a$ and $b$ and that $\theta$ is trivial. Thus there is a $d \in \bB$ with $(c,d) \in \Sg_{\bB^2}\{(a,b),(b,a)\}$ such that $(c,d) \rightarrow (a,b)$. Let $\RR = \Sg_{\bB^3}\{(a,a,b),(a,b,a),(b,a,a)\}$. Since $c \in \Sg_{\bB}\{a,b\}$, we have $(a,c,d), (c,a,d) \in \RR$. Applying $f$ to $(a,c,d), (c,a,d)$, we see that $(a,a,d) \in \RR$. Since $b$ is reachable from $c$, $a$ is reachable from $d$, so if $A$ is the strongly connected component of $\bB$ containing $a$, then $\RR \cap A^3 \ne \emptyset$. Thus $(a,a,a) \in \RR$ by Lemma \ref{maj-triple}, and so by Theorem \ref{subalg} we are done.

Now consider the case where the linking congruence of $\Sg_{\bB^2}\{(a,b),(b,a)\}$ is nontrivial. By Corollary \ref{ac-link}, we may assume it is contained in $\theta$. By the above argument, there is a term $g' \in \Clo(f)$ such that the restriction of $g'$ to $\{a/\theta, b/\theta\}$ acts as majority. Now apply Lemma \ref{al1}.
\end{proof}

\begin{lem}\label{al8} Let $\bA, \bB$ be a good pair generated by $a,b$. If there is some element $c \in \bB$ and some maximal congruence $\theta$ such that $c/\theta \rightarrow a/\theta$ and $c$ is not maximal in $\bB/\theta$, then $a,b$ are yellow connected.
\end{lem}
\begin{proof} If $\Sg_{\bB}\{b,c\} \ne \bB$ then we are done by Lemmas \ref{al6} and \ref{al7}. Otherwise, since there is clearly no automorphism of $\bB/\theta$ which interchanges $b$ and $c$, $\Sg_{\bB^2}\{(b,c),(c,b)\}$ must be linked or we are done by Corollary \ref{ac-link}. So we can assume there is a sequence $a = a_1, a_2, ..., a_k$ with $a_k$ in the strongly connected component of $b$ and with each $a_i$ maximal, such that for each $i$ there is a maximal $e_i$ with $(a_i,e_i), (a_{i+1},e_i) \in \Sg_{\bB^2}\{(b,c),(c,b)\}$. If $\Sg_{\bB}\{a_i,a_{i+1}\} \ne \bB$ for all $i$, then we easily finish using Lemma \ref{al6}. Thus we may assume that there is some maximal $e$ such that $\{e\} \times \bB \subseteq \Sg_{\bB^2}\{(b,c),(c,b)\}$.

If $\Sg_{\bB}\{b,e\} = \bB$, then since $(b,c), (e,c)$ are in $\Sg_{\bB^2}\{(b,c),(c,b)\}$, we have $(c,c) \in \Sg_{\bB^2}\{(b,c),(c,b)\}$, so $b \rightarrow c$, a contradiction. Similarly, if $\Sg_{\bB}\{c,e\} = \bB$ then $c \rightarrow b$ and we finish using Lemma \ref{al7}. Thus we may assume that $\Sg_{\bB}\{b,e\}, \Sg_{\bB}\{c,e\}$ are proper subalgebras of $\bB$. In particular, $b,e$ are yellow connected.

From $\Sg_{\bB}\{c,e\} \ne \bB$, there is some $c'$ which is reachable from $c$ and some $e'$ in the strongly connected component of $e$ such that $\{c',e'\}$ is a majority subalgebra of $\bA$. Letting $c''$ be a maximal element of $\bB$ which is reachable from $c'$, we can use Lemmas \ref{al6} and \ref{al7} to see that $a,c''$ are yellow connected and that $c'',e'$ are yellow connected, respectively, so $a,b$ are yellow connected by two applications of Lemma \ref{al6}.
\end{proof}

\begin{thm} Let $\bA = (A,f,g)$ be a minimal bounded width algebra with $f,g$ chosen as in Theorem \ref{connect}, and let $\bA_f = (A,f)$. Then $\bA_f$ is hereditarily yellow connected.
\end{thm}
\begin{proof} Suppose for a contradiction that there is a subalgebra $\bB \le \bA_f$ which is not yellow connected, and take such a $\bB$ of minimal size. Then for any pair of maximal elements $a,b \in \bB$ which are not yellow connected we must have $\bB = \Sg_{\bA_f}\{a,b\}$. By Theorem \ref{connect}, there is a sequence $a = p_1, p_2, ..., p_n = b$ of elements of $\bB$ such that $\{p_i,p_{i+1}\}$ is a two element subalgebra of $\bB$ for each $i$.

Let $\theta$ be a maximal congruence of $\bB$. If $\bB/\theta$ has no non-maximal elements, then each $\{p_i/\theta, p_{i+1}/\theta\}$ gives a one or two element subalgebra connecting two elements of maximal strongly connected components of $\bB/\theta$, so the corresponding maximal strongly connected components of $\bB$ are yellow connected by Lemma \ref{al1}. Then by Lemma \ref{al6} we have $a,b$ yellow connected.

On the other hand, if $\bB/\theta$ has a non-maximal element, then it must have some elements $c/\theta \rightarrow x/\theta$ with $c/\theta$ non-maximal and $x$ maximal. Then Lemma \ref{al7} shows that $a,x$ are yellow connected and $b,x$ are yellow connected, so we are done by Lemma \ref{al6}.
\end{proof}

\section{Construction of special weak near-unanimity terms}\label{a-special}

\begin{thm}\label{special} Let $\bA$ be a finite algebra of bounded width. Then $\bA$ has an idempotent term $t(x,y)$ satisfying the identity
\[
t(x,t(x,y)) \approx t(x,y)
\]
along with an infinite sequence of idempotent weak near-unanimity terms $w_n$ of every arity $n > 2\lcm\{1,2,...,|\bA|-1\}$ such that for every sequence $(a_1, ..., a_n)$ with $\{a_1, ..., a_n\} = \{x,y\}$ and having strictly less than $\frac{n}{2\lcm\{1,2,...,|\bA|-1\}}$ of the $a_i$ equal to $y$ we have
\[
w_n(a_1, ..., a_n) \approx t(x,y).
\]
\end{thm}
\begin{proof} Let $f, h_{\cF}$ be as in the conclusion to Theorem \ref{intersect}. Define a sequence of functions $t^i$ by $t^0(x,y) \approx x$, $t^1(x,y) \approx f(x,y)$,
\[
t^{i+1}(x,y) \approx f(x,t^i(x,y)).
\]
We will inductively construct for each $i,k \ge 1$ and each $n > 2ik$ a sequence of functions $w_{n,k}^i(a_1, ..., a_n)$ such that for every sequence $(a_1, ..., a_n)$ with $\{a_1, ..., a_n\} = \{x,y\}$ and having at most $k$ of the $a_j$ equal to $y$ we have
\[
w_{n,k}^i(a_1, ..., a_n) \approx t^i(x,y).
\]
We start with $w_{n,k}^1 = h_{\cF}$ for some maximal intersecting family containing the family of all subsets of $\{1, ..., n\}$ of size at least $n-k$ (which is an intersecting family if $n > 2k$). Now assume we have already constructed $w_{n,k}^i$, and we will construct $w_{n+2k,k}^{i+1}$.

Fix a bijection of $s:\binom{\{1, ..., n+2k\}}{n} \rightarrow \{1, ..., \binom{n+2k}{n}\}$. Let $\cF$ be a maximal intersecting family of subsets of $\{1,...,\binom{n+2k}{n}\}$ containing the intersecting family of subsets
\[
\{s(\tbinom{A}{n}) \mid A \subseteq \{1, ..., n+2k\}, |A| \ge n+k\}.
\]
To see that the above forms an intersecting family, note that if $A,B \subseteq \{1, ..., n+2k\}$ with $|A|, |B| \ge n+k$, then $|A\cap B| \ge n$, so there is some $n$-element set which is contained in both $A$ and $B$. Then we take
\[
w_{n+2k,k}^{i+1}(a_1, ..., a_{n+2k}) \approx h_{\cF}(u_1, ..., u_{\binom{n+2k}{n}}),
\]
where $u_{s(\{j_1, ..., j_n\})} = w_{n,k}^i(a_{j_1},a_{j_2}, ..., a_{j_n})$ for $1 \le j_1 < j_2 < \cdots < j_n \le n+2k$. If $\{a_1, ..., a_{n+2k}\} = \{x,y\}$ and the collection of $i$ such that $a_i = x$ is $X$, then $u_A \approx x$ when $X \subseteq A$ by idempotence, and $u_A \approx t^i(x,y)$ otherwise by induction on $i$ if at most $k$ of the $a_j$s are equal to $y$.

Now take
\[
w_n(a_1, ..., a_n) = w_{n,k}^i(a_1, ..., a_n)
\]
with $i = \lcm\{1, ..., |\bA|-1\}$ and $k = \lfloor\frac{n-1}{2i}\rfloor$, and note that for this $i$ we necessarily have
\[
t^i(x,t^i(x,y)) \approx t^i(x,y).\qedhere
\]
\end{proof}

\section{Counterexample to (SM 5) and (SM 6)}\label{a-sm}

Let $\bA = \bA_{11} = (\{a,b,c\},f,g)$ be the idempotent algebra given by
\begin{align*}
f(a,b) = f(b,a) &= f(b,c) = g(a,b,c) = b,\\
f(a,c) = f(c,a) &= f(c,b) = g(a,c,b) = c,\\
f(x,y) &\approx g(x,x,y),\\
g(x,y,z) &\approx g(y,z,x).
\end{align*}
Note that this algebra satisfies $f(f(x,y),f(y,x)) \approx f(x,y)$, and thus has bounded width. Also, note that
\begin{align*}
\begin{bmatrix} b\\ c\end{bmatrix} &\not\in \Sg_{\bA^2}\left\{\begin{bmatrix} a\\ a\end{bmatrix}, \begin{bmatrix} a\\ b\end{bmatrix}, \begin{bmatrix} b\\ b\end{bmatrix}, \begin{bmatrix} c\\ c\end{bmatrix}\right\} = \left\{\begin{bmatrix} a\\ a\end{bmatrix}, \begin{bmatrix} a\\ b\end{bmatrix}, \begin{bmatrix} b\\ b\end{bmatrix}, \begin{bmatrix} c\\ b\end{bmatrix}, \begin{bmatrix} c\\ c\end{bmatrix}\right\},\\
\begin{bmatrix} b\\ b\end{bmatrix} &\not\in \Sg_{\bA^2}\left\{\begin{bmatrix} a\\ a\end{bmatrix}, \begin{bmatrix} a\\ c\end{bmatrix}, \begin{bmatrix} b\\ c\end{bmatrix}, \begin{bmatrix} c\\ b\end{bmatrix}\right\} = \left\{\begin{bmatrix} a\\ a\end{bmatrix}, \begin{bmatrix} a\\ c\end{bmatrix}, \begin{bmatrix} b\\ c\end{bmatrix}, \begin{bmatrix} c\\ b\end{bmatrix}, \begin{bmatrix} c\\ c\end{bmatrix}\right\},
\end{align*}
that $\Sg_{\bA^n} \{x_1, ..., x_k\}$ contains the vector $(a,a,...,a)$ if and only if one of the $x_i$ is equal to $(a,a,...,a)$, and that $(b\;c) \in \Aut(\bA)$.

\begin{thm} The algebra $\bA$ does not have any $4$-ary term $t$ satisfying
\begin{align*}
t(x,y,z,y) &\approx t(y,x,z,z) \approx t(z,x,x,y),\tag{SM 5}
\end{align*}
and also does not have any $4$-ary term $t$ satisfying
\begin{align*}
t(x,y,z,y) &\approx t(y,x,x,z) \approx t(z,y,x,x).\tag{SM 6}
\end{align*}
\end{thm}
\begin{proof} If a term satisfying (SM 5) existed, then for every algebra $\bF$ in the variety generated by $\bA$ and for every $x,y,z \in \bF$, we would have
\[
\Sg_{\bF^3}\left\{\begin{bmatrix}x\\y\\z\end{bmatrix}, \begin{bmatrix}y\\x\\x\end{bmatrix}, \begin{bmatrix}z\\z\\x\end{bmatrix}, \begin{bmatrix}y\\z\\y\end{bmatrix}\right\} \cap \Delta \ne \emptyset,
\]
where $\Delta$ is the diagonal of $\bF^3$. To show that this is not the case, take $\bF = \bA^2$, $x = (a,b), y = (b,c), z = (c,a)$. We aim to show that
\[
\Sg_{\bA^6}\left\{\begin{bmatrix}a & b\\b & c\\c & a\end{bmatrix}, \begin{bmatrix}b & c\\a & b\\a & b\end{bmatrix}, \begin{bmatrix}c & a\\c & a\\a & b\end{bmatrix}, \begin{bmatrix}b & c\\c & a\\b & c\end{bmatrix}\right\} \cap \left\{\begin{bmatrix}\alpha & \beta\\\alpha & \beta\\\alpha & \beta\end{bmatrix} \Bigg|\ \alpha, \beta \in \bA\right\} = \emptyset.
\]
Suppose for contradiction that $\begin{bmatrix}\alpha & \beta\\\alpha & \beta\\\alpha & \beta\end{bmatrix}$ was in the intersection. We clearly must have $\alpha, \beta \in \{b,c\}$, and by
\[
\begin{bmatrix}* & *\\b & *\\b & *\end{bmatrix} \not\in \Sg_{\bA^6}\left\{\begin{bmatrix}* & *\\b & *\\c & *\end{bmatrix}, \begin{bmatrix}* & *\\a & *\\a & *\end{bmatrix}, \begin{bmatrix}* & *\\c & *\\a & *\end{bmatrix}, \begin{bmatrix}* & *\\c & *\\b & *\end{bmatrix}\right\}
\]
and
\[
\begin{bmatrix}* & b\\ * & b\\ * & *\end{bmatrix} \not\in \Sg_{\bA^6}\left\{\begin{bmatrix}* & b\\ * & c\\ * & *\end{bmatrix}, \begin{bmatrix}* & c\\ * & b\\ * & *\end{bmatrix}, \begin{bmatrix}* & a\\ * & a\\ * & *\end{bmatrix}, \begin{bmatrix}* & c\\ * & a\\ * & *\end{bmatrix}\right\}
\]
we see that $\alpha = \beta = c$. But this is a contradiction, since
\[
\begin{bmatrix}c & *\\ * & *\\ * & c\end{bmatrix} \not\in \Sg_{\bA^6}\left\{\begin{bmatrix}a & *\\ * & *\\ * & a\end{bmatrix}, \begin{bmatrix}b & *\\ * & *\\ * & b\end{bmatrix}, \begin{bmatrix}c & *\\ * & *\\ * & b\end{bmatrix}, \begin{bmatrix}b & *\\ * & *\\ * & c\end{bmatrix}\right\}.
\]

Similarly, if a term satisfying (SM 6) existed, we would have
\[
\Sg_{\bF^3}\left\{\begin{bmatrix}x\\y\\z\end{bmatrix}, \begin{bmatrix}y\\x\\y\end{bmatrix}, \begin{bmatrix}z\\x\\x\end{bmatrix}, \begin{bmatrix}y\\z\\x\end{bmatrix}\right\} \cap \Delta \ne \emptyset,
\]
so it is enough to show that
\[
\Sg_{\bA^6}\left\{\begin{bmatrix}a & b\\b & c\\c & a\end{bmatrix}, \begin{bmatrix}b & c\\a & b\\b & c\end{bmatrix}, \begin{bmatrix}c & a\\a & b\\a & b\end{bmatrix}, \begin{bmatrix}b & c\\c & a\\a & b\end{bmatrix}\right\} \cap \left\{\begin{bmatrix}\alpha & \beta\\\alpha & \beta\\\alpha & \beta\end{bmatrix} \Bigg|\ \alpha, \beta \in \bA\right\} = \emptyset.
\]
This follows from
\begin{align*}
\begin{bmatrix}* & b\\ * & *\\ b & *\end{bmatrix} &\not\in \Sg_{\bA^6}\left\{\begin{bmatrix}* & b\\ * & *\\ c & *\end{bmatrix}, \begin{bmatrix}* & c\\ * & *\\ b & *\end{bmatrix}, \begin{bmatrix}* & a\\ * & *\\ a & *\end{bmatrix}, \begin{bmatrix}* & c\\ * & *\\ a & *\end{bmatrix}\right\},\\
\begin{bmatrix}* & *\\ * & c\\ b & *\end{bmatrix} &\not\in \Sg_{\bA^6}\left\{\begin{bmatrix}* & *\\ * & c\\ c & *\end{bmatrix}, \begin{bmatrix}* & *\\ * & b\\ b & *\end{bmatrix}, \begin{bmatrix}* & *\\ * & b\\ a & *\end{bmatrix}, \begin{bmatrix}* & *\\ * & a\\ a & *\end{bmatrix}\right\},\\
\begin{bmatrix}c & *\\ * & *\\ * & c\end{bmatrix} &\not\in \Sg_{\bA^6}\left\{\begin{bmatrix}a & *\\ * & *\\ * & a\end{bmatrix}, \begin{bmatrix}b & *\\ * & *\\ * & c\end{bmatrix}, \begin{bmatrix}c & *\\ * & *\\ * & b\end{bmatrix}, \begin{bmatrix}b & *\\ * & *\\ * & b\end{bmatrix}\right\},\\
\begin{bmatrix}* & b\\ c & *\\ * & *\end{bmatrix} &\not\in \Sg_{\bA^6}\left\{\begin{bmatrix}* & b\\ b & *\\ * & *\end{bmatrix}, \begin{bmatrix}* & c\\ a & *\\ * & *\end{bmatrix}, \begin{bmatrix}* & a\\ a & *\\ * & *\end{bmatrix}, \begin{bmatrix}* & c\\ c & *\\ * & *\end{bmatrix}\right\}.\qedhere
\end{align*}
\end{proof}

\begin{rem} The three element algebra considered in this section is closely connected to a counterexample from the theory of CSPs which are solved by the basic Linear Programming relaxation, due to Kun, which can be found in \cite{dalmau-approximation}. Consider the commutative binary operation $s_2$ (for symmetric, not semilattice) on $\{-,0,+\}$ given by the following table.
\begin{center}
\begin{tabular}{c|ccc} $s_2$ & $-$ & $0$ & $+$\\ \hline $-$ & $-$ & $-$ & $0$\\ $0$ & $-$ & $0$ & $+$\\ $+$ & $0$ & $+$ & $+$\end{tabular}
\end{center}
If we define symmetric operations $s_n : \{-,0,+\}^n \rightarrow \{-,0,+\}$ by
\[
s_n(x_1, ..., x_n) = \begin{cases}+ & \text{if }\sum_i x_i > 0,\\ 0 & \text{if }\sum_i x_i = 0,\\ - & \text{if }\sum_i x_i < 0,\end{cases}
\]
where we have identified $-,0,+$ with $-1,0,1$, then an inductive argument shows that $s_n \in \Clo(s_2)$ for every $n$, so the associated CSP is solved by the basic LP relaxation \cite{lp-width-1}. However, $\Clo(s_2)$ contains no totally symmetric operations of arity $3$, so it does not have a set operation and consequently does not have width 1 \cite{dalmau-width-1} (a ternary operation $t$ is \emph{totally symmetric} if it is symmetric and additionally satisfies the identity $t(x,x,y) \approx t(x,y,y)$).

The connection between this counterexample and the algebra $\bA_{11}$ considered in this appendix is that if we rename $-,0,+$ to $b,a,c$, then $f,g \in \Clo(s_2)$, and $\Clo(g)$ is the only minimal bounded width clone contained in $\Clo(s_2)$. In fact, $\Clo(s_2)$ already is a counterexample to the existence of terms satisfying (SM 5) or (SM 6), by exactly the same argument as the above.
\end{rem}

\end{document}